\documentclass{amsart}

\usepackage{amsmath}
\usepackage{amsfonts}
\usepackage{amsthm}
\usepackage{mathrsfs}
\usepackage[utf8]{inputenc}
\usepackage{enumitem}
\usepackage{dsfont}
\usepackage[style=ext-alphabetic,maxnames=5, maxalphanames=5,doi=false,url=false]{biblatex}
\usepackage[margin=1.5in]{geometry}
\usepackage{amssymb}
\usepackage{pifont}
\usepackage[T1]{fontenc}
\usepackage{mathtools}
\usepackage{color}
\usepackage[dvipsnames, hyperref]{xcolor}
\usepackage{hyperref, cleveref}
\usepackage{tikz}
\usepackage{float}

\hypersetup{
    colorlinks=true,
    linkcolor=Blue,
    citecolor=Blue
  }

\usepackage{todonotes}
\usepackage{comment}

\newtheorem{thm}{Theorem}[section]
\newtheorem{theorem}[thm]{Theorem}

\newtheorem{lem}[thm]{Lemma}
\newtheorem{lemma}[thm]{Lemma}

\newtheorem{prop}[thm]{Proposition}
\newtheorem{proposition}[thm]{Proposition}

\newtheorem{cor}[thm]{Corollary}
\newtheorem{corollary}[thm]{Corollary}

\theoremstyle{definition}
\newtheorem*{claim*}{Claim}

\newtheorem{definition}[thm]{Definition}

\newtheorem{example}[thm]{Example}

\newtheorem{remark}[thm]{Remark}

\setlist[description]{leftmargin=0cm}

\newcommand{\eps}{\varepsilon}
\newcommand{\dee}{\ensuremath{\partial}}
\newcommand{\mc}[1]{\ensuremath{\mathcal{#1}}}

\newcommand{\N}{\ensuremath{\mathbb{N}}}
\newcommand{\Z}{\ensuremath{\mathbb{Z}}}

\newcommand{\R}{\ensuremath{\mathbb{R}}}
\newcommand{\C}{\ensuremath{\mathbb{C}}}

\newcommand{\X}{\ensuremath{\mathbb{X}}}
\renewcommand{\H}{\ensuremath{\mathbb{H}}}
\renewcommand{\P}{\ensuremath{\mathbb{P}}}

\DeclareMathOperator{\Hom}{Hom}

\DeclareMathOperator{\PSL}{PSL}

\DeclareMathOperator{\PGL}{PGL}

\DeclareMathOperator{\PO}{PO}
\DeclareMathOperator{\Fix}{Fix}

\DeclareMathOperator{\Cay}{Cay}
\DeclareMathOperator{\Stab}{Stab}

\DeclareMathOperator{\PU}{PU}

\DeclareMathOperator{\SU}{SU}

\newcommand{\identity}{\mathrm{id}}
\newcommand{\minus}{\setminus}
\newcommand{\st}{:}

\newcommand{\Opp}{\mathrm{Opp}}

\newcommand{\flags}{\ensuremath{G/Q}}
\newcommand{\cay}{\ensuremath{\mathrm{Cay}}}

\newcommand{\ghom}{\mathrm{Hom}_{\mathrm{geom}}}
\newcommand{\relghom}{\ghom}

\newcommand{\dgam}{\ensuremath{d_{\Gamma}}}
\newcommand{\dcon}{\ensuremath{d_{\hat{\Gamma}}}}
\newcommand{\dcusp}{\ensuremath{d_X}}
\newcommand{\lgam}[1]{\ensuremath{|#1|_{\Gamma}}}
\newcommand{\lcon}[1]{\ensuremath{|#1|_{\hat{\Gamma}}}}
\newcommand{\lcusp}[1]{\ensuremath{|#1|_{X}}}

\newcommand{\dgams}{\ensuremath{d_{\Gamma}^\sigma}}
\newcommand{\dcusps}{\ensuremath{d_X^\sigma}}
\newcommand{\lgams}[1]{\ensuremath{|#1|^\sigma_\Gamma}}
\newcommand{\lcusps}[1]{\ensuremath{|#1|_X^\sigma}}
\newcommand{\lcons}[1]{\ensuremath{|#1|_{\hat{\Gamma}}^\sigma}}
\newcommand{\bdist}{\ensuremath{d_{\partial}}}

\newcommand{\hcay}{\widehat{\operatorname{Cay}}}
\newcommand{\conedgraph}{\ensuremath{\hcay(\Gamma; \mc{P})}}

\newcommand{\Hull}{\mathrm{Hull}}
\newcommand{\chab}{\mathcal{C}}

\newcommand{\bgamp}{\ensuremath{\partial(\Gamma, \pargps)}}
\newcommand{\bigamp}{\ensuremath{\partial(\Gamma, \ipargps)}}
\newcommand{\pargps}{\ensuremath{\mathcal{P}}}
\newcommand{\ipargps}{\ensuremath{\mathcal{P}_\infty}}

\newcommand{\horb}[1]{\ensuremath{\mathcal{H}(#1)}}

\newcommand{\repquot}{\Gamma^\sigma}
\newcommand{\bquot}{\partial(\repquot, \mathcal{P}^\sigma)}
\newcommand{\biquot}{\partial(\repquot, \mathcal{P}_\infty^\sigma)}
\newcommand{\quotgps}{\mathcal{P}^\sigma}
\newcommand{\iquotgps}{\mathcal{P}_\infty^\sigma}
\newcommand{\quotpair}{(\repquot, \mathcal{P}^\sigma)}

\NewDocumentCommand\nbhd{o m m}{
  \IfNoValueTF{#1}
  {\ensuremath{\mathcal{N}_{#3}(#2)}}
  {\ensuremath{\mathcal{N}_{#3}(#2; #1)}}
}
\NewDocumentCommand\ball{o m m}{
  \IfNoValueTF{#1}
  {\ensuremath{\mathcal{B}_{#3}(#2)}}
  {\ensuremath{\mathcal{B}_{#3}(#2; #1)}}
}

\makeatletter
\newsavebox{\@brx}
\newcommand{\llangle}[1][]{\savebox{\@brx}{\(\m@th{#1\langle}\)}%
  \mathopen{\copy\@brx\kern-0.5\wd\@brx\usebox{\@brx}}}
\newcommand{\rrangle}[1][]{\savebox{\@brx}{\(\m@th{#1\rangle}\)}%
  \mathclose{\copy\@brx\kern-0.5\wd\@brx\usebox{\@brx}}}
\makeatother

\title{Dehn filling in semisimple Lie groups}

\author{Theodore Weisman}
\address{Department of Mathematics, University of Michigan, Ann Arbor,
  MI 48109, USA
}
\email{tjwei@umich.edu}

\begin{document}

\maketitle

\begin{abstract}
  We generalize one part of Thurston's hyperbolic Dehn filling theorem
  to arbitrary-rank semisimple Lie groups by showing that certain
  deformations of extended geometrically finite subgroups of a
  semisimple Lie group are still extended geometrically finite. As a
  special case, our theorem gives a criterion which guarantees that a
  deformation of a relatively Anosov subgroup is (non-relatively)
  Anosov, and also ensures that limit sets vary continuously. Our
  result also applies to several higher-rank examples in convex
  projective geometry which are outside of the relatively Anosov
  setting.
\end{abstract}

\tableofcontents

\section{Introduction}

Thurston's celebrated hyperbolic Dehn surgery theorem is a
foundational result in the theory of hyperbolic manifolds. Given a
single noncompact finite-volume hyperbolic 3-manifold $M_\infty$, the
theorem provides a construction for a countable family $M_n$ of
pairwise distinct isometry classes of hyperbolic $3$-manifolds, each
of which is a topological Dehn filling of $M_\infty$; moreover the
hyperbolic structures on $M_n$ converge to the hyperbolic structure on
$M_\infty$ in a geometric sense.

Thurston's theorem is striking both because it provides an abundance
of closed hyperbolic 3-manifolds and because it closely connects the
topological study of hyperbolic manifolds to the analysis of character
varieties of their fundamental groups. It is natural to ask if similar
phenomena can occur for other types of geometric manifolds, \emph{not}
locally modeled on $\H^3$, and in fact examples of ``geometric Dehn
filling'' have since been observed in several different contexts. See
e.g. \cite{Schroeder,schwartz07,acosta2016,acosta2019,mr2018,clm2020,BDLM}.

In this article, our goal is to provide a general framework for
studying and constructing analogs and generalizations of hyperbolic
Dehn filling in the context of ``geometrically finite'' discrete
subgroups of a semisimple Lie group $G$ with arbitrary rank. This
framework gives a means to study geometric convergence of manifolds
and orbifolds locally modeled on homogeneous $G$-spaces; it describes
several instances of ``geometric Dehn filling'' previously studied by
other authors, as well as some new examples to be explored in
forthcoming work.

Our main result gives general conditions which guarantee that, if
$\Gamma$ is a relatively hyperbolic group, $G$ is a semisimple Lie
group, and $\rho:\Gamma \to G$ is a representation which is
``geometrically finite'' in a precise sense, then a sequence of
representations $\rho_n$ converging to $\rho$ also converges
\emph{strongly} to $\rho$, and gives a countable family of
representations which are also ``geometrically finite.'' To make all
of this concrete, we work in the context of (relativized) \emph{Anosov
  representations}.

\subsection{Anosov representations and extended geometrical
  finiteness}

Originally introduced in a seminal paper of Labourie
\cite{Labourie2006}, and subsequently studied by many other authors
(see e.g. \cite{gw2012anosov,GGKW2017,KLP2017,bps2019anosov}), Anosov
representations provide a natural arbitrary-rank generalization of
convex cocompact subgroups of rank-one Lie groups. There are many
equivalent definitions of an Anosov representation; the most relevant
one for this paper is in terms of a \emph{boundary embedding}, i.e. an
equivariant embedding from the Gromov boundary of a hyperbolic group
into a flag manifold associated to a semisimple Lie group.

Several notions of \emph{relative} Anosov representation have recently
been introduced to study generalizations of geometrical finiteness in
higher rank \cite{kl2018relativizing,Z2021,CZZ,zz1}. In this paper, we
mainly work with the most general definition available: that of an
\emph{extended geometrically finite} (or EGF) representation,
previously introduced by the author in \cite{Weisman2022}. There are
two main advantages to doing so. First, several of the
previously-studied examples of discrete groups undergoing ``Dehn
filling'' phenomena are simply not covered by alternative
definitions. Second, also unlike other definitions, EGF
representations are well-suited to an understanding of the transition
between two relatively hyperbolic subgroups $\Gamma, \Gamma'$ of some
semisimple Lie group in the situation where $\Gamma$ and $\Gamma'$
have qualitatively different peripheral behavior. This is precisely
what is needed to describe both Dehn filling in $\PSL(2, \C)$ and many
of its analogs in other Lie groups.

Although our results apply in the broad context of EGF
representations, our theorems are new even when we consider special
cases corresponding to more specific notions of relative Anosov
representation (see \Cref{thm:rel_anosov_fillings}). In fact, we
obtain useful results even in the case of geometrically finite groups
in rank one. See \Cref{thm:rank_one_dehn_filling} and
\Cref{thm:rank_one_relative_strong} below.

\subsection{The general Dehn filling theorem}

To get an idea of our main theorem, first note that the proof of
Thurston's hyperbolic Dehn filling theorem can essentially be broken
down into two parts:
\begin{enumerate}[label=\arabic*.]
\item\label{item:analyze_character_variety} Prove that if
  $\Gamma < \PSL(2, \C)$ is the holonomy group of a finite-volume
  hyperbolic 3-manifold $M$ with $n$ cusps, then the character variety
  $X(\Gamma, \PSL(2, \C))$ is locally an $n$-dimensional complex
  manifold in a neighborhood of the inclusion
  $\Gamma \hookrightarrow \PSL(2, \C)$. This manifold is identified
  with a neighborhood $U$ of $(\infty, \ldots, \infty)$ in
  $(\mathbb{C}\mathbb{P}^1)^n$, where $\mathbb{C}\mathbb{P}^1$ is
  identified with $\C \cup \{\infty\}$.
\item\label{item:show_geom_finite} Show that (possibly after shrinking
  $U$) every point $(z_1, \ldots, z_n) \in U$ such that each $z_j$
  lies in $\Z[i] \cup \{\infty\}$ corresponds to a representation of
  $\Gamma$ whose image is a lattice in $\PSL(2, \C)$.
\end{enumerate}

The Dehn filling theorem in this paper can be thought of as giving a
general tool for carrying out the analog of step
\ref{item:show_geom_finite} above, when $\Gamma$ is an arbitrary
relatively hyperbolic group included into an arbitrary semisimple Lie
group in a ``geometrically finite'' way (to be precise, when
$\Gamma \hookrightarrow G$ is an EGF representation). We note that,
although our proof does not follow anything resembling Thurston's
original approach, our result can still be used to recover this part
of the original hyperbolic Dehn filling theorem (see
\Cref{ex:classical_example} below).

\begin{remark}
  It is difficult to envision a general theorem giving an analog of
  step \ref{item:analyze_character_variety} above for an arbitrary EGF
  representation $\Gamma \to G$. Analysis of the character variety
  $X(\Gamma, G)$ can already be challenging for specific examples of a
  relatively hyperbolic group $\Gamma$ and semisimple Lie group $G$,
  even in a neighborhood of an explicit EGF representation
  $\Gamma \to G$.
\end{remark}
  
Our main theorem is stated using the language of EGF representations
and \emph{extended Dehn filling spaces}. We give all of the details in
\Cref{sec:egf_background}, but the rough idea of these spaces is the
following. If $\Gamma$ is a relatively hyperbolic group, relative to a
collection $\pargps$ of peripheral subgroups, we can equip the set of
representations $\Gamma \to G$ with the topology of \emph{relative
  geometric convergence}; this space is denoted
$\ghom(\Gamma, G; \pargps)$. A sequence of representations $\rho_n$
converges in $\ghom(\Gamma, G; \pargps)$ to some representation $\rho$
if and only if $\rho_n$ converges to $\rho$ in the compact-open
topology, \emph{and} if, for each $P \in \pargps$, the restrictions
$\rho_n|_P$ converge \emph{strongly} or \emph{geometrically} to the
restriction $\rho|_P$ (see \Cref{sec:chabauty}). In particular, if
$\rho$ is discrete and faithful, then for large $n$, the restrictions
$\rho_n|_P$ are discrete (but not necessarily faithful) for all
$P \in \pargps$.

Each element $\rho \in \Hom_{\mathrm{geom}}(\Gamma, G; \pargps)$ also
determines an action of each $P \in \pargps$ on the various flag
manifolds associated to $G$. An \emph{extended Dehn filling space} is
a certain subset of $\Hom_{\mathrm{geom}}(\Gamma, G; \pargps)$ where
the images of the orbit maps associated to these peripheral actions
deform in a ``uniformly continuous'' way.

\begin{remark}
  The precise definition of an extended Dehn filling space (given in
  \Cref{sec:edf}) is rather technical, but in special cases there can
  be a more concrete description. In fact, in some situations, it
  turns out that the whole space
  $\Hom_{\mathrm{geom}}(\Gamma, G; \pargps)$ already satisfies the
  definition (see e.g. \Cref{thm:rank_one_dehn_filling}).
\end{remark}

For the formal statement of the main theorem, we need a bit more
notation.
\begin{definition}
  Let $(\Gamma, \mc{P})$ be a relatively hyperbolic pair, and let
  $\sigma:\Gamma \to G$ be a homomorphism. We let $\repquot$ denote
  the quotient group $\Gamma / N^\sigma$, where
  $N^\sigma = \llangle \bigcup_{P \in \pargps}
  \ker(\sigma|_P)\rrangle$, and we let $\mc{P}^\sigma$ denote the
  collection of groups $\{P / \ker(\sigma|_P)\}_{P \in \mc{P}}$.
\end{definition}

Our main theorem is then as follows:
\begin{theorem}
  \label{thm:main_theorem}
  Let $(\Gamma, \mc{P})$ be a relatively hyperbolic pair, let $Q$ be a
  symmetric parabolic subgroup of a semisimple Lie group $G$, let
  $\rho:\Gamma \to G$ be a $Q$-EGF representation, and let
  $W \subseteq \ghom(\Gamma, G; \pargps)$ be an extended Dehn filling
  space for $\rho$.

  Then there is a neighborhood $O$ of $\rho$ in
  $\ghom(\Gamma, G; \pargps)$ such that, if $\sigma \in O \cap W$,
  then $(\repquot, \mc{P}^\sigma)$ is a relatively hyperbolic pair,
  and $\sigma$ induces a $Q$-EGF representation $\repquot \to G$.
\end{theorem}

\begin{remark}\
  \label{rem:main_thm}
  \begin{enumerate}[label=(\alph*)]
  \item Our proof of \Cref{thm:main_theorem} also gives us some
    control over the ``limit sets'' of the deformed EGF
    representations $\sigma$, as well as geometric convergence of
    $\sigma$ as it approaches $\rho$ in $\ghom(\Gamma, G;
    \pargps)$. For the precise statements, see
    \Cref{prop:main_prop,prop:relative_strong_to_strong}.
  \item It is not obvious from the definition that the inclusions
    $P \hookrightarrow \Gamma$ induce injective maps
    $P / \ker(\sigma|_P) \to \Gamma^\sigma$. This follows from the
    \emph{relatively hyperbolic Dehn filling theorem} of Osin
    \cite{Osin2007} and Groves--Manning \cite{GM2008} (see
    \Cref{sec:rel_hyp_dehn_filling}).
  \item\label{main_thm_faithful_deformation} \Cref{thm:main_theorem}
    can be viewed as a generalization of the main result in
    \cite{Weisman2022}, which also gives a criterion determining when
    a deformation of an EGF representation is still EGF (see also
    \cite{MMW24} for an analogous result in a different
    context). Crucially, the deformed actions considered in both
    \cite{Weisman2022} and \cite{MMW24} have kernel no larger than
    that of the original action, which means they do not come from
    representations of nontrivial Dehn fillings of the original
    relatively hyperbolic group.
  \end{enumerate}
\end{remark}

\subsubsection{Obtaining Anosov representations by filling}

By applying an Anosov relativization theorem (see \cite[Section
4]{Weisman2022}), we can use \Cref{thm:main_theorem} to describe
situations where a non-Anosov subgroup occurs as a limit of Anosov
subgroups. This generalizes the way that the original hyperbolic Dehn
filling theorem describes situations in which nonuniform lattices in
$\PSL(2, \C)$ occur as limits of uniform lattices. The precise
statement is the following.
\begin{corollary}
  \label{cor:fill_to_anosov}
  In the context of \Cref{thm:main_theorem}, if, for each
  $P \in \mc{P}$, the group $P / \ker(\sigma|_P)$ is hyperbolic and
  $\sigma$ induces a $Q$-Anosov representation of
  $P / \ker(\sigma|_P)$, then $\repquot$ is hyperbolic and $\sigma$
  induces a $Q$-Anosov representation of $\repquot$.
\end{corollary}

\begin{example}[Dehn filling in $\H^3$]
  \label{ex:classical_example}
  Consider the case where $\Gamma$ is the fundamental group of a
  finite-volume hyperbolic 3-manifold $M$, and $\pargps$ is the
  collection of cusp subgroups.  Then $(\Gamma, \pargps)$ is
  relatively hyperbolic, and the holonomy representation
  $\rho:\Gamma \to \PSL(2, \C)$ is geometrically finite (in
  particular, EGF). Below we explain how to apply
  \Cref{thm:main_theorem} towards hyperbolic Dehn fillings of $M$.

  It turns out that in this situation, the entire space
  $\ghom(\Gamma, \PSL(2, \C); \pargps)$ is an extended Dehn filling
  space about $\rho$ (see \Cref{thm:rank_one_dehn_filling} below). For
  simplicity, assume that $M$ has one cusp $C$. Then, to apply
  \Cref{thm:main_theorem}, one must check that, if
  $\rho_{p,q}:\Gamma \to \PSL(2, \C)$ is a sequence of representations
  parameterized by Dehn filling slopes $(p,q)$ tending to infinity,
  then the restrictions $\rho_{p,q}|_{\pi_1C}$ converge
  \emph{strongly} to the restriction $\rho|_{\pi_1C}$. This
  essentially boils down to considering the behavior of the developing
  maps for the corresponding sequence of complex affine structures on
  a torus cross-section of $C$.

  Once this has been verified, one can use \Cref{cor:fill_to_anosov}
  to see that all but finitely many $\rho_{p,q}$ descend to Anosov
  (i.e. convex cocompact) representations
  $\pi_1M_{p,q} \to \PSL(2, \C)$, where $M_{p,q}$ is the $(p,q)$ Dehn
  filling of $M$. From here, to establish that $\rho_{p,q}(\Gamma)$ is
  actually a lattice (without circularly relying on the original Dehn
  filling theorem), one can for example employ \cite{GMS2019} to see
  that the limit set of $\rho_{p,q}(\Gamma)$ is eventually
  homeomorphic to a 2-sphere, and therefore equal to $\dee \H^3$. (An
  alternative approach might use an analog of
  \cite[Lem. 6.10]{MMW24}.)

  We emphasize that none of the steps above necessarily rely on
  anything like a triangulation of $M$, a well-behaved ``thick-thin''
  decomposition of $\H^3/\Gamma$, or a nice fundamental domain for the
  $\Gamma$-action on $\H^3$. This points towards the applicability of
  \Cref{thm:main_theorem} in settings where such tools might not be
  available (in particular, in higher rank).
\end{example}

\subsection{Relatively Anosov representations}

An important special case of $Q$-EGF representations are the
relatively $Q$-Anosov representations originally defined by by
Kapovich--Leeb \cite{kl2018relativizing} and Zhu \cite{Z2021} (see also
\cite{zz1}). Our main \Cref{thm:main_theorem} already applies to
relatively Anosov representations, but in this special case, we also
get better control over the convergence $\rho_n \to \rho$ and over the
convergence of the corresponding $Q$-limit sets.
\begin{theorem}
  \label{thm:rel_anosov_fillings}
  Let $(\Gamma, \pargps)$ be a relatively hyperbolic pair, and let
  $\rho:\Gamma \to G$ be a relatively $Q$-Anosov representation with
  $Q$-limit set $\Lambda$. Suppose that $\rho_n$ converges to $\rho$
  in an extended Dehn filling space
  $W \subseteq \ghom(\Gamma, G; \pargps)$. Then:
  \begin{enumerate}[label=(\arabic*)]
  \item $\rho_n$ converges strongly to $\rho$, and
  \item for sufficiently large $n$, $\rho_n$ induces a $Q$-EGF
    representation of $\Gamma^{\rho_n}$.
  \end{enumerate}
  If, in addition, $\rho_n(\Gamma)$ is $Q$-divergent (in particular,
  if it is relatively $Q$-Anosov) for every $n$, then the $Q$-limit
  sets $\Lambda_n$ of $\rho_n$ converge to $\Lambda$ with respect to
  Hausdorff distance on $\flags$.
\end{theorem}

Note that \Cref{cor:fill_to_anosov} applies in this context, meaning
that the theorem shows how to obtain (non-relatively) Anosov
representations as fillings of relatively Anosov representations.

\begin{remark}
  It is typically not difficult to check the extra hypothesis in the
  last part of the theorem, since a result of Wang \cite{wang} implies
  that it can often be reduced to a peripheral criterion (see
  \Cref{sec:relativization}). Also, as $Q$-EGF representations
  themselves have a (weaker) notion of ``limit set,'' a version of
  this part of the theorem still holds even without the added
  assumptions on the Dehn fillings $\rho_n$, but we defer the
  statement until \Cref{sec:relative_anosov}.
\end{remark}

\subsection{The rank one case}

When $(\Gamma, \pargps)$ is a relatively hyperbolic pair with every
$P \in \pargps$ virtually nilpotent, then a representation $\rho$ from
$\Gamma$ to some rank one Lie group $G$ is EGF precisely when $\rho$
has finite kernel and geometrically finite image (see
\cite{GW2024}). We can assume in this case that the peripheral
subgroups $\pargps$ are a maximal collection of pairwise non-conjugate
\emph{$\rho$-horospherical} subgroups (see
\Cref{defn:horospherical_subgroups}).

In this situation, the notion of an extended Dehn filling space in
$\relghom(\Gamma, G; \pargps)$ is conceptually much simpler.
\begin{theorem}
  \label{thm:rank_one_dehn_filling}
  Let $G$ be a rank-one semisimple Lie group, let $\rho:\Gamma \to G$
  be a geometrically finite representation, and let $\mc{P}$ be a
  maximal collection of pairwise-nonconjugate $\rho$-horospherical
  subgroups of $\Gamma$. Then $\relghom(\Gamma, G; \pargps)$ is an
  extended Dehn filling space about $\rho$.
\end{theorem}

By applying the result above, we obtain a general rank-one version of
our main Dehn filling theorem. To state this, let $\ghom(\Gamma, G)$
denote the space of representations $\Gamma \to G$ equipped with the
topology of geometric (or \emph{strong}) convergence.
\begin{theorem}
  \label{thm:rank_one_relative_strong}
  Suppose that $G$ is rank-one and $\rho:\Gamma \to G$ is
  geometrically finite. Let $\pargps$ be a maximal collection of
  pairwise non-conjugate $\rho$-horospherical subgroups of $\Gamma$,
  and let $\rho_n$ be a sequence of representations converging to
  $\rho$ in the compact-open topology on $\Hom(\Gamma, G)$. Then the
  following are equivalent:
  \begin{enumerate}[label=(\arabic*)]
  \item\label{item:rank1_strong} $\rho_n$ converges to
    $\rho$ \emph{strongly}, i.e. in the space $\ghom(\Gamma, G)$.
  \item\label{item:rank1_rel_strong} $\rho_n$ converges to $\rho$
    \emph{relatively strongly}, i.e. in the space
    $\relghom(\Gamma, G; \pargps)$.
  \item\label{item:rank1_edf} The set
    $\{\rho\} \cup \{\rho_n\}_{n \in \N}$ is an extended Dehn filling
    space about $\rho$.
  \end{enumerate}

  When this holds, then $\rho_n(\Gamma)$ is geometrically finite for
  all sufficiently large $n$, and the limit sets
  $\Lambda_n := \Lambda(\rho_n(\Gamma))$ converge to the limit set
  $\Lambda(\rho(\Gamma))$ with respect to Hausdorff distance.
\end{theorem}

As mentioned previously, this theorem applies directly to the case of
Dehn filling in real hyperbolic 3-space.

\begin{remark}
  When $G$ is the isometry group of $d$-dimensional real hyperbolic
  space, then \Cref{thm:rank_one_relative_strong} is equivalent to an
  earlier result of McMullen \cite{McMullen}. We remark that
  McMullen's proof implicitly assumes that each group $\rho_n(\Gamma)$
  is torsion-free, which avoids some technicalities. This is partially
  justified by using Selberg's lemma to replace $\Gamma$ with a
  finite-index subgroup $\Gamma'$ so that $\rho(\Gamma')$ is
  torsion-free. However, torsion cannot be entirely sidestepped this
  way: it is possible to construct examples where
  \Cref{thm:rank_one_relative_strong} applies, but each
  $\rho_n(\Gamma)$ contains torsion elements whose orders tend to
  infinity in $n$; in fact this is already possible when
  $G = \PSL(2, \R)$. In such situations, $\rho_n(\Gamma')$ eventually
  has nontrivial torsion elements for any fixed finite-index
  $\Gamma' < \Gamma$. Even in these cases, however, our proof of
  \Cref{thm:rank_one_relative_strong} goes through with no
  modification needed.

  Note also that if one assumes that $\rho$ and all of the
  representations $\rho_n$ are \emph{faithful} (meaning they do
  \emph{not} descend to nontrivial Dehn fillings of $\Gamma$), then
  \Cref{thm:rank_one_relative_strong} follows from the recent paper
  \cite{GW2024}, which uses the special case of the main result
  mentioned in \Cref{rem:main_thm}\ref{main_thm_faithful_deformation}.
\end{remark}

\subsection{Examples} Below we describe a few interesting situations
where \Cref{thm:main_theorem} applies. In all of these examples, we
start with a representation $\rho$ which is the holonomy of a manifold
(or orbifold) $O_\infty$ with a \emph{$(G,X)$-structure}; then $\rho$
is deformed to obtain a sequence of holonomy representations $\rho_n$
for $(G,X)$-orbifolds $O_n$ converging to $O_\infty$, obtained by a
form of Dehn filling on $O_\infty$.

In several of the examples below, it is also possible to obtain a
\emph{topological} description of the Dehn filled orbifolds $O_n$, and
indeed such a topological description is an important part of
Thurston's original $\H^3$ Dehn filling theorem. As stated, our main
theorem really only gives group-theoretic information about $O_n$, but
it would be very interesting to explore what topological statements
can be derived from our result in various specific cases.

\subsubsection{Spherical CR structures on 3-manifolds} In
\cite{schwartz07}, Schwartz considered a form of Dehn filling for
3-manifolds with \emph{spherical CR structures},
i.e. $(G,X)$-structures where $G = \SU(2,1)$ and $X$ is the 3-sphere
forming the ideal boundary of the complex hyperbolic plane
$\H^2_\C$. The main result in \cite{schwartz07} is a version of
\Cref{thm:rank_one_relative_strong} for ``horotube representations,''
which are a certain class of discrete representations
$\rho:\Gamma \to \SU(2,1)$. Schwartz additionally analyzed the
topology of the quotients $\Omega / \rho(\Gamma)$, where $\Omega$ is
the domain of discontinuity in $\dee \H^2_\C$ for a horotube
representation $\rho$. By considering perturbations of a specific
horotube representation $\rho$, Schwartz constructed infinitely many
examples of closed hyperbolic 3-manifolds admitting spherical CR
structures. Although this is not verified explicitly, it appears that
Schwartz's initial representation is geometrically finite, which means
that \Cref{thm:rank_one_relative_strong} applies to these examples.

Later, Acosta \cite{acosta2016,acosta2019} proved a related spherical
CR Dehn filling theorem, and applied it to perturbations of some other
discrete representations in $\SU(2,1)$, similar to Schwartz's horotube
representations. In some instances, Acosta's construction actually
yields spherical CR structures on closed non-hyperbolic
$3$-manifolds. It appears that the initial representations in Acosta's
examples are also geometrically finite; if this is the case, then they
should fit into our framework as well. This provides a strategy for
showing that the holonomy representations of Acosta's Dehn fillings
are actually discrete, which is not obvious from his original results.

\subsubsection{Coxeter group actions on real projective space} The
projective model for real hyperbolic $d$-space realizes $\H^d$ as a
convex ball in real projective space $\P(\R^{d+1})$, invariant under
the subgroup $\PO(d,1) < \PGL(d+1, \R)$. Thus any real hyperbolic
orbifold $O$ automatically has a \emph{convex projective structure};
in particular $O$ is a $(G,X)$-orbifold with $G = \PGL(d+1, \R)$ and
$X = \P(\R^{d+1})$. If the holonomy group of the orbifold is
geometrically finite in $\PO(d,1)$, then its inclusion into
$\PGL(d+1, \R)$ is a relatively Anosov representation.

In \cite{clm2020}, Choi--Lee--Marquis used the theory of Coxeter
polytopes to show that certain finite-volume hyperbolic orbifolds
$O_\infty$ in dimensions 4-7 can be realized as a limit of a sequence
$O_n$ of \emph{projective Dehn fillings}: convex projective orbifolds
whose orbifold fundamental groups $\pi_1O_n$ are (abstractly)
nontrivial Dehn fillings of $\pi_1O_\infty$. In their examples, each
holonomy representation $\pi_1O_n \to \PGL(d+1, \R)$ is an EGF
representation, and the sequence converges to the (relatively Anosov)
holonomy of $O$.

One can verify that the hypotheses of \Cref{thm:rel_anosov_fillings}
are satisfied by these examples. Notably, the holonomies of the $O_n$
are \emph{not} relatively Anosov, so this is one situation where the
more general EGF framework is useful.

\subsubsection{Exotic convex projective structures on 3-manifolds}
Similarly to the above, upcoming work of
Ballas--Danciger--Lee--Marquis (communicated to the author, to appear
as \cite{BDLM}) shows that certain finite-volume hyperbolic
3-manifolds can be Dehn filled to yield closed convex projective
3-manifolds, whose holonomy in $\PGL(d+1, \R)$ converges to the
holonomy of the original manifold. The closed 3-manifolds in question
admit hyperbolic structures, but the projective structure yielded by
the Dehn filling is \emph{not} projectively equivalent to the
hyperbolic one (unique by Mostow rigidity). For these examples,
\Cref{cor:fill_to_anosov} also applies, and the Dehn-filled
representations are actually Anosov.

\subsubsection{Exotic geometrically finite $\H^3_\C$-manifolds} In
work in progress with Jeff Danciger, we will show how
\Cref{thm:rank_one_relative_strong} can be applied to obtain
``exotic'' Dehn fillings of $3$-manifold groups in $\PU(3, 1)$,
yielding new examples of geometrically finite and convex cocompact
group actions on complex hyperbolic $3$-space. These examples are
closely related to the ones in \cite{BDLM}, and we expect that many
of them cannot be obtained as faithful deformations of previously
known examples.

\subsubsection{Rigidity obstructions; other notions of Dehn filling}
When $G = \PSL(2, \C)$, Thurston's Dehn filling theorem can be applied
to construct examples of \emph{lattices} in $G$, but it turns out that
this case is rather special. Due to seminal rigidity theorems of
Garland--Raghunathan \cite{gr1970} and Margulis \cite{margulis},
whenever $\Gamma$ is an irreducible lattice in a semisimple Lie group
$G$ not locally isomorphic to $\PSL(2, \C)$ or $\PSL(2, \R)$, then the
inclusion $\Gamma \hookrightarrow G$ is locally rigid in the character
variety $X(\Gamma, G)$. So we should only hope to find applications of
\Cref{thm:main_theorem} to $(G,X)$-orbifolds in cases where either $X$
is \emph{not} a symmetric space, or else the initial orbifold
$O_\infty$ has infinite volume.

We mention that in \cite{mr2018}, Martelli--Riolo showed that a form
of hyperbolic Dehn surgery nevertheless exists for certain
finite-volume hyperbolic 4-manifolds. Their examples are not covered
by \Cref{thm:main_theorem}; they do not violate the rigidity theorems
mentioned above because the finite-volume manifolds in question are
first \emph{drilled} before their hyperbolic structures are deformed
(and subsequently surgered). Thus one deforms the holonomy
representation for an \emph{incomplete} hyperbolic 4-manifold, and the
initial representation has infinite kernel while its deformations may
not. This contrasts with the framework of \Cref{thm:main_theorem},
where the original EGF representation always has \emph{finite} kernel.

\subsection{Proof strategy}

Our proof of \Cref{thm:main_theorem} uses two different combinatorial
tools. The first is the \emph{cusped space} $X = X(\Gamma, \pargps)$
associated to any relatively hyperbolic pair $(\Gamma, \pargps)$. This
is a locally finite graph originally defined by Groves--Manning
\cite{GM2008}, which allows us to study the interplay between the
coarse geometry of a relatively hyperbolic group and its
\emph{group-theoretic Dehn fillings} (see
\Cref{sec:rel_hyp_dehn_filling}). The second tool is a \emph{relative
  quasi-geodesic automaton} $\mc{G}$ which gives a ``relatively
locally finite'' encoding of the action of a relatively hyperbolic
group on a flag manifold $\flags$.

Given a $Q$-EGF representation $\rho:\Gamma \to G$, we use a
construction from \cite{Weisman2022} to build such a relative
automaton $\mc{G}$, which has finitely many vertices in its underlying
graph. Each vertex of the automaton is associated to an open subset of
$\flags$. The edges of the automaton correspond to rules describing
when these open subsets must map inside of each other under
translation by certain elements of $\Gamma$.

In \cite{Weisman2022} (see also \cite{MMW24}), it is shown that if one
deforms $\rho$ in $\Hom(\Gamma, G)$ in a way which respects all of
these rules, then one can actually recover a form of ``geometrical
finiteness'' from the combinatorial picture. We take a similar
approach here, but the key difference is that we allow ourselves to
deform $\rho$ in a way which does \emph{not} respect all of the
conditions imposed by the automaton. Specifically, if the deformation
descends to a representation of a group-theoretic Dehn filling
$\Gamma/N$, then we are allowed to \emph{ignore} conditions related to
elements that are close to $N$. We construct ``limit points'' in
$\flags$ of geodesics in the cusped space $X(\Gamma/N)$ for the
Dehn-filled group by relating those geodesics to paths in the
automaton $\mc{G}$ that only use valid rules. These limit points can
be organized into a \emph{boundary extension} for the deformed
representation, and ultimately this shows that the deformation
descends to an EGF representation of $\Gamma/N$.

\subsection{Outline of the paper} We provide necessary background on
relatively hyperbolic Dehn filling and EGF representations in
\Cref{sec:relhyp,sec:egf_background}. For technical reasons, the
definition of an EGF representation in this paper is slightly
different than the one that appeared in the original version of
\cite{Weisman2022}, but \Cref{sec:egf_background} also contains a
result (proved in the appendix) resolving some of these issues.

In \Cref{sec:automaton}, we review some properties of the relative
automaton mentioned above. Then we put everything together and prove
our main theorem in \Cref{sec:main_theorem}. We discuss applications
towards relatively Anosov representations in
\Cref{sec:relative_anosov}, and towards rank-one geometrical
finiteness in \Cref{sec:rank_one}.

\subsection{Acknowledgments} The author thanks Jeff Danciger for
initially suggesting the connection between EGF representations and
Dehn filling and for kindly explaining the details of some relevant
upcoming work. The author also thanks Antonin Guilloux for further
helpful explanations. This work was partially supported by NSF grant
DMS-2202770.

\section{Relative hyperbolicity and group-theoretic Dehn filling}
\label{sec:relhyp}

In this section we give some reminders about the coarse geometry of
relatively hyperbolic groups and their \emph{Dehn fillings},
group-theoretic analogs of Dehn fillings of hyperbolic
3-manifolds. Our main aim is to establish some basic relationships
between the various metric spaces associated to a relatively
hyperbolic group and its Dehn fillings. We refer to
\cite{Bowditch2012}, \cite{GM2008} for background; see also Section 2
of \cite{Weisman2022} for a brief introduction to the theory of
relatively hyperbolic groups, with an eye towards the convergence
group viewpoint (relevant in this paper).

\subsection{Setup}

Throughout this section, we fix a relatively hyperbolic pair
$(\Gamma, \pargps)$. Here $\Gamma$ is a finitely generated group, and
$\pargps$ is the finite collection of peripheral subgroups. We will
assume throughout this paper that all of our relatively hyperbolic
groups are finitely generated; this ensures that every group in
$\pargps$ is finitely generated also (see \cite{gerasimov}).

We do \emph{not} assume that each $P \in \pargps$ is an infinite
group. This differs from the convention in the earlier paper
\cite{Weisman2022}, and introduces some minor extra technicalities in
the definition of the Bowditch boundary of $(\Gamma, \pargps)$
(\Cref{sec:bowditch_boundary}) and of an EGF representation
$\Gamma \to G$ (\Cref{sec:egf_background}). Throughout, we let
$\ipargps$ denote the set of infinite subgroups in $\pargps$. It is
always true that $(\Gamma, \ipargps)$ is also a relatively hyperbolic
pair whenever $(\Gamma, \pargps)$ is.

The Bowditch boundary of $(\Gamma, \pargps)$ (see
\Cref{sec:bowditch_boundary}) is denoted $\bgamp$. Each
$P \in \pargps$ is the stabilizer of a unique parabolic point
$p \in \bgamp$. We let $\Pi \subset \bgamp$ denote the (finite) set of
parabolic points fixed by the groups in $\pargps$, and let
$\Pi_\infty \subset \bigamp$ denote the parabolic points fixed by
groups in $\ipargps$.

We also fix a finite generating set $S$ for $\Gamma$ which is
\emph{compatible} with $\pargps$, in the sense that $S \cap P$
generates $P$ for each $P \in \pargps$. Cayley graphs of $\Gamma$ and
each $P \in \pargps$ will always be defined with respect to $S$ or
$S \cap P$.

\subsection{Metrics on relatively hyperbolic groups}

In general, a relatively hyperbolic group comes equipped with several
natural (but inequivalent) metrics. In this paper we will frequently
work with three of them:
\begin{enumerate}
\item The \emph{word metric}: we use $\lgam{\gamma}$ to denote the
  word-length of $\gamma$ with respect to the (implicit) finite
  generating set $S$. Then the distance $\dgam(h, g)$ is equal to
  $\lgam{h^{-1}g}$; equivalently, $\dgam$ is the path metric on
  $\cay(\Gamma) = \cay(\Gamma, S)$.
  
\item The ``electrified'' or \emph{coned-off} metric: for this metric,
  we consider the \emph{relative Cayley graph} $\conedgraph$, obtained
  by adding one vertex to $\cay(\Gamma)$ for each coset of each group
  in $\pargps$, and then adding an edge from each element of the coset
  to that vertex. The metric $\dcon$ is the path metric on
  $\conedgraph$ (restricting to a metric on $\Gamma$), and the
  \emph{relative length} $\lcon{\gamma}$ denotes the $\dcon$-distance
  from $\gamma$ to $\identity$.
\item The \emph{cusped} metric: this metric, denoted $\dcusp$, is
  obtained by gluing \emph{combinatorial horoballs} to cosets of
  parabolic subgroups in $\cay(\Gamma)$, rather than coning them
  off. We describe the construction (originally due to Groves--Manning
  \cite{GM2008}) in more detail below, since we will need to work
  closely with it at several points in the paper.
\end{enumerate}

\begin{definition}[{Combinatorial horoballs; see
    \cite[Def. 3.1]{GM2008}}]
  \label{defn:gm_combinatorial_horoballs}
  Let $Y$ be a connected graph. The \emph{combinatorial horoball} over
  $Y$, denoted $\horb{Y}$, is a graph defined as follows:
  \begin{itemize}
  \item The vertex set of $\horb{Y}$ is the set
    $Y^{(0)} \times (\{0\} \cup \N)$, where $Y^{(0)}$ is the set of
    vertices of $Y$.
  \item There are two types of edges in $\horb{Y}$:
    \begin{itemize}
    \item For any $k \ge 0$ and any vertices $u,v$ of $Y$ such that
      $0 < d_Y(u,v) \le 2^k$, there is an edge joining $(u, k)$ to
      $(v,k)$. Such edges are called \emph{horizontal}.
    \item For any $k \ge 0$ and any vertex $u \in Y$, there is an edge
      joining $(u, k)$ to $(u, k+1)$. Such edges are called
      \emph{vertical}.
    \end{itemize}
  \end{itemize}
\end{definition}

Note that there is always an embedding of graphs
$Y \hookrightarrow \horb{Y}$ induced by the mapping of vertices
$y \mapsto (y, 0)$. Due to \cite{GM2008}, there is a uniform
$\delta > 0$ so that any combinatorial horoball $\horb{Y}$ is a
$\delta$-hyperbolic metric space (independent of $Y$). This hyperbolic
metric space always has a unique point in its ideal boundary,
corresponding to any ray in $\horb{Y}$ consisting entirely of vertical
edges.

The \emph{cusped space} $X(\Gamma, \pargps)$ for the relatively
hyperbolic pair $(\Gamma, \pargps)$ is the space obtained by the
following procedure: for each coset $gP$ of each $P \in \mc{P}$,
identify $gP$ with $P$, and glue a copy of the horoball
$\horb{P} := \horb{\cay(P, S \cap P)}$ to $gP \subset \cay(\Gamma, S)$
along the inclusion
$\cay(P, S \cap P) \hookrightarrow \horb{\cay(P, S \cap P)}$.

Observe that $\Gamma$ acts isometrically and properly discontinuously
on $X(\Gamma, \pargps)$. By \cite[Thm. 3.25]{GM2008}, there exists
some $\delta > 0$ so that $X(\Gamma, \pargps)$ is $\delta$-hyperbolic,
and in fact this property characterizes relative hyperbolicity of
$(\Gamma, \pargps)$.

The \emph{cusped metric}, denoted $\dcusp$, is the induced path metric
on $X(\Gamma, \mc{P})$. This restricts to a metric on $\cay(\Gamma)$
via the inclusion $\cay(\Gamma) \hookrightarrow X(\Gamma, \mc{P})$. We
use $\lcusp{\gamma}$ to denote $\dcusp(\identity, \gamma)$.

\subsection{The Bowditch boundary}
\label{sec:bowditch_boundary}

The \emph{Bowditch boundary} $\bgamp$ of the relatively hyperbolic
pair $(\Gamma, \pargps)$ may be defined as the Gromov boundary of the
the $\delta$-hyperbolic metric space $X = X(\Gamma, \pargps)$. The
isometric action of $\Gamma$ on $X$ induces a homeomorphic action of
$\Gamma$ on $\bgamp$. Since we assume that $\Gamma$ is finitely
generated, $\bgamp$ is compact and metrizable. We will often work with
a fixed metric $\bdist$ on $\bgamp$. This can be taken to be a visual
metric, but none of our arguments are sensitive to the choice. By
adjoining $\bgamp = \dee X$ to $X$, we obtain the \emph{compactified
  cusped space} $\overline{X} = \overline{X}(\Gamma, \pargps)$. The
\emph{Bowditch compactification} $\overline{\Gamma}$ of $\Gamma$ is
the closure of $\Gamma$ in $\overline{X}$.

If all of the groups in $\pargps$ are infinite, then every point in
$\dee X$ is an accumulation point of $\Gamma$, meaning that the
Bowditch compactification can be identified (as a $\Gamma$-set) with
the disjoint union $\Gamma \sqcup \bgamp$. This is no longer true if
some group $P \in \pargps$ is finite. In this case, the combinatorial
horoball $\mc{H}(P)$ over $P$ is quasi-isometric to a ray, the single
point in the ideal boundary of $\mc{H}(P) \subset X$ is identified
with an isolated point in $\dee X$, and this point is not an
accumulation point of any sequence in $\Gamma$.

The intersection $\overline{\Gamma} \cap \dee X$ is canonically
identified with the Bowditch boundary $\dee(\Gamma, \ipargps)$, and it
is always compact and $\Gamma$-invariant. Whenever $(\Gamma, \pargps)$
is non-elementary (meaning that $\bigamp$ contains at least three
points), then $\bigamp$ is precisely the set of \emph{non-isolated}
points in $\dee X$, and the $\Gamma$-action on $\bigamp$ is minimal.

\subsection{Convergence actions}
\label{sec:convergence_actions}

In this paper, we link relative hyperbolicity to geometrical
finiteness through the theory of \emph{convergence group actions}. A
group $\Gamma$ \emph{acts as a (discrete) convergence group} on a
compact metrizable space $M$ if it acts properly discontinuously on
the space of triples of pairwise distinct points in $M$. An equivalent
characterization, due to Bowditch \cite{Bowditch99}, is the following:
for every sequence $\gamma_n$ of pairwise distinct elements in
$\Gamma$, after extracting a subsequence, there are points
$a, b \in M$ such that $\gamma_n$ converges to the constant map $b$ on
$M \minus \{a\}$, uniformly on compact subsets of $M \minus \{a\}$. If
$Y$ is a proper hyperbolic metric space, then any discrete group of
isometries acts on both $\dee Y$ and $\overline{Y}$ as a discrete
convergence group \cite{Tukia94}.

When $\Gamma$ acts on a general space $M$ as a convergence group, a
point $z \in M$ is a \emph{conical limit point} if there exists some
sequence $\gamma_n$ and \emph{distinct} points $a, b \in M$ so that
$\gamma_n^{-1}z \to a$ and $\gamma_n^{-1}y \to b$ for every $y \in M$
distinct from $z$. A point $p \in M$ is \emph{parabolic} if it has
infinite stabilizer in $\Gamma$, and every infinite-order element in
its stabilizer uniquely fixes $p$. A parabolic point is \emph{bounded}
if its stabilizer acts cocompactly on $M \minus \{p\}$. A key fact we
need throughout this paper is:
\begin{theorem}[\cite{Bowditch2012}]
  If $(\Gamma, \pargps)$ is a relatively hyperbolic pair, then
  $\Gamma$ acts on the Bowditch boundary $\bgamp$ as a convergence
  group, and every point in $\bgamp$ is either a conical limit point
  or a bounded parabolic point for the $\Gamma$-action. The set of
  stabilizers of parabolic points are precisely the conjugates of the
  groups in $\pargps$.
\end{theorem}

A theorem of Yaman \cite{yaman2004topological} shows that the behavior
in the theorem actually characterizes relative hyperbolicity of the
pair $(\Gamma, \pargps)$, but we will not use that fact in this paper.

\subsection{Geometry of cusped spaces}

Below, we recall a few more features of the cusped space
$X = X(\Gamma, \pargps)$, and prove some routine lemmas which allow us
to use hyperbolicity of $X$ to deduce properties of certain paths in
the group $\Gamma$.

\subsubsection{Horoballs and regular geodesics}

The \emph{depth} of a vertex in $X$ is its distance to
$\Cay(\Gamma) \subset X$; this extends affinely on edges in $X$ to a
function $D:X \to \R_{\ge 0}$, also called depth. For any $k \ge 0$,
the intersection of $D^{-1}([k, \infty))$ with a combinatorial
horoball $\mc{H} \subset X$ is a connected set called a
\emph{$k$-horoball}, or simply a \emph{horoball}. Due to
\cite[Lem. 3.26]{GM2008}, if $X$ is $\delta$-hyperbolic, then any
$k$-horoball for $k \ge \delta$ is isometrically embedded in $X$. In
particular, this implies that any horoball $\mc{H}$ contains a
horoball $\mc{H}'$, isometrically embedded in $X$, such that the
Hausdorff distance between $\mc{H}$ and $\mc{H}'$ is at most $\delta$.

If $\mc{H}$ is a horoball, a segment in $\mc{H}$ is called
\emph{regular} if it consists of a (possibly degenerate) vertical
segment, followed by a horizontal segment of length at most 3,
followed by a (possibly degenerate) vertical segment. Lemma 3.10 in
\cite{GM2008} says that every pair of points in a combinatorial
horoball $\mc{H}$ can be joined by a regular segment which is a
geodesic for the path metric on $\mc{H}$.

\subsubsection{Paths and quasi-geodesics in $X$}

In this paper, a \emph{path} in a graph $Y$ (that is, a finite,
infinite, or bi-infinite sequence of vertices in $Y$, such that
consecutive vertices are adjacent) will always be implicitly
identified with its geometric realization, which extends the sequence
to a continuous map $c$ from an interval $I \subseteq \R$ to the
geometric realization of $Y$.

Any geodesic in $Y$ is a path, but the same need not be true for a
quasi-geodesic. However, for technical convenience, we will nearly
always work with quasi-geodesic \emph{paths} rather than arbitrary
quasi-geodesics. A \emph{ray} $r:[0, \infty) \to Y$ or a
\emph{segment} $c:[a,b] \to Y$, quasi-geodesic or otherwise, is always
assumed to be a path. When we refer to the ``distance'' or ``Hausdorff
distance'' between two paths, or between a path and some other subset
of $Y$, we always mean this with respect to images of geometric
realizations.

\subsubsection{Comparing $\dgam$ and $\dcusp$}

We have the following comparison estimate for the cusped metric and
word metric on $\Gamma$.
\begin{lemma}
  \label{lem:cusp_cay_comparison}
  Let $u, v \in \Gamma$. Then we have
  \[
    \dcusp(u,v) \le \dgam(u, v) \le \dcusp(u,v)\sqrt{2}^{\dcusp(u,v)}.
  \]
\end{lemma}
The right-hand side of this estimate is rather weak. Note, however,
that it allows us to bound $\dgam$ in terms of $\dcusp$, in a way
which does not depend at all on $\Gamma$ or $\pargps$.

\begin{proof}
  Let $X = X(\Gamma)$. The left-hand inequality is immediate since the
  inclusion $\cay(\Gamma) \hookrightarrow X$ is an embedding of
  graphs, and $\dgam$ and $\dcusp$ are respectively defined as path
  metrics on $\cay(\Gamma)$ and $X$.

  For the second inequality, consider a geodesic $c:I \to X$ between
  $u$ and $v$, and suppose that the length of this geodesic is
  $L$. Suppose $(a,b)$ is a connected component of
  $c^{-1}(X \minus \cay(\Gamma))$, meaning the restriction of $c$ to
  $(a,b)$ is contained in a single combinatorial horoball
  $\mc{H}$. Consider a regular geodesic in $\mc{H}$ between $c(a)$ and
  $c(b)$. The length of this geodesic is at most $L$, and the vertical
  parts of this regular geodesic have length at most $L/2$, which
  means that the $\dgam$-distance between the endpoints of the regular
  geodesic is at most $3 \sqrt{2}^L$.

  Observe that the restriction of $c$ to each component of
  $c^{-1}(X \minus \cay(\Gamma))$ is a sub-path containing at least
  two vertical edges and at least one horizontal edge, meaning it must
  have length at least $3$. So the number of components of
  $c^{-1}(X \minus \cay(\Gamma))$ is most $L/3$. Thus, if we we
  replace all of the sub-paths in $c^{-1}(X \minus \cay(\Gamma))$ with
  geodesics in $\cay(\Gamma)$, the argument from the previous
  paragraph shows that the overall length of the path increases by a
  multiplicative factor of at most
  $1/3 \cdot 3\sqrt{2}^L = \sqrt{2}^L$ and we are done.
\end{proof}

\subsubsection{Coarse geometry of paths in $X$}

Below, we prove several results which formalize the idea that
intersections of geodesics and quasi-geodesics with horoballs in the
cusped space $X$ are coarsely well-defined. We start with the
following lemma, which quantifies the fact that any geodesic whose
endpoints lie near some horoball $\mc{H}$ must enter ``deep into''
$\mc{H}$.
\begin{lem}
  \label{lem:geodesics_enter_horoballs_deeply}
  Let $X = X(\Gamma, \pargps)$ be the cusped space for
  $(\Gamma, \pargps)$, and suppose that $X$ is
  $\delta$-hyperbolic. Let $x, y \in X$ be points lying within
  $X$-distance $C$ of some horoball $\mc{H}$ in $X$, and let $\ell$ be
  a geodesic segment in $X$ from $x$ to $y$. For any $z \in \ell$, we
  have
  \[
    d_X(z, \{x, y\}) \le d_X(z, X \minus \mc{H}) + 3C + 7\delta.
  \]
\end{lem}
\begin{proof}
  Let $\mc{H}' \subset X$ be an isometrically embedded horoball within
  Hausdorff distance $\delta$ of $\mc{H}$, and let $x', y'$ be points
  in $\mc{H}'$ satisfying
  \[
    d_X(x, x') \le C + \delta, \qquad d_X(y, y') \le C + \delta.
  \]
  Let $\ell'$ be a regular geodesic segment between $x'$ and
  $y'$. Then the Hausdorff distance between $\ell$ and $\ell'$ is at
  most $C + 3\delta$.

  Let $z'$ be a point on $\ell'$ so that $d_X(z, z') \le C +
  3\delta$. Since $\ell'$ is a regular geodesic between
  $x', y' \in \mc{H}'$, we have
  \begin{equation}
    \label{eq:horoball_complement_dist}
    d_X(z', X \minus \mc{H}) \ge \min(d_X(z', x'), d_X(z', y')) - 3.
  \end{equation}
  On the other hand we have
  \begin{align*}
    d_X(z', x') &\ge d_X(z, x) - d_X(x, x') - d_X(z', z)\\
                &\ge d_X(z, x) - 2C - 4\delta.
  \end{align*}
  Similarly, we have $d_X(z', y') \ge d_X(z, y) - 2C - 4\delta$, so
  putting these together we get
  \[
    \min(d_X(z', x'), d_X(z', y')) \ge d_X(z, \{x, y\}) - 2C -
    4\delta.
  \]
  Then applying this with \eqref{eq:horoball_complement_dist} we get
  \begin{align*}
    d_X(z, X \minus \mc{H}) &\ge d_X(z', X \minus \mc{H}) - d_X(z, z')\\
                            &\ge d_X(z, \{x, y\}) - 3C - 7\delta
  \end{align*}
  which is equivalent to the desired inequality.
\end{proof}

For the statement of the next result (and for the rest of the paper),
when $Y$ is a subset of some space equipped with a metric $d$, and
$R > 0$, then $\nbhd[d]{Y}{R}$ denotes the $R$-neighborhood of $Y$
with respect to the metric $d$. If the metric is understood from
context, then we will write $\nbhd{Y}{R} = \nbhd[d]{Y}{R}$.
\begin{lem}
  \label{lem:cayley_hausdorff_neighborhoods}
  Let $X = X(\Gamma, \pargps)$ be the cusped space for
  $(\Gamma, \pargps)$, and suppose $X$ is $\delta$-hyperbolic. Given
  constants $K \ge 1$ and $A, D \ge 0$, there exists $R \ge 0$
  (depending only on $K, A, D, \delta$) so that the following
  holds. If $c:I \to X$ is a $(K,A)$-quasi-geodesic path, and each
  (non-ideal) endpoint of $c$ lies in $\cay(\Gamma)$, then the
  pairwise Hausdorff distances between the three sets
  \[
    c(I) \cap \cay(\Gamma), \quad \nbhd[\dgam]{c(I)}{D} \cap
    \cay(\Gamma), \quad \nbhd[\dcusp]{c(I)}{D} \cap \cay(\Gamma)
  \]
  with respect to any of the metrics $\dgam$, $\dcon$, or $\dcusp$ are
  bounded by $R$.
\end{lem}

\begin{remark}
  We allow the quasi-geodesic $c$ in the statement of the lemma to
  have two, one, or zero non-ideal endpoints (meaning $c$ could be a
  segment, a ray, or a bi-infinite path in $X$). Since interiors of
  combinatorial horoballs are pairwise disjoint, our assumption on
  endpoints only disallows the cases where $c$ is a ray entirely
  contained in the interior of a single combinatorial horoball, or a
  finite segment which either starts or ends in the interior of a
  horoball. This also applies to \Cref{cor:hausdorff_distance_bound}
  below.
\end{remark}

\begin{proof}[Proof of \Cref{lem:cayley_hausdorff_neighborhoods}]
  Because of the right-hand inequality in
  \Cref{lem:cusp_cay_comparison}, and because of the fact that
  $\dcon(u,v) \le \dgam(u,v)$ for all $u,v \in \cay(\Gamma)$, it
  suffices to prove the bound on Hausdorff distance with respect to
  the metric $\dcusp$. We will show that there is a uniform constant
  $R$ so that $\nbhd[\dcusp]{c(I)}{D} \cap \cay(\Gamma))$ lies in an
  $R$-neighborhood of $c(I) \cap \cay(\Gamma)$, with respect to
  $\dcusp$. This is sufficient, because the left-hand inequality in
  \Cref{lem:cusp_cay_comparison} tells us that
  \[
    c(I) \subseteq \nbhd[\dgam]{c(I)}{D} \subseteq
    \nbhd[\dcusp]{c(I)}{D}.
  \]

  So, suppose that $y \in \nbhd[\dcusp]{c(I)}{D} \cap \cay(\Gamma)$,
  and let $z$ be a point on $c(I)$ such that $\dcusp(y, z) \le D$. If
  $z$ lies in $\cay(\Gamma)$ then $w$ already lies in the
  $D$-neighborhood of $c(I) \cap \cay(\Gamma)$, so assume that $z$
  lies in $c((a, b))$, where $(a,b) \subset I$ is a component of
  $c^{-1}(X \minus \cay(\Gamma))$. Since distinct combinatorial
  horoballs have disjoint interiors, the set $c((a,b))$ lies in a
  single combinatorial horoball $\mc{H}$.

  It is possible that either $a = -\infty$ or $b = +\infty$. However,
  we cannot have $(a, b) = (-\infty, +\infty)$, since $\mc{H}$ is a
  quasi-isometrically embedded subset of $X$ with a unique point in
  its ideal boundary. Thus, up to reversing the direction of $c$, we
  can assume that $a$ is finite, and that
  $\dcusp(c(a), z) \le \dcusp(c(b), z)$. Our assumption on the
  endpoints of $c$ means that $c(a)$ must lie in
  $\dee \mc{H} \subset \cay(\Gamma)$. We can also choose some finite
  $b' \in [a, b]$ so that $c(b) \in \mc{H}$ and
  $d_X(c(a), z) \le d_X(c(b'), z)$; if $b$ is already finite then we
  can take $b = b'$, and otherwise any sufficiently large $b' > a$
  will do.

  Let $\ell$ be a geodesic segment in $X$ in joining $c(a)$ to
  $c(b')$. By the Morse lemma, there is a uniform constant $M$
  (depending only on $K, A, \delta$) and a point $w$ on $\ell$ so that
  $d_X(z, w) \le M$. From \Cref{lem:geodesics_enter_horoballs_deeply}
  we know that
  \[
    \dcusp(w, \{c(a), c(b')\}) \le \dcusp(w, X \minus \mc{H}) +
    7\delta.
  \]
  Now, since $y \in \cay(\Gamma)$, we have
  \[
    \dcusp(w, X \minus \mc{H}) \le \dcusp(w, z) + \dcusp(z,y) \le M +
    D.
  \]
  We also know that
  \begin{align*}
    \dcusp(z, c(a)) = \dcusp(z, \{c(a), c(b')\}) \le \dcusp(w, \{c(a),
    c(b)\}) + M,
  \end{align*}
  so we conclude that $\dcusp(z, c(a)) \le 2M + D$. Thus
  $\dcusp(y, c(I) \cap \cay(\Gamma)) < 2M + 2D$, which is what we
  wanted to show.
\end{proof}

\begin{corollary}
  \label{cor:hausdorff_distance_bound}
  For any constants $K \ge 1$ and $A, D \ge 0$, there exists a
  constant $R \ge 0$ satisfying the following. If $c:I \to X$ and
  $c:I' \to X$ are a pair of $(K,A)$-quasi-geodesic paths in $X$ whose
  non-ideal endpoints all lie in $\cay(\Gamma)$, and the Hausdorff
  distance between $c$ and $c'$ with respect to $\dcusp$ is at most
  $D$, then the Hausdorff distance between $c(I) \cap \cay(\Gamma)$
  and $c'(I') \cap \cay(\Gamma)$ with respect to $\dgam$ is at most
  $R$.
\end{corollary}
\begin{proof}
  If the Hausdorff distance in $X$ between $c(I)$ and $c'(I')$ is at
  most $D$, then $c'(I')$ is contained in $\nbhd[\dcusp]{c(I)}{D}$,
  and thus $c'(I') \cap \cay(\Gamma)$ is contained in
  $\nbhd[\dcusp]{c(I)}{D} \cap \cay(\Gamma)$. By the previous lemma,
  this set is in turn contained in the $R$-neighborhood of
  $c(I) \cap \cay(\Gamma)$ for a uniform constant $R > 0$ depending
  only on $K, A, D$, with respect to the metric $\dgam$. So,
  $c(I') \cap \cay(\Gamma)$ is contained in the $D + R$ neighborhood
  of $c(I) \cap \cay(\Gamma)$. Arguing symmetrically completes the
  proof, using $D + R$ for the constant $R$ in the statement of the
  corollary.
\end{proof}

\subsubsection{Approximating elements by rays}
Below, we record one more useful property of the cusped space
$X = X(\Gamma, \pargps)$.
\begin{prop}
  \label{prop:cusp_rays_quasidense}
  Suppose that $\pargps \ne \emptyset$ and $\pargps \ne \{\Gamma\}$,
  and that the cusped space $X = X(\Gamma, \mc{P})$ is
  $\delta$-hyperbolic. Then, for any $\gamma \in \Gamma$, there is a
  geodesic ray $r:[0, \infty) \to X$ so that $r(0) = \identity$ and
  $\dcusp(r(t), \gamma) \le 8 + 21\delta$ for some $t \ge 0$.
\end{prop}
\begin{proof}
  Fix $\gamma \in \Gamma$ and let $P \in \mc{P}$ be a proper subgroup
  of $\Gamma$. There is some $s$ in our compatible generating set $S$
  so that the cosets $P, sP$ are distinct, which in turn means that
  the cosets $\gamma P$, $\gamma sP$ are distinct. Let
  $\mc{H}_1, \mc{H}_2$ be isometrically embedded horoballs in $X$
  which are within Hausdorff distance $\delta$ of the combinatorial
  horoballs attached at $\gamma P$, $\gamma sP$, respectively, and for
  $i = 1,2$, let $p_i$ be the unique point in the ideal boundary of
  $\mc{H}_i$.

  For $i = 1,2$, let $x_i \in \dee \mc{H}_i$ be a point such that
  $\dcusp(\gamma, x_i) \le 1 + \delta$, and let $Q_i$ be a partly
  ideal geodesic quadrilateral with vertices
  $(\identity, \gamma, x_i, p_i)$ in that order.  We may choose $Q_i$
  so that the side between $x_i$ and $p_i$ is a vertical geodesic ray
  $[x_i, p_i) \subset \mc{H}_i$. The side of $Q_i$ between $\identity$
  and $p_i$ is parameterized by a unit-speed ray
  $r_i:[0, \infty) \to X$, which eventually lies in the horoball
  $\mc{H}_i$. We will show that $\gamma$ is within distance
  $8 + 21\delta$ of at least one of the rays $r_1, r_2$.

  Let $y_i \in \dee \mc{H}_i$ be the first point at which $r_i$ enters
  the horoball $\mc{H}_i$. We know $Q_i$ is $3\delta$-thin, so $y_i$
  is contained in a $3\delta$-neighborhood of either one of the
  geodesic segments $[\identity, \gamma]$ or $[\gamma, x_i]$, or else
  the ray $[x_i, p_i)$.

  If $y_i$ is contained in a $3\delta$-neighborhood of
  $[\gamma, x_i]$, then $\dcusp(y_i, \gamma) \le 1 +
  4\delta$. Alternatively, if $y_i$ is contained in a
  $3\delta$-neighborhood of some point $z_i \in [x_i, p_i)$, then
  $\dcusp(z_i, \dee \mc{H}_i) \le 3\delta$, hence
  $\dcusp(z_i, x_i) \le 3\delta$ and therefore
  $\dcusp(y_i, \gamma) \le 1 + 7\delta$. So we are done unless there
  are points $z_1, z_2 \in [\identity, \gamma]$ so that
  \[
    \dcusp(y_1, z_1) \le 3\delta, \qquad \dcusp(y_2, z_2) \le 3\delta.
  \]
  Without loss of generality, the points $z_1, z_2$ appear in that
  order on the geodesic $[\identity, \gamma]$. By considering a
  $2\delta$-thin geodesic quadrilateral with vertices
  $\gamma, x_1, y_1, z_1$, we see that the sub-geodesic
  $[z_1, \gamma] \subseteq [\identity, \gamma]$ is contained in the
  $(1 + 5\delta)$-neighborhood of a regular geodesic
  $[x_1, y_1] \subset \mc{H}_1$. Arguing similarly, a regular geodesic
  $[x_2, y_2] \subset \mc{H}_2$ is contained in the
  $(1 + 5\delta)$-neighborhood of
  $[z_2, \gamma] \subseteq [z_1, \gamma]$, hence in the
  $(2 + 10\delta)$-neighborhood of $[x_1, y_1]$.

  In particular, $[x_2, y_2]$ is contained in the
  $(2 + 10\delta)$-neighborhood of $X \minus \mc{H}_2$, and since
  $[x_2, y_2]$ is a regular geodesic in $\mc{H}_2$, this means that it
  has length at most $2(2 + 10\delta) + 3 = 7 + 20\delta$. Thus
  $\dcusp(y_2, \gamma) \le \dcusp(y_2, x_2) + \dcusp(x_2, \gamma) \le
  8 + 21\delta$.
\end{proof}

\begin{remark}
  A version of \Cref{prop:cusp_rays_quasidense} still holds even if
  $\pargps$ is empty (see e.g. \cite{bogopolski97infinite}). However,
  for some applications of the proposition, we will need to use the
  fact that the distance between $r$ and $\gamma$ depends only on
  $\delta$, and not on the space $X$ itself. It is easier to obtain
  this control when $\mc{P}$ is nonempty.
\end{remark}

\subsection{Dehn fillings}
\label{sec:rel_hyp_dehn_filling}

Below we recall the definition of a
\emph{Dehn filling} of the relatively hyperbolic pair
$(\Gamma, \pargps)$.

\begin{definition}
  Let $\pargps = \{P_1, \ldots, P_k\}$, and let $(N_1, \ldots, N_k)$
  be a set of subgroups of $\Gamma$ with $N_i \trianglelefteq P_i$ for
  each $i$. Let $N = \llangle \bigcup_{i=1}^k N_i \rrangle$. The
  quotient map $\pi:\Gamma \to \Gamma/N$ is called a \emph{Dehn
    filling} of $(\Gamma, \pargps)$, and the subgroups
  $N_i \trianglelefteq P_i$ are called \emph{Dehn filling kernels}.

  When $\pi:\Gamma \to \Gamma/N$ is a Dehn filling as above, then we
  use $\pargps^\pi$ to denote the collection of subgroups
  $\{\pi(P_1), \ldots \pi(P_k)\}$ of $\Gamma/N$. A property of
  $(\Gamma, \pargps)$ is said to hold for \emph{sufficiently long
    fillings} if there exists some finite set $B \subset \Gamma$ so
  that, whenever $(N_1, \ldots, N_k)$ is a set of Dehn filling kernels
  satisfying $N_i \cap B = \emptyset$ for all $i$, and
  $\pi:\Gamma \to \Gamma/N$ is the corresponding Dehn filling, then
  the pair $(\Gamma/N, \pargps^\pi)$ has this property.
\end{definition}

We have the following important theorem regarding Dehn fillings of
relatively hyperbolic groups, due to Osin and Groves--Manning. This
theorem can be thought of as a group-theoretic analog of the
hyperbolic Dehn filling theorem for Kleinian groups.
\begin{theorem}[\cite{Osin2007}, \cite{GM2008}]
  \label{thm:dehn_filling_criterion} Let $F$ be a finite subset of
  $\Gamma$. Then for any sufficiently long filling
  $\pi:\Gamma \to \Gamma/N$ of the relatively hyperbolic pair
  $(\Gamma, \pargps = \{P_1, \ldots, P_k\})$, all of the following
  hold:
  \begin{enumerate}
  \item For each $P_i \in \pargps$, the induced map
    $\pi:P_i/N_i \to \Gamma/N$ is injective. Thus $\pargps^\pi$ is
    identified with the set $\{P_1/N_1, \ldots, P_k/N_k\}$.
  \item The pair $(\Gamma/N, \pargps^\pi)$ is relatively hyperbolic.
  \item The restriction of $\pi$ to $F$ is injective.
  \end{enumerate}
\end{theorem}

\subsubsection{Geometry of Dehn fillings}
\label{sec:dehnfill_geometry} Throughout this paper, we will need a
basic understanding of the relationship between the coarse geometry of
a relatively hyperbolic group $\Gamma$ and the geometry of its Dehn
fillings $\pi:\Gamma \to \Gamma/N$. Here we make note of several
results in this direction.

Consider a Dehn filling $\pi:\Gamma \to \Gamma/N$. The subgroup $N$
acts by automorphisms on all three graphs $\cay(\Gamma)$,
$\hcay(\Gamma, \pargps)$, and $X(\Gamma, \pargps)$ (recall that these
are defined with respect to a fixed compatible generating set
$S$). The resulting quotient graphs are \emph{almost} isomorphic to
the graphs $\cay(\Gamma/N)$, $\hcay(\Gamma/N, \pargps^\pi)$, and
$X(\Gamma/N, \pargps^\pi)$, defined with respect to the finite
compatible generating set $\pi(S)$; there is only ``almost'' an
isomorphism because the latter graphs are obtained from the former by
deleting self-loops. Following e.g. \cite{gm2021}, since such
self-loops do not affect the metric on the zero-skeleton of any of
these graphs, we will frequently employ a slight abuse of notation and
assume that the loops are present in each of $\cay(\Gamma/N)$,
$\hcay(\Gamma/N, \pargps^\pi)$, and $X(\Gamma/N, \pargps^\pi)$.

This means that the Dehn filling $\pi:\Gamma \to \Gamma/N$ induces
equivariant surjective maps of graphs
\begin{align*}
  &\pi:\cay(\Gamma) \to \cay(\Gamma/N),\\
  &\hat{\pi}:\hcay(\Gamma, \pargps) \to \hcay(\Gamma/N, \pargps^\pi),\\
  &\pi_X:X(\Gamma, \pargps) \to X(\Gamma, \pargps^\sigma).
\end{align*}
The maps $\hat{\pi}$ and $\pi_X$ both restrict to $\pi$ on the copies
of $\cay(\Gamma)$ respectively embedded in $\hcay(\Gamma, \pargps)$
and $X(\Gamma, \pargps)$. Moreover, $\pi_X$ sends combinatorial
horoballs to combinatorial horoballs, and preserves vertical and
horizontal edges; vertical edges are never sent to loops. Finally,
since $\pi, \hat{\pi}$, and $\pi_X$ are all surjective, and the metric
on each graph is a path metric, each of $\pi, \hat{\pi}$, and $\pi_X$
is 1-Lipschitz.

A straightforward consequence of \Cref{thm:dehn_filling_criterion} is
the following:
\begin{proposition}
  \label{prop:dehnfill_local_isometry}
  Let $R > 0$. For all sufficiently long fillings
  $\pi:\Gamma \to \Gamma/N$, the restriction of $\pi_X$ to any ball of
  radius $R$ in $X(\Gamma, \pargps)$ centered at a point in
  $\cay(\Gamma)$ is an isometric embedding whose image is a metric
  ball.
\end{proposition}

The lemma above can be used to prove the following, which sharpens
part of \Cref{thm:dehn_filling_criterion}. This result is originally
due to Agol--Groves--Manning \cite{AGM2009}; the precise version here
is stated (with the given level of generality) in e.g. \cite{GMS2019}.
\begin{proposition}[{\cite[Prop. 2.3]{AGM2009}}]
  \label{prop:uniformly_hyperbolic_fillings}
  There exists a constant $\delta \ge 0$ satisfying the following: the
  cusped space $X = X(\Gamma, \pargps)$ is $\delta$-hyperbolic, and
  for all sufficiently long fillings $\pi:\Gamma \to \Gamma/N$, the
  cusped space $X(\Gamma/N, \pargps^\pi)$ is $\delta$-hyperbolic.
\end{proposition}

\subsubsection{Lifts} Whenever $\pi:\Gamma \to \Gamma/N$ is a
sufficiently long filling, it restricts to an injection on the
generating set $S$ and thus the induced map
$\cay(\Gamma) \to \cay(\Gamma/N)$ is a local homeomorphism of locally
finite graphs, hence a covering map. In this case paths in
$\cay(\Gamma)$ lift to paths in $\cay(\Gamma/N)$, uniquely after a
choice of basepoint in $\cay(\Gamma/N)$.

On the other hand, while the map
$\pi_X:X(\Gamma, \pargps) \to X(\Gamma/N, \pargps^\pi)$ is typically
\emph{not} a covering map, we can still work with lifts.
\begin{definition}
  Let $\pi:\Gamma \to \Gamma/N$ be a Dehn filling. If
  $c:I \to X(\Gamma/N, \mc{P}^\pi)$ is a path in the quotient cusp
  space, a \emph{lift} of $c$ is any path
  $\tilde{c}:I \to X(\Gamma, \mc{P})$ such that
  $\pi_X \circ \tilde{c} = c$.
\end{definition}

Since $\pi_X$ is a surjective map of (combinatorial) graphs, any path
in $X(\Gamma/N, \pargps^\pi)$ can be lifted edge-by-edge to obtain a
path in $X$. However, while lifts exist, they need not be unique, even
after fixing a basepoint. We observe:

\begin{proposition}
  Let $c:I \to X(\Gamma, \pargps)$ be a path, and let $\tilde{c}$ be a
  lift of $c$. If $c$ is a $(K,A)$-quasi-geodesic, then so is
  $\tilde{c}$.
\end{proposition}
\begin{proof}
  This is immediate from the fact that $\pi_X$ is 1-Lipschitz with
  respect to the metrics $\dcusp$, $d_X^\pi$ on $X(\Gamma, \pargps)$
  and $X(\Gamma/N, \pargps^\pi)$.
\end{proof}

In the opposite direction, we also have the following:
\begin{proposition}
  \label{prop:cay_geodesics_descend}
  For any $R > 0$, and any $K > 1$, any sufficiently long Dehn filling
  $\pi:\Gamma \to \Gamma/N$ satisfies the following. Suppose that
  $\tilde{c}:I \to X(\Gamma, \pargps)$ is a geodesic, and the maximum
  depth of any point on $\tilde{c}$ is at most $R$. Then
  $c = \pi_X \circ \tilde{c}$ is a $(K, 2\delta)$-quasi-geodesic in
  $X(\Gamma/N, \pargps^\pi)$.
\end{proposition}
\begin{proof}
  Fix $\delta$ as in \Cref{prop:uniformly_hyperbolic_fillings} so that
  the cusped space for all sufficiently long fillings of
  $(\Gamma, \pargps)$ is $\delta$-hyperbolic, and let $k >
  8\delta$. By \Cref{prop:dehnfill_local_isometry}, for any
  sufficiently long filling $\pi:\Gamma \to \Gamma/N$, if $B$ is any
  ball of radius $R + k$ centered at a point in $\cay(\Gamma)$, then
  $\pi_X$ restricts to an isometric embedding on $B$. By assumption,
  every point in $\tilde{c}$ is contained in the ball of radius $R$
  centered at some point in $\cay(\Gamma)$, so this implies that
  $\pi_X \circ \tilde{c}$ is a $k$-local geodesic in
  $X(\Gamma/N, \pargps^\pi)$.

  It now follows from the local-to-global principle in
  $\delta$-hyperbolic metric spaces that $\pi_X \circ \tilde{c}$ is a
  (global) quasi-geodesic. For a precise estimate, we can use e.g.
  \cite[III.H.1.13]{bridsonhaefliger} to see that it is a
  $\left(\dfrac{k + 4\delta}{k - 4\delta},
    2\delta\right)$-quasigeodesic. Since $k > 8\delta$ could be chosen
  arbitrarily large, this completes the proof.
\end{proof}

\section{Extended geometrically finite representations}
\label{sec:egf_background}

In this section, we recall the definition of an \emph{extended
  geometrically finite} (EGF) representation of a relatively
hyperbolic group $\Gamma$ into a semisimple Lie group $G$. EGF
representations were originally defined in \cite{Weisman2022}, and we
refer to that paper for further detail. See also Section 2 of
\cite{Weisman2023} for a brief introduction, focusing on the case
where $G = \PGL(d, \R)$.

When $G$ has rank one, then EGF representations are closely related to
(but not exactly the same as) representations with finite kernel and
geometrically finite image. We give a precise statement connecting the
two notions in \Cref{sec:rank_one} of this paper. See also Section 2
of \cite{GW2024} for a detailed account of the relationship.

\subsection{Extended convergence actions}
\label{sec:extended_convergence}
The definition of an EGF representation is based on the notion of an
\emph{extended convergence action}. Extended convergence actions
generalize the notion of a convergence action of a group $\Gamma$ on a
compact metrizable space $M$ (see \Cref{sec:convergence_actions}), in
a way which accommodates many discrete group actions on the flag
manifold $G/Q$ associated to a semisimple Lie group $G$ and a
parabolic subgroup $Q < G$.

\begin{definition}[See {\cite[Def. 1.2]{Weisman2022}}]
  \label{defn:extended_convergence_group}
  Let $(\Gamma, \mc{P})$ be a relatively hyperbolic pair, with
  $\Gamma$ acting on a compact metrizable space $M$ by
  homeomorphisms. Let $\Lambda \subset M$ be a closed
  $\Gamma$-invariant set. We say that a $\Gamma$-equivariant
  surjective map $\phi:\Lambda \to \bigamp$ \emph{extends the
    convergence group action} of $\Gamma$ on $\bigamp$ if there exists
  an assignment $z \mapsto C_z$ of points $z \in \bigamp$ to open
  subsets $C_z \subset M$ satisfying the following properties:
  \begin{enumerate}[label=(C\arabic*)]
  \item\label{item:repelling_strata_interior} For every compact subset
    $K \subset \bigamp$, the intersection $\bigcap_{z \in K} C_z$
    contains an open neighborhood of $\phi^{-1}(\bigamp \minus K)$.
  \item\label{item:extended_convergence_prop} For every sequence
    $\gamma_n \in \Gamma$ such that (in the Bowditch compactification
    $\overline{\Gamma}$) we have $\gamma_n \to z_+$ and
    $\gamma_n^{-1} \to z_-$ for $z_\pm \in \bigamp$, every compact
    subset $K \subset C_{z_-}$, and every open set $U \subset M$
    containing $\phi^{-1}(z_+)$, we have $\gamma_n K \subset U$ for
    all sufficiently large $n$.
  \end{enumerate}
  The map $\phi$ is called a \emph{boundary extension}, the set
  $\Lambda$ is called a \emph{boundary set}, and the sets $C_z$ are
  called the \emph{repelling strata}.
\end{definition}

Note that, in the context of \Cref{defn:extended_convergence_group},
if the equivariant map $\phi:\Lambda \to \bigamp$ is a homeomorphism,
and if each repelling stratum $C_z$ is equal to
$M \minus \phi^{-1}(z)$, then the second part of the definition says
that $\Gamma$ acts on $M$ as a convergence group with limit set
$\Lambda$.

\begin{remark}
  \label{rem:extended_convergence}
  The original version of \Cref{defn:extended_convergence_group} given
  in \cite{Weisman2022} differs from the definition above in several
  respects. Most significantly, the definition in \cite{Weisman2022}
  replaces condition \ref{item:repelling_strata_interior} with the
  weaker condition:
  \begin{enumerate}[label=(C1*)]
  \item\label{item:repelling_strata_containment} for every
    $z \in \bigamp$, the set $C_z$ contains $\bigamp \minus \{z\}$.
  \end{enumerate}
  However, the stronger assumption
  \ref{item:repelling_strata_interior} is actually used implicitly in
  some of the general constructions in \cite{Weisman2022}, so we
  correct the definition by including it here. In
  \Cref{sec:egf_definition} below, when we consider $Q$-extended
  geometrically finite representations $\rho$ into a semisimple Lie
  group $G$, we will see that condition
  \ref{item:repelling_strata_interior} is equivalent to the
  combination of \ref{item:repelling_strata_containment} with a
  natural additional assumption on $\rho$ called \emph{relative
    $Q$-divergence}.

  We also comment that the version of
  \Cref{defn:extended_convergence_group} in \cite{Weisman2022} assumes
  that $\pargps = \ipargps$; however, this difference in the
  definitions is superficial, since if $(\Gamma, \pargps)$ is a
  relatively hyperbolic pair then so is $(\Gamma, \ipargps)$.
\end{remark}

\begin{remark}
  It is also possible to define extended convergence group actions for
  groups which are not necessarily relatively hyperbolic, but we will
  not address this here.
\end{remark}

\subsection{Semisimple Lie groups, parabolic subgroups, and flag
  manifolds}

Let $G$ be a semisimple Lie group, let $K < G$ be a maximal compact
subgroup, and let $\X$ be the associated nonpositively curved
Riemannian symmetric space $G/K$. A subgroup $Q < G$ is
\emph{parabolic} if it is the stabilizer in $G$ of some point
$z \in \dee \X$. Two parabolic subgroups $Q_+, Q_-$ are
\emph{antipodal} if they are respectively the stabilizers of points
$z_+, z_-$ in $\dee \X$ joined by a bi-infinite geodesic in $\X$. A
parabolic subgroup $Q$ is \emph{symmetric} if it has an antipodal
subgroup $Q'$ which is conjugate to $Q$ in $G$.

For any parabolic subgroup $Q < G$, the homogeneous $G$-space $G/Q$ is
called a \emph{flag manifold}. Two flags $\xi_1 \in G/Q_1$ and
$\xi_2 \in G/Q_2$ are \emph{antipodal} if their stabilizers in $G$ are
antipodal parabolic subgroups. If $Q_+, Q_-$ are antipodal parabolic
subgroups, and $\xi_+ \in G/Q_+$, then $\Opp(\xi_+) \subset G/Q_-$
denotes the set of flags antipodal to $\xi_+$. It is an open dense
subset of $G/Q_-$.

\begin{example}
  If $G$ has rank one, then $\X$ is negatively curved, so any two
  distinct points in $\dee \X$ are joined by a geodesic in $\X$. In
  addition $G$ acts transitively on $\dee \X$, so there is only one
  conjugacy class of parabolic subgroup in $G$ and all parabolic
  subgroups are symmetric. All of the flag manifolds for $G$ thus have
  a unique $G$-equivariant identification with the visual boundary of
  $\X$, and two flags in $\dee \X$ are antipodal if and only if they
  are distinct.
\end{example}

\begin{example}
  Let $G = \PGL(d, \R)$. Then $\PO(d)$ is a maximal compact subgroup
  of $G$. For each $k \in \{1, \ldots, d-1\}$, the stabilizer of a
  $k$-plane in $\R^d$ is a maximal parabolic subgroup of $G$, meaning
  that each Grassmannian of $k$-planes can be identified with a flag
  manifold $G/Q$. More generally, all of the different flag manifolds
  for $G$ can be uniquely identified with \emph{partial flag spaces}
  associated to nonempty subsets of $\{1, \ldots, d-1\}$; the partial
  flag space associated to a set of integers $\{i_1 < \ldots < i_k\}$
  is the space of \emph{partial flags}
  \[
    V_{i_1} \subset  \ldots \subset V_{i_k},
  \]
  where each $V_{i_j}$ is a subspace of $\R^d$ with dimension $i_j$.

  A pair of partial flags $V_{i_1} \subset \ldots \subset V_{i_k}$ and
  $W_{j_1} \subset \ldots \subset W_{j_k}$ are antipodal if they are
  \emph{transverse}, meaning that for any
  $i \in \{i_1, \ldots, i_k\}$, $V_i$ is defined if and only if
  $W_{d-i}$ is defined, and $V_i + W_{d - i} = \R^d$. A parabolic
  subgroup $Q$ is symmetric if it is the stabilizer of a partial flag
  of the form
  \[
    V_{i_1} \subset \ldots \subset V_{i_k} \subseteq V_{d-i_k} \subset
    \ldots \subset V_{d - i_1}.
  \]
\end{example}

\subsubsection{$Q$-divergence}

Returning to the context of a general semisimple Lie group $G$, let
$Q < G$ be a symmetric parabolic subgroup, and let $\Lambda$ be a
subset of the flag manifold $G/Q$.  We will say that a sequence $g_n$
in $G$ is \emph{$Q$-divergent} if, for every subsequence of $g_n$,
there exists an open set $U \subset \flags$ and a further subsequence
$g_m$ so that $g_mU$ converges to a singleton $\{\xi\}$ in
$\flags$. In this case $\xi$ is said to be a \emph{$Q$-limit point} of
the sequence $g_n$. Note that $Q$-divergence depends only on the
conjugacy class of $Q$ in $G$. Also, since we have assumed that $Q$ is
a symmetric parabolic subgroup, a sequence $g_n$ is $Q$-divergent if
and only if the sequence $g_n^{-1}$ of inverses is $Q$-divergent.

A subgroup $\Gamma < G$ is $Q$-divergent if all sequences of pairwise
distinct elements in $\Gamma$ are $Q$-divergent. The \emph{$Q$-limit
  set} of a sequence or subgroup in $G$ is the set of all of its
$Q$-limit points. The $Q$-limit set of a group $\Gamma < G$ is a
$\Gamma$-invariant subset of $\flags$, and it is closed if $\Gamma$ is
$Q$-divergent.

If $G$ has rank one, a sequence $g_n$ is $Q$-divergent (for some,
equivalently any, parabolic $Q$) if and only if it is divergent in the
sense that it eventually leaves every compact subset of $G$. Thus in
the rank-one case a group $\Gamma < G$ is $Q$-divergent if and only if
it is discrete, and its $Q$-limit set is precisely the limit set in
the usual sense.

There are several notions equivalent to $Q$-divergence which will not
appear in this paper, but arise frequently in the literature; usually
these definitions are stated in terms of the behavior of a
\emph{Cartan projection} for $G$ \cite{GGKW2017} or the
\emph{vector-valued distance} on the Riemannian symmetric space $G/K$
\cite{KLP2017}. For the equivalence of the definitions, see
e.g. \cite[Lem. 4.23(i)]{KLP2017} and the appendix of
\cite{Weisman2022}.

We record one more property of $Q$-divergent sequences, which we will
use several times throughout the paper:
\begin{lemma}[{\cite[Lem. 4.19]{KLP2017}}]
  \label{lem:q_convergence_sequences}
  Let $g_n$ be a sequence in $G$, and let $\xi^\pm \in \flags$. The
  following are equivalent:
  \begin{enumerate}[label=(\arabic*)]
  \item The sequence $g_n$ converges to the constant map $\xi^+$ on
    $\Opp(\xi^-)$, uniformly on compact sets.
  \item The sequence $g_n$ is $Q$-divergent, the unique $Q$-limit
    point of $g_n$ is $\xi^+$, and the unique $Q$-limit point if
    $g_n^{-1}$ is $\xi^-$.
  \end{enumerate}
\end{lemma}

Note that every $Q$-divergent sequence in $G$ has a subsequence
satisfying the conditions of the lemma.

\subsection{Extended geometrically finite representations}
\label{sec:egf_definition}

We are now ready to give the main definition of this section. Below,
let $(\Gamma, \pargps)$ be a relatively hyperbolic pair, let $G$ be a
semisimple Lie group, and let $Q < G$ be a symmetric parabolic
subgroup.

\begin{definition}[{See \cite[Def. 1.3]{Weisman2022}}]
  \label{defn:egf_definition}
  A representation $\rho:\Gamma \to G$ is $Q$-\emph{extended
    geometrically finite} (with respect to the peripheral structure
  $\pargps$) if there exists a compact $\rho$-invariant subset
  $\Lambda \subset G/Q$ and a surjective antipodal map
  $\phi:\Lambda \to \bigamp$ extending the convergence action of
  $\Gamma$ on $\bigamp$ (\Cref{defn:extended_convergence_group}).
\end{definition}

Here, a map $\phi:\Lambda \to Z$ is \emph{antipodal} if, for any
distinct $z_1, z_2 \in Z$, every flag in $\phi^{-1}(z_1)$ is antipodal
to every flag in $\phi^{-1}(z_2)$, i.e. if
\[
  \phi^{-1}(z_1) \subseteq \bigcap_{\xi \in \phi^{-1}(z_2)} \Opp(\xi).
\]
When $G$ has rank one, this condition is vacuous.

As mentioned in \Cref{rem:extended_convergence}, it is not clear that
\Cref{defn:egf_definition} agrees exactly with the version originally
given in \cite{Weisman2022}, since it is possible that condition
\ref{item:repelling_strata_interior} in
\Cref{defn:extended_convergence_group} is actually stronger than
condition \ref{item:repelling_strata_containment}. Note, however, that
we are frequently in the situation where the repelling strata $C_z$
are all equal to the sets $\Opp(\phi^{-1}(z))$, and in this case both
\ref{item:repelling_strata_interior} and
\ref{item:repelling_strata_containment} follow automatically from the
antipodality assumption on $\phi$ and compactness of $\Lambda$.

Even when the above does not hold, we can still obtain
\ref{item:repelling_strata_interior} from
\ref{item:repelling_strata_containment} if we also assume that $\rho$
is \emph{relatively $Q$-divergent} in the following sense:

\begin{definition}
  A representation $\rho:\Gamma \to G$ is \emph{relatively
    $Q$-divergent} if, for every sequence $\gamma_n \in \Gamma$
  satisfying $\lcon{\gamma} \to \infty$, the sequence $\rho(\gamma_n)$
  is $Q$-divergent.
\end{definition}

The proof of Lemma 8.1 in \cite{Weisman2022} shows that any
representation which is $Q$-EGF in the sense of
\Cref{defn:egf_definition} is also relatively $Q$-divergent (see also
\Cref{rem:rel_q_div} in the present paper). In an appendix to this
paper, we prove:
\begin{proposition}
  \label{prop:weak_egf_to_egf}
  Let $\rho:\Gamma \to G$ be a representation, and suppose that there
  exists a compact $\rho$-invariant set $\Lambda \subset \flags$, an
  equivariant surjective antipodal map $\phi:\Lambda \to \bigamp$, and
  open sets $C_z \subset \flags$ for each $z \in \bigamp$ satisfying
  conditions \ref{item:repelling_strata_containment} and
  \ref{item:extended_convergence_prop} above.

  Then, if $\rho$ is relatively $Q$-divergent, there exists a
  (possibly different) equivariant surjective antipodal map
  $\hat{\phi}:\hat{\Lambda} \to \bigamp$ and open sets
  $\{\hat{C_z}\}_{z \in \bigamp}$ satisfying conditions
  \ref{item:repelling_strata_interior} and
  \ref{item:extended_convergence_prop} in
  \Cref{defn:extended_convergence_group}; in other words, $\rho$ is
  $Q$-EGF.
\end{proposition}

\subsection{The Chabauty topology and geometric convergence}
\label{sec:chabauty}

Below we give some basic reminders about the \emph{Chabauty topology}
on the space of closed subgroups of a Lie group $G$. We refer to
\cite{BHK} for further background. This material is required to
develop the notion of an \emph{extended Dehn filling space} for an EGF
representation $\rho:\Gamma \to G$.

\begin{definition}
  Let $G$ be a Lie group, and let $\chab = \chab(G)$ denote the set of
  closed subgroups of $G$. We equip $\chab$ with the \emph{Chabauty
    topology}, defined in terms of a basis as follows. For a compact
  subset $K \subset G$, a neighborhood $U$ of the identity in $G$, and
  $C \in \chab$, the basic open subset $V_{K, U, C} \subset \chab(G)$
  is
  \[
    V_{K,U,C} := \{D \in \chab : D \cap K \subset CU \textrm{ and } C
    \cap K \subset DU\}.
  \]
  Then the Chabauty topology on $\chab(G)$ is the topology generated
  by the basis of all open sets $V_{K,U,C}$. This makes $\chab(G)$
  into a compact space.
\end{definition}

Using some fixed metrization of $G$, the Chabauty topology on a Lie
group $G$ can also be viewed as the topology of \emph{local Hausdorff
  convergence} of closed subgroups in $G$: a sequence of closed
subgroups $C_n < G$ converges to some closed subgroup $C < G$ in
$\chab(G)$ if and only if, for every bounded open subset
$U \subset G$, the intersections $C_n \cap U$ converge to the
intersection $C \cap U$, with respect to Hausdorff distance.

We also have the following alternative characterization of convergence
in $\chab(G)$:
\begin{proposition}
  \label{prop:chab_converge}
  Let $G$ be a Lie group, let $C \in \chab(G)$, and let $C_n$ be a
  sequence in $\chab(G)$. Then $C_n$ converges to $C$ in $\chab(G)$ if
  and only if both of the following hold:
  \begin{enumerate}[label=(\Alph*)]
  \item\label{item:chab_contain_1} For every $g \in C$, there exists a
    sequence $g_n \in C_n$ so that $g_n \to g$.
  \item\label{item:chab_contain_2} If a sequence $g_n \in C_n$
    converges to some $g \in G$, then $g \in C$.
  \end{enumerate}
\end{proposition}
This is the description of the Chabauty topology we will typically use
in this paper.

\subsubsection{Geometric convergence}

Let $\Gamma$ be a finitely generated group and let $G$ be a Lie
group. Typically, the set $\Hom(\Gamma, G)$ is equipped with the
compact-open topology, where $\Gamma$ is viewed as an abstract
topological group with the discrete topology. Sometimes this is called
the topology of \emph{algebraic convergence}, since in this space a
representation $\rho$ is the limit of a sequence of representations
$\rho_n$ if and only if, for each $s$ in a finite generating set $S$
for $\Gamma$, we have $\rho_n(s) \to \rho(s)$. In this paper, unless
explicitly stated otherwise, ``convergence in $\Hom(\Gamma, G)$''
always means convergence with respect to this topology.

However, there is also a natural finer topology on $\Hom(\Gamma, G)$,
which we define below.
\begin{definition}
  Let $\Gamma$ be a finitely generated group and let $G$ be a Lie
  group. Then $\ghom(\Gamma, G)$ is the set $\Hom(\Gamma, G)$ equipped
  with the topology of \emph{geometric} or \emph{strong}
  convergence. Precisely, the topology on $\ghom(\Gamma, G)$ is
  defined to be the coarsest refinement of the topology on
  $\Hom(\Gamma, G)$ so that the map $\ghom(\Gamma, G) \to \chab(G)$
  given by $\rho \mapsto \overline{\rho(\Gamma)}$ is continuous.
\end{definition}

Equivalently, a sequence of representations $\rho_n$ converges to
$\rho$ in $\ghom(\Gamma, G)$ if and only if $\rho_n$ converges to
$\rho$ in $\Hom(\Gamma, G)$, and $\overline{\rho_n(\Gamma)}$ converges
to $\overline{\rho(\Gamma)}$ in $\chab(G)$. In this situation we say
that $\rho_n$ converges \emph{strongly} to $\rho$.

\subsection{Relative geometric convergence}

In practice, checking that a sequence of representations
$\rho_n:\Gamma \to G$ converges strongly to some $\rho$ can be
challenging, since one needs some form of uniform control on the
behavior of every element in every group $\rho_n(\Gamma)$. The process
is often much easier with some additional assumptions on
$\Gamma$---for instance if $\Gamma$ is assumed to be virtually
nilpotent or virtually abelian.

Now suppose that $\Gamma$ is a relatively hyperbolic group, relative
to some collection $\pargps$ of peripheral subgroups, and let
$\rho_n:\Gamma \to G$ be a sequence of representations. As above, if
the groups in $\pargps$ are e.g. virtually nilpotent, then it can be
much easier to directly control geometric convergence of the
\emph{restrictions} $\rho_n|_P$ for each $P \in \pargps$ than it is to
control Chabauty convergence of the full group $\rho_n(\Gamma)$; one
part of the main theorem in this paper essentially gives a method for
upgrading Chabauty convergence of the subgroups $\rho_n(P)$ to
Chabauty convergence of the full group $\rho_n(\Gamma)$.

To that end, we recall the following definition from the introduction
of the paper:
\begin{definition}
  Let $\Gamma$ be a finitely generated group, let $\pargps$ be a
  collection of finitely generated subgroups of $\Gamma$, and let $G$
  be a Lie group. We let $\ghom(\Gamma, G; \pargps)$ denote the space
  of representations $\Gamma \to G$ with the topology of
  \emph{relative geometric convergence}: the coarsest refinement of
  the compact-open topology on $\Hom(\Gamma, G)$ so that for every
  $P \in \pargps$, the restriction map
  $\ghom(\Gamma, G; \pargps) \to \ghom(P, G)$ is continuous.
\end{definition}

Thus, a sequence of representations $\rho_n:\Gamma \to G$ converges to
$\rho$ if and only if $\rho_n$ converges to $\rho$ in
$\Hom(\Gamma, G)$, and for each $P \in \pargps$,
$\overline{\rho_n(P)}$ converges to $\overline{\rho(P)}$ in
$\chab(G)$.

\subsection{Extended Dehn filling spaces}
\label{sec:edf}

With the notation and terminology developed above, we can give the
precise definition of an extended Dehn filling space. For the below,
fix a relatively hyperbolic pair $(\Gamma, \pargps)$, a symmetric
parabolic subgroup $Q$ of a semisimple Lie group $G$, and a $Q$-EGF
representation $\rho:\Gamma \to G$ with boundary extension
$\phi:\Lambda \to \bigamp$ and repelling strata
$\{C_z\}_{z \in \bigamp}$. Let $\Pi_\infty \subset \bigamp$ be the
finite set of points fixed by the groups $P \in \ipargps$, and for
each $p \in \Pi_\infty$ let $\Gamma_p \in \pargps$ be the
$\Gamma$-stabilizer of $p$.

\begin{definition}
  \label{defn:edf}
  A subspace $W \subseteq \ghom(\Gamma, G; \mc{P})$ is an
  \emph{extended Dehn filling space} if the following holds. For any
  $p \in \Pi_\infty$, any open set $U \subset \flags$ containing
  $\phi^{-1}(p)$, any finite subset $F \subseteq \Gamma_p$, and any
  compact set $K \subset C_p$ such that
  $\rho(\Gamma_p \minus F)K \subset U$, there exists an open
  neighborhood $O \subset \ghom(\Gamma, G; \pargps)$ such that for any
  $\sigma \in O \cap W$, we have
  \begin{align}
    \label{eq:edf_condition}
    (\sigma(\Gamma_p) \minus \sigma(F))K \subset U.
  \end{align}
\end{definition}

\begin{remark}
  If we replace \eqref{eq:edf_condition} above with the more
  restrictive condition
  \begin{equation}
    \label{eq:ps_condition}
    \sigma(\Gamma_p \minus F)K \subset U,
  \end{equation}
  we recover the definition of a \emph{peripherally stable subspace}
  given in \cite{Weisman2022}. Note that \eqref{eq:ps_condition}
  implies that for each $P \in \pargps$, the kernels of the
  restrictions $\sigma|_P$ are no larger than the kernels of
  $\rho|_P$, if $\sigma$ is sufficiently close to $\rho$ in
  $\Hom(\Gamma, G)$. Weakening from \eqref{eq:ps_condition} to
  \eqref{eq:edf_condition} allows the deformation $\sigma$ to
  potentially acquire a larger kernel on peripherals, allowing it to
  descend to a representation of a nontrivial Dehn filling of
  $\Gamma$.
\end{remark}

\section{Relative quasi-geodesic automata}
\label{sec:automaton}

In this section, we review a key tool we need for the proof of the
main theorem in this paper: a \emph{relative quasi-geodesic
  automaton}, originally developed in \cite{Weisman2022} in order to
prove a less general version of our main theorem. We recall the basic
definitions and main results concerning these automata here, and refer
to sections 5 and 6 of \cite{Weisman2022} for further details.

At the end of the section, we prove several results relating the
combinatorics of the automaton to the coarse geometry of the
corresponding relatively hyperbolic group; our main aim is to show
that there is a reasonably nice correspondence between the ``codings''
given by the automaton and quasi-geodesics (see
\Cref{cor:gpaths_track_quasigeodesics}). Similar results were proved
previously in \cite{Weisman2022} and \cite{MMW22,MMW24}, but we also
provide proofs in this paper since we sometimes need slightly
different statements.

For this section of the paper, fix a non-elementary relatively
hyperbolic pair $(\Gamma, \pargps)$. We assume that
$\pargps \ne \emptyset$. In this section we also assume that every
group in $\pargps$ is infinite, i.e. that $\pargps = \ipargps$. Let
$\Pi \subset \bgamp$ denote the finite set of fixed points of groups
in $P$.

\begin{definition}
  A \emph{$(\Gamma, \mc{P})$-graph} is a finite directed graph
  $\mc{G}$ with vertex set $V(\mc{G})$ and vertex labels
  $v \mapsto T_v$ for $T_v$ a subset of $\Gamma$, subject to the
  following condition: for every vertex $v \in V(\mc{G})$, either:
  \begin{enumerate}[label=(\roman*)]
  \item $T_v$ is a singleton $\{\alpha_v\}$, or
  \item $T_v$ is a set of the form $gP \minus F_v$, where
    $g \in \Gamma$, $P \in \pargps$ and $F_v \subset gP$ is finite.
  \end{enumerate}
  If $v$ is a vertex of the second kind above, so
  $T_v = gP \minus F_v$ for $P \in \pargps$, then we say the vertex
  $v$ is a \emph{parabolic vertex}. In this case (since $P$ is
  infinite) the coset $gP$ is uniquely determined by the label set
  $T_v$, so we let $q_v$ denote the \emph{associated parabolic point}
  $gp_v$, where $p_v \in \Pi$ is the unique point fixed by $P$.
\end{definition}

\begin{definition}
  Let $\mc{G}$ be a $(\Gamma, \mc{P})$-graph. A \emph{$\mc{G}$-path}
  is a sequence of pairs $\{(v_i, \alpha_i)\}_{i=1}^N$, with
  $N \in \N \cup \{\infty\}$, such that the sequence $\{v_i\}_{i=1}^N$
  is a vertex path in $\mc{G}$, and each $\alpha_i$ lies in the label
  set $T_{v_i}$.
\end{definition}

\begin{definition}
  \label{defn:limit_endpoint}
  If $z \in \bgamp$, we say that a $\mc{G}$-path
  $\{(v_i, \alpha_i)\}_{i=1}^N$ is a $\mc{G}$-path \emph{limiting to
    $z$} if one of the following holds:
  \begin{itemize}
  \item $z$ is a conical limit point, $N = \infty$, and the infinite
    sequence $\{g_n\}_{n=1}^\infty$ defined by
    \[
      g_n = \alpha_1 \cdots \alpha_n
    \]
    is unbounded and lies within a uniform neighborhood of a ray
    $r:[0, \infty) \to X(\Gamma, \pargps)$ with ideal endpoint $z$.
  \item $z$ is a parabolic point, $N$ is finite, $v_N$ is a parabolic
    vertex corresponding to a parabolic point $q_v$, and
    \[
      z = \alpha_1 \cdots \alpha_{N-1}q_v.
    \]
  \end{itemize}
\end{definition}

In this paper, we use a $(\Gamma, \pargps)$-graph $\mc{G}$ to encode
``convergence-like'' actions of $\Gamma$ on certain compact metrizable
spaces $M$; in particular we want to encode \emph{extended convergence
  actions} as in \Cref{sec:extended_convergence}. We express the
connection between $\Gamma$-actions and $(\Gamma, \pargps)$-graphs
with the definition below.

\begin{definition}
  Let $\mc{G}$ be a $(\Gamma, \mc{P})$-graph, and suppose that
  $\Gamma$ acts by homeomorphisms on a compact metric space $M$. A
  \emph{$\mc{G}$-compatible system of subsets} of $M$ is an assignment
  $v \mapsto U_v$ of vertices of $\mc{G}$ to open subsets of $M$ such
  that, for every edge $v \to w$ in $\mc{G}$, there is some $\eps > 0$
  so that $\overline{\nbhd{U_w}{\eps}} \ne M$, and
  \begin{equation*}
    \label{eq:g_compatibility}
    \alpha \cdot \overline{\nbhd{U_w}{\eps}} \subset U_v
  \end{equation*}
  for every $\alpha \in T_v$.
\end{definition}

\subsection{Automata for EGF representations}

One of the main results of \cite{Weisman2022} is that, given a $Q$-EGF
representation $\rho:\Gamma \to G$, it is always possible to construct
a $(\Gamma, \pargps)$-graph $\mc{G}$ and $\mc{G}$-compatible systems
of subsets of both $\bgamp$ and $\flags$, which satisfy certain
desirable properties. Then, the finiteness properties of $\mc{G}$ then
can be used to deduce local stability properties of the representation
$\rho$.

The key construction in \cite{Weisman2022} can be stated as
follows. Recall that we have fixed a relatively hyperbolic pair
$(\Gamma, \pargps)$ with $\pargps = \ipargps$. Below, and for the rest
of the paper, we fix an auxiliary metric on the flag manifold
$\flags$; neighborhoods of points and sets in $\flags$ will be with
respect to this metric, while neighborhoods in $\bgamp$ are taken with
respect to a fixed choice of metric $\bdist$.
\begin{prop}[{\cite[Proposition 6.1]{Weisman2022}}]
  \label{prop:automaton_exists}
  Let $\rho:\Gamma \to G$ be a $Q$-EGF representation of $\Gamma$ with
  boundary extension $\phi:\Lambda \to \bgamp$ and repelling strata
  $\{C_z : z \in \bgamp\}$.

  Then, for every $\eps > 0$, there is a $(\Gamma, \mc{P})$-graph
  $\mc{G}$ and a pair of $\mc{G}$-compatible systems
  $\{U_v \subset \flags\}_{v \in V(\mc{G})}$,
  $\{W_v \subset \bgamp\}_{v \in V(\mc{G})}$ satisfying the following
  conditions:
  \begin{enumerate}[label=(G\arabic*),series=automaton]
  \item\label{item:g_paths_exist} For every $z \in \bgamp$, there is a
    $\mc{G}$-path limiting to $z$.
  \item \label{item:g_sets_small} For every vertex $v \in V(\mc{G})$,
    there is some $z \in W_v$ so that
    \[
      \phi^{-1}(W_v) \subset U_v \subset \nbhd{\phi^{-1}(z)}{\eps}.
    \]
  \item \label{item:outgoing_edge} Every vertex of $\mc{G}$ has at
    least one outgoing edge.
  \item \label{item:parabolic_edge_condition1} For every $p \in \Pi$,
    there is a parabolic vertex $v$ with $q_v = p$. Moreover, for
    every parabolic vertex $w$ with $q_w = gp$, there is an edge
    $v \to b$ in $\mc{G}$ if and only if there is also an edge
    $w \to b$.
  \item \label{item:parabolic_edge_condition2} If $v$ is a parabolic
    vertex with $q_v = gp$ for $p \in \Pi$, and $(v, w)$ is an edge of
    $\mc{G}$, then $q_v \in W_v$ and $U_w \subset C_p$.
  \end{enumerate}
\end{prop}

\begin{remark}\
  \begin{enumerate}[label=(\alph*)]
  \item In \cite{Weisman2022}, this proposition was originally stated
    for representations satisfying the weaker condition
    \ref{item:repelling_strata_containment} discussed in the previous
    section. However, the proof actually relies on the stronger
    condition \ref{item:repelling_strata_interior} in
    \Cref{defn:extended_convergence_group}.
  \item Property \ref{item:outgoing_edge} is not explicitly stated as
    part of Proposition 6.1 in \cite{Weisman2022}, but it is a
    consequence of the proof (see the discussion above Proposition 6.9
    in \cite{Weisman2022}).
  \end{enumerate}
\end{remark}

For the rest of the paper, we will refer to any
$(\Gamma, \mc{P})$-graph as in the proposition above as a
\emph{relative automaton} adapted to the EGF representation
$\rho$. The precise construction yielding the automaton and
$\mc{G}$-compatible systems given by \Cref{prop:automaton_exists} is
rather technical, so we will not discuss any details here. However, we
will take advantage of a few additional properties of the automaton,
which follow from either the properties above or arguments in
\cite{Weisman2022}.

The first two properties are straightforward consequence of compactness of
the sets $\Lambda, \bgamp$.
\begin{proposition}
  \label{prop:vertex_sets_proper}
  If the parameter $\eps > 0$ in \Cref{prop:automaton_exists} is
  sufficiently small, then:
  \begin{enumerate}[resume=automaton,label=(G\arabic*)]
  \item\label{item:bgamp_sets_proper} For each set $W_v$ in the system
    $\{W_v\}_{v \in V(\mc{G})}$, the set
    $\bgamp \minus \overline{W_v}$ is nonempty.
  \item\label{item:flag_sets_proper} For each set $U_v$ in the system
    $\{U_v\}_{v \in V(\mc{G})}$, the set
    $\Opp(\overline{U_v}) = \bigcap_{\xi \in \overline{U_v}}
    \Opp(\xi)$ is nonempty.
  \end{enumerate}  
\end{proposition}
\begin{proof}
  Note that for every compact subset $K \subset \bgamp$, the set
  $\phi^{-1}(K)$ is a compact subset of flags satisfying
  $\phi^{-1}(\bgamp \minus K) \subset \Opp(\phi^{-1}(K))$. In
  particular, if there is some $z \in \bgamp \minus K$, then
  $\phi^{-1}(z) \in \Opp(\phi^{-1}(K))$. Then, since antipodality is
  an open condition, for each proper compact subset
  $K \subsetneq \bgamp$, we can find some $\eps > 0$ and some
  $z \in \bgamp$ so that $\Opp(\nbhd{\phi^{-1}(K)}{\eps})$ contains
  $\phi^{-1}(z)$.

  We can cover $\bgamp$ with finitely many compact proper subsets, and
  then use the above to find a uniform $\eps > 0$ so that for each
  $z \in \bgamp$, there is some $z' \in \bgamp$ so that
  $\phi^{-1}(z')$ is contained in
  $\Opp(\nbhd{\phi^{-1}(z)}{\eps})$. Then both
  \ref{item:bgamp_sets_proper} and \ref{item:flag_sets_proper} above
  follow from \ref{item:g_sets_small}.
\end{proof}

The next property says that we can choose the $\mc{G}$-compatible
system of sets $\{U_v\}_{v \in V(\mc{G})}$ so that, if a pair of
points $z_1, z_2 \in \bgamp$ are well-separated, then the lifts of
$z_1$ and $z_2$ in $\flags$ lie in uniformly antipodal neighborhoods
$U_{v_1}, U_{v_2}$.
\begin{prop}[{See \cite[Proposition 9.6]{Weisman2022}}]
  \label{prop:distinguished}
  For any given $\Delta > 0$, the automaton and sets in
  \Cref{prop:automaton_exists} can also be chosen to satisfy:
  \begin{enumerate}[resume=automaton,label=(G\arabic*)]
  \item \label{item:distinguish_points} If $z_1, z_2 \in \bgamp$
    satisfy $\bdist(z_1, z_2) > \Delta$, and
    $\phi^{-1}(z_1) \subset U_{v_1}$,
    $z_2 \in \phi^{-1}(z_2) \subset U_{v_2}$, then
    $\overline{U_{v_1}}$ and $\overline{U_{v_2}}$ are pairwise
    antipodal, i.e.
    \[
      \overline{U_{v_i}} \subseteq \bigcap_{\xi \in
        \overline{U_{v_j}}} \Opp(\xi).
    \]
    for $\{i,j\} = \{1,2\}$.
  \end{enumerate}
\end{prop}

The following property says that if a point $z$ in $\bgamp$ is ``far
from'' some parabolic point $p \in \Pi$, then there is a $\mc{G}$-path
limiting to $z$ which has a certain special form. This property is
useful for showing that the action of $\rho(\Gamma_p)$ on
$\phi^{-1}(z)$ is stable under certain perturbations of $\rho$, where
$\Gamma_p$ is the stabilizer of $p$.
\begin{prop}[{\cite[Proposition 6.15]{Weisman2022}}]
  \label{prop:parabolic_approx_path}
  For each parabolic point $p \in \Pi$, let $K_p$ be a compact subset
  of $\bgamp \minus \{p\}$. The automaton and sets in
  \Cref{prop:automaton_exists} can further be chosen to satisfy:
  \begin{enumerate}[resume=automaton,label=(G\arabic*)]
  \item \label{item:parabolic_domain_path} For every parabolic vertex
    $w$ with $p_w = g \cdot p$ for $p \in \Pi$, and every $z \in K_p$,
    there is a $\mc{G}$-path $\{(v_i, \alpha_i)\}_{i=1}^N$ limiting to
    $z$, whose first vertex $v_1$ is connected to $w$ by an edge
    $(w, v_1)$ in $\mc{G}$.
  \end{enumerate}
\end{prop}

The final property says the automaton can be chosen so that, for any
compact subset $Z \subseteq \bgamp$, the initial open subsets
appearing in $\mc{G}$-paths limiting to points in $Z$ can be chosen to
``approximate'' $\phi^{-1}(Z)$ arbitrarily well.

\begin{prop}[{\cite[Prop. 6.14]{Weisman2022}}]
  Let $Z$ be a compact subset of $\bgamp$, and let $\eps > 0$ be
  given. The automaton and sets in \Cref{prop:automaton_exists} can
  also be chosen to satisfy:
  \begin{enumerate}[resume=automaton,label=(G\arabic*)]
  \item \label{item:approximate_compact_subset} For every $z \in Z$,
    there exists a $\mc{G}$-path limiting to $z$, with initial vertex
    $v_1$, such that $U_{v_1} \subset \nbhd{\phi^{-1}(Z)}{\eps}$.
  \end{enumerate}
\end{prop}

\subsection{$\mc{G}$-paths and the geometry of $X$}

For the rest of the section, suppose that $\rho:\Gamma \to G$ is a
$Q$-EGF representation, and that $\mc{G}$ is a relative automaton with
$\mc{G}$-compatible systems of sets
$\mc{U} = \{U_v \subset G/Q\}_{v \in V(\mc{G})}$ and
$\mc{W} = \{W_v \subset \bgamp\}_{v \in V(\mc{G})}$ satisfying
properties \ref{item:g_paths_exist}-\ref{item:flag_sets_proper}
above. (Since properties
\ref{item:distinguish_points}-\ref{item:approximate_compact_subset}
depend on additional extrinsic parameters, we will ignore them until
the next section.)

It turns out that the relative automaton $\mc{G}$ and the system
$\mc{W}$ are also well-adapted to the metric geometry of the
relatively hyperbolic group $\Gamma$, in the sense that $\mc{G}$-paths
determine sequences in $\Gamma$ which are closely related to geodesic
rays and segments in the cusped space $X(\Gamma, \mc{P})$. The next
several lemmas make this precise. We note any analogies with similar
results in \cite{MMW24} or \cite{Weisman2022} where they occur.

First, we make the following observation:
\begin{lem}
  \label{lem:bounded_backtrack_I}
  There is a constant $k > 0$ satisfying the following: for any
  $\mc{G}$-path $\{(v_i, \alpha_i)\}_{i=1}^N$, if $g_n$ is the
  sequence $g_n = \alpha_1 \cdots \alpha_n$, then no element in
  $\Gamma$ can appear more than $k$ times in $g_n$.
\end{lem}
\begin{proof}
  We know by \ref{item:flag_sets_proper} that each $\overline{U_v}$ is
  a proper nonempty subset of $\flags$. Then, since the flag manifold
  $\flags$ is connected, and
  $\overline{\alpha_iU_{v_{i+1}}} \subset U_{v_i}$ for all $i$, we
  must have $\alpha_iU_{v_{i+1}} \subsetneq U_{v_i}$ for all $i$. It
  follows that for any $m < n$, we have
  \begin{equation}
    \label{eq:proper_inclusion}
    g_nU_{v_{n+1}} \subsetneq g_mU_{v_{m+1}}.
  \end{equation}
  Now we can prove the lemma by arguing as in
  \cite[Lem. 4.10]{MMW24}. Let $k$ be the number of vertices in the
  automaton $\mc{G}$. If there is some $g$ such that $g_n = g$ for
  more than $k$ different indices $n$, then there are indices $m, n$
  with $v_{m+1} = v_{n+1}$ and $g_m = g_n$, but this contradicts
  \eqref{eq:proper_inclusion}, proving the claim.
\end{proof}

\begin{lemma}[{See \cite[Prop 5.11]{Weisman2022} or \cite[Lem
    4.7]{MMW24}}]
  \label{lem:gpaths_geodesic_neighborhood}
  There exists a constant $D > 0$ such that, for any (finite or
  infinite) $\mc{G}$-path $\{(v_i, \alpha_i)\}_{i=1}^N$, any point
  $z \in \bigcap_{i=1}^N \alpha_1 \cdots \alpha_{i-1} W_{v_{i+1}}$,
  and any geodesic ray $c:[0, \infty) \to X = X(\Gamma, \pargps)$ from
  the identity to $z$, the set of points
  \[
    \{g_n = \alpha_1 \cdots \alpha_n\}_{n=0}^N
  \]
  lies within a $D$-neighborhood of a geodesic ray in $X$ from
  $\identity$ to $z$, with respect to the metric $\dcusp$.
\end{lemma}
\begin{proof}
  First, fix $r > 0$ so that for every edge $u \to v$ in $\mc{G}$, and
  every $\alpha \in T_u$, the inclusion
  \begin{equation}
    \label{eq:edge_inclusion_lem}
    \alpha \cdot \nbhd{W_v}{r} \subset W_u.
  \end{equation}
  By property \ref{item:bgamp_sets_proper}, we may also assume that
  $r$ is sufficiently small so that $\nbhd{W_v}{r} \ne \bgamp$ for
  every $W_v \in \mc{W}$.

  The vertex sequence $v_1, v_2, \ldots,$ of the $\mc{G}$-path
  determines a sequence of sets $W_i := W_{v_i}$. Fix a point
  $z_- \in \bgamp \minus \nbhd{W_1}{r}$. Since
  $z \in \overline{\alpha_1W_2} \subset W_1$, we know that, with
  respect to our fixed metric $\bdist$ on $\bgamp$, we have
  $\bdist(z, z_-) \ge \eps$. Therefore, there is a bi-infinite
  geodesic $c:(-\infty, \infty) \to X$ passing within distance $R$ of
  the identity, for $R$ depending only on $r$ and the metric $\bdist$.

  For each $n$, we know that $g_n^{-1}c$ is a bi-infinite geodesic
  joining $g_n^{-1}z_-$ to $g_n^{-1}z$. By assumption, $g_n^{-1}z$
  lies in $W_{n+1}$, and \eqref{eq:edge_inclusion_lem} guarantees that
  $\nbhd{W_{n+1}}{r} \subset g_n^{-1}W_1$, hence
  \[
    \bgamp \minus g_n^{-1}W_1 \subset \bgamp \minus
    g_n^{-1}\nbhd{W_{n+1}}{r}.
  \]
  It follows that $\bdist(g_n^{-1}z_-, g_n^{-1}z) > r$ and so
  $g_n^{-1}c$ also passes within distance $R$ of the
  identity. Therefore, the sequence $g_n$ lies within an
  $R$-neighborhood of $c$.

  The argument above actually shows that $g_n$ lies within distance
  $R$ of \emph{any} bi-infinite geodesic joining $z$ to some point
  $z_-$ in $\bgamp \minus \nbhd{W_1}{r}$. Since $\bgamp$ is perfect
  and compact, we can choose $r$ small enough so that for every vertex
  $v$ in the automaton, the complement $\bgamp \minus \nbhd{W_v}{r}$
  contains at least two points whose distance is at least $r$. So, by
  choosing two points $z_-, z_-'$ in $\nbhd{W_1}{r}$ with
  $\bdist(z_-, z_-') > r$, we see that $g_n$ lies in the intersection
  of $R$-neighborhoods of bi-infinite geodesics $c, c'$ joining
  $z_-, z_-'$ respectively to $z$. By e.g. \cite[Lemma 4.6]{MMW24},
  this means that $g_n$ also lies in a uniform neighborhood (in $X$)
  of a geodesic ray based at the identity with ideal endpoint $z$.
\end{proof}

As a consequence, we have:
\begin{lem}[{see \cite[Cor. 4.11]{MMW24}}]
  \label{lem:infinite_codes_conical}
  Any infinite-length $\mc{G}$-path $\{(v_i, \alpha_i)\}_{i=1}^\infty$
  limits to a unique conical limit point $z$. Moreover, this point is
  the unique element contained in the infinite intersection
  \[
    \bigcap_{i=1}^\infty \alpha_1 \cdots \alpha_n W(v_{n+1}).
  \]
\end{lem}
\begin{proof}
  For each $n$, let $g_n = \alpha_1 \cdots \alpha_n$. Since $\bgamp$
  is compact, there is at least one point $z$ in the infinite
  intersection
  $\bigcap_{i=1}^\infty \alpha_1 \cdots \alpha_n W(v_{n+1})$. It
  follows from \Cref{lem:gpaths_geodesic_neighborhood} that $g_n$ lies
  within a uniform neighborhood of some ray in $X$ whose ideal
  endpoint is $z$. By \Cref{lem:bounded_backtrack_I}, the sequence
  $g_n$ is unbounded in $\Gamma$, so $z$ must be a conical limit point
  (see e.g. \cite[Prop. A.2]{HealyHruska10}), and $z$ is uniquely
  determined by $g_n$.
\end{proof}

Next we want to prove a more specific version of
\Cref{lem:gpaths_geodesic_neighborhood}, but we need an intermediate
result first.

\begin{lem}
  \label{lem:coding_segments_in_horoballs}
  For any vertex $v$ in $\mc{G}$, there exists a constant $C > 0$ so
  that, if, $\alpha$ is any element in $T_v$, and $\ell$ is a geodesic
  segment in $X = X(\Gamma, \pargps)$ between $\identity$ and
  $\alpha$, then for any $R > 0$, the Hausdorff distance (with respect
  to $\dcusp$) between $\nbhd[\dcusp]{\ell}{R} \cap \cay(\Gamma)$ and
  $\{\identity, \alpha\}$ is at most $C + 2R$.
\end{lem}
For the proof below, all distances are taken with respect to the
metric $\dcusp$.
\begin{proof}
  If $v$ is not a parabolic vertex, then $T_v$ is a singleton
  $\{\alpha_v\}$ and we can just take $C = |\alpha_v|_X$. So, suppose
  that $v$ is parabolic. Then the associated parabolic point $q_v$ has
  the form $q_v = gp$ for some $g \in \Gamma$ and $p \in \Pi$. Let $P$
  be the stabilizer in $\Gamma$ of $p$ and let $\mc{H}$ be the
  combinatorial horoball in $X$ based at $gP$, so that any
  $\alpha \in T_v$ is in $\mc{H}$. Thus, if $\ell$ is a geodesic
  segment from $\identity$ to $\alpha$, the endpoints of $\ell$ lie
  within distance $|g|_X$ of $\mc{H}$.

  Now, if $u \in \nbhd[\dcusp]{\ell}{R} \cap \cay(\Gamma)$, then there
  is some $w \in \ell$ with $\dcusp(w, u) \le R$, hence
  $\dcusp(w, \cay(\Gamma)) \le R$. So by
  \Cref{lem:geodesics_enter_horoballs_deeply}, we have
  $\dcusp(w, \{\identity, \alpha\}) \le R + 3|g|_X + 7\delta$, and
  therefore
  $\dcusp(u, \{\identity, \alpha\}) \le 2R + 3|g|_X + 7\delta$.
\end{proof}

Now we strengthen \Cref{lem:gpaths_geodesic_neighborhood}. Compare the
proof of the below to the proof of \cite[Lem. 4.13]{MMW24}.
\begin{lem}
  \label{lem:gpaths_track_geodesics}
  There exists a constant $D > 0$ satisfying the following. Suppose
  that $\{(v_i, \alpha_i)\}_{i=1}^N$ is a (finite or infinite)
  $\mc{G}$-path, let $g_n = \alpha_1 \cdots \alpha_n$, and let
  $z \in \bigcap_{n=1}^Ng_{n-1}W_{v_{n}}$. If $c:[0, \infty) \to X$ is
  a geodesic ray from $\identity$ to $z$, then there is a subinterval
  $I \subseteq [0, \infty)$ containing $0$ such that the set
  $\{g_n\}_{n=0}^{N-1}$ lies within Hausdorff distance $D$ of
  $c(I) \cap \cay(\Gamma)$, with respect to $\dcusp$.
\end{lem}
\begin{proof}
  We know from \Cref{lem:gpaths_geodesic_neighborhood} that there is a
  constant $D_0$ so that the set $\mc{S} = \{g_n\}_{n=0}^{N-1}$ is
  contained in a $D_0$-neighborhood of $c([0, \infty))$, so let
  $I \subseteq [0, \infty)$ be a minimal interval containing $0$ such
  that a $D_0$-neighborhood of $c(I)$ also contains $\mc{S}$. That is,
  $\nbhd[\dcusp]{c(I)}{D_0} \cap \cay(\Gamma)$ contains $\mc{S}$, so
  because of \Cref{lem:cayley_hausdorff_neighborhoods} we just need to
  check that $c(I) \cap \cay(\Gamma)$ lies in a uniform neighborhood
  of $\mc{S}$.

  For each $1 \le n < N - 1$, let $\ell_n$ be a geodesic segment in
  $X$ joining $g_{n-1}$ to $g_n$, and let $Y$ be the union
  $\bigcup_{n=1}^{N-1} \ell_n$. Each $\ell_n$ is contained in a
  $(D_0 + 2\delta)$-neighborhood of $c(I)$, meaning that $Y$ is a
  connected subset also lying in a $(D_0 + 2\delta)$-neighborhood of
  $c(I)$.

  Let $u$ be a point in $c(I) \cap \cay(\Gamma)$, and let
  $R = D_0 + 2\delta$. Now, if $u$ lies within distance $2R$ of an
  endpoint of $c(I)$, then by minimality of $I$, $u$ must also be
  within distance $2R + D_0$ of some $g_n$ and we are done. So, we may
  assume the metric ball $\ball[\dcusp]{u}{2R}$ separates $c(I)$ into
  two nonempty connected components $B_-, B_+$. Since $c$ is a
  geodesic, we have
  \begin{equation}
    \label{eq:component_distance}
    \dcusp(B_-, B_+) \ge 4R.
  \end{equation}

  Suppose for a contradiction that $Y \cap \ball[\dcusp]{u}{3R}$ is
  empty. Then for any $w \in\ball[\dcusp]{u}{2R}$ and any $y \in Y$,
  we have $\dcusp(y, w) > 3R - 2R = R$. Thus, since $Y$ is in an
  $R$-neighborhood of $B_- \cup \ball[\dcusp]{u}{2R} \cup B_+$, for
  any $y \in Y$, there is some $w \in B_- \cup B_+$ with
  $\dcusp(y, w) \le R$. That is, for any $y \in Y$, we have
  \[
    \dcusp(y, c(I)) = \min(\dcusp(y, B_-), \dcusp(y, B_+)) \le R.
  \]
  Since $Y$ contains $\identity$, and $I$ was chosen minimally, there
  are points $y_-, y_+ \in Y$ so that $\dcusp(y_-, B_-) \le R$ and
  $\dcusp(y_+, B_+) \le R$. By \eqref{eq:component_distance}, we have
  $\dcusp(y_-, B_+) \ge 3R$ and $\dcusp(y_+, B_-) \ge 3R$. Then since
  $Y$ is connected, continuity of the distance function implies that
  there is some point $y \in Y$ so that
  $\dcusp(y, B_-) = \dcusp(y, B_+)$, hence $\dcusp(y, B_\pm) \le R$
  and $\dcusp(B_-, B_+) \le 2R$. But this contradicts
  \eqref{eq:component_distance}.

  We conclude that there is some $y \in Y$ so that
  $\dcusp(y, u) \le 3R$. Since $y \in \ell_n$ for some $n$, and
  $u \in \cay(\Gamma)$ by assumption, we may apply
  \Cref{lem:coding_segments_in_horoballs} to see that there is a
  uniform constant $C$ such that
  $\dcusp(u, \{g_{n-1}, g_n\}) \le C + 6R$. Thus
  $c(I) \cap \cay(\Gamma)$ lies in a $(C + 6R)$-neighborhood of
  $\mc{S}$, completing the proof.
\end{proof}

We can further strengthen the previous result so that it also applies
to \emph{uniform quasi-geodesic} rays. For this result and the next,
no exact analogy appears in either \cite{MMW24} or \cite{Weisman2022}.
\begin{corollary}
  \label{cor:gpaths_track_quasigeodesics}
  For any $K \ge 1, A \ge 0$, there exists a constant $D > 0$
  satisfying the following. For any $z \in \bgamp$, any (finite or
  infinite) $\mc{G}$-path $\{(v_i, \alpha_i)\}_{i=1}^N$ limiting to
  $z$, and any $(K,A)$-quasi-geodesic ray $r:[0, \infty) \to X$ from
  the identity to $z$, the set
  $\{g_n = \alpha_1 \cdots \alpha_n\}_{n=0}^{N-1}$ is within Hausdorff
  distance $D$ of $r([0, \infty)) \cap \cay(\Gamma)$, with respect to
  $\dcusp$.
\end{corollary}
\begin{proof}
  We first claim that, if the $\mc{G}$-path
  $\{(v_i, \alpha_i)\}_{i=1}^N$ limits to $z \in \bgamp$, then $z$
  lies in the intersection of sets
  \[
    \bigcap_{n=1}^N g_{n-1}W_{v_n}.
  \]
  If $N = \infty$ this follows from
  \Cref{lem:infinite_codes_conical}. If $N$ is finite, then $v_N$ is a
  parabolic vertex, and by condition
  \ref{item:parabolic_edge_condition2}, we know that the corresponding
  parabolic point $q = q_{v_N}$ lies in $W_{v_N}$. Since the given
  $\mc{G}$-path limits to $z$, by definition we have
  $z = \alpha_1 \cdots \alpha_{N-1}q$ and therefore
  $z \in \alpha_1 \cdots \alpha_{N-1}W_N$.

  Now fix a geodesic ray $c:[0, \infty) \to X$ from the identity to
  $z$. By \Cref{cor:hausdorff_distance_bound} and the Morse lemma, it
  suffices to show that the Hausdorff distance between the sets
  $\mc{S} = \{g_n\}_{n=0}^{N-1}$ and
  $c([0, \infty)) \cap \cay(\Gamma)$ is uniformly bounded. By
  \Cref{lem:gpaths_track_geodesics}, there is a constant $D_0$ and a
  sub-interval $I \subseteq [0, \infty)$ containing $0$ so that
  $\mc{S}$ lies within Hausdorff distance $D_0$ of
  $c(I) \cap \cay(\Gamma)$, so we will be done if we can show that the
  Hausdorff distance between $c(I) \cap \cay(\Gamma)$ and
  $c([0, \infty)) \cap \cay(\Gamma)$ is uniformly bounded.

  If $I = [0, \infty)$ then there is nothing to prove, so assume that
  $I = [0, t]$ for some $t < \infty$. This can only occur if the
  sequence $g_n$ is bounded in $\cay(\Gamma)$, meaning (by
  \Cref{lem:bounded_backtrack_I}) that $N$ is finite and $z$ is a
  parabolic point in the $\Gamma$-orbit of some $p \in \Pi$. Let $P$
  be the stabilizer of $p$ in $\Gamma$. Since our $\mc{G}$-path
  $\{(v_i, \alpha_i)\}_{i=1}^N$ limits to $z$ by assumption, $z$ must
  be the center of the combinatorial horoball $\mc{H}$ based on the
  coset $g_{N-1}P$.

  Fix $t' \le t$ so that $\dcusp(g_{N-1}, c(t')) < D_0$, and suppose
  that $c(s) \in \cay(\Gamma)$ for some $s > t$. Since $c$ is
  asymptotic to $z$, there is some $t''$, much larger than $s$, so
  that $c(t'') \in \mc{H}$ and
  $\dcusp(c(s), c(t')) \le \dcusp(c(s), c(t''))$. Then the sub-segment
  of $c$ between $c(t')$ and $c(t'')$ has endpoints within distance
  $D_0$ of $\mc{H}$, so by
  \Cref{lem:geodesics_enter_horoballs_deeply}, we have
  \[
    \dcusp(c(s), c(t')) \le 3D_0 + 7\delta.
  \]
  This shows that any point in $c((t, \infty)) \cap \cay(\Gamma)$ lies
  within distance $3D_0 + 7\delta$ of $c(I)$, meaning that
  $\nbhd[\dcusp]{c(I)}{3D_0 + 7\delta} \cap \cay(\Gamma)$ contains
  $c([0, \infty)) \cap \cay(\Gamma)$. Thus by
  \Cref{lem:cayley_hausdorff_neighborhoods} there is a uniform
  neighborhood of $c(I) \cap \cay(\Gamma)$ containing
  $c([0, \infty)) \cap \cay(\Gamma)$, and therefore a uniform bound on
  the Hausdorff distance between $c([0, \infty)) \cap \cay(\Gamma)$
  and $c(I) \cap \cay(\Gamma)$, as desired.
\end{proof}

\subsection{$\mc{G}$-paths close to $\cay(\Gamma)$}

The final lemma in this section says that, when the elements
$\alpha_i$ appearing in a $\mc{G}$-path are all short, then any
geodesic in $X$ ``tracking'' the $\mc{G}$-path does not enter deeply
into any horoball.
\begin{lemma}
  \label{lem:short_jumps_close_to_cay}
  Let $\{(v_i, \alpha_i)\}_{i=1}^\infty$ be an infinite $\mc{G}$-path,
  and assume that the quantity $C = \max_{i\ge 1}\{\lcusp{\alpha_i}\}$
  is finite. Let $c:[0, \infty) \to X$ be a geodesic ray based at the
  identity, and suppose that the $\dcusp$-Hausdorff distance between
  $c([0, \infty)) \cap \cay(\Gamma)$ and the set
  $\{g_n = \alpha_1 \cdots \alpha_n\}$ is at most $D$. Then
  $c([0, \infty))$ lies in a $(C + 3D)$-neighborhood of
  $\cay(\Gamma)$, with respect to $\dcusp$.
\end{lemma}
\begin{proof}
  Since the $\mc{G}$-path is infinite, it follows from
  \Cref{lem:infinite_codes_conical} that the diameter of
  $c([0, \infty)) \cap \cay(\Gamma)$ is infinite. So, it suffices to
  show that for any $x \in c([0, \infty)) \cap \cay(\Gamma)$, the
  sub-segment of $c$ between $\identity$ and $x$ lies in a
  $(C + 3D)$-neighborhood of $\cay(\Gamma)$.

  Fix $g_n$ so that $\dcusp(x, g_n) \le D$. Then for each index
  $1 \le j \le n$, there is some point $x_j$ on $c$ so that
  $\dcusp(x_j, g_j) \le D$. For every $1 \le j < n$, we have
  $\dcusp(g_j, g_{j+1}) \le C$ by assumption, meaning
  $\dcusp(x_j, x_{j+1}) \le C + 2D$. Thus, the ball of radius $C + 3D$
  centered at $g_j$ contains the entire sub-segment of $c$ between
  $x_j$ and $x_{j+1}$. The union of all of these segments contains the
  segment of $c$ between $\identity$ and $x$, so this segment is in
  the $(C + 3D)$-neighborhood of $\cay(\Gamma)$, as required.
\end{proof}

\section{Extended Dehn filling}
\label{sec:main_theorem}

The aim of this section of the paper is to prove our main Dehn filling
theorem, \Cref{thm:main_theorem}. At the same time, we will prove that
the Dehn fillings given by \Cref{thm:main_theorem} satisfy several
additional desirable properties; in particular, we will show that the
boundary sets of Dehn fillings vary (semi)continuously, and that, with
a mild additional assumption on boundary extensions, convergence to an
EGF representation $\rho$ in an extended Dehn filling space is
\emph{strong}.

For the rest of the section, fix a non-elementary relatively
hyperbolic pair $(\Gamma, \pargps)$, with $\mc{P}$ a non-empty
collection peripheral subgroups. We will assume that each group in
$\pargps$ is infinite, but later in the section we will consider other
relatively hyperbolic pairs where the peripheral subgroups are
\emph{not} necessarily infinite.

Let $Q$ be a symmetric parabolic subgroup in a semisimple Lie group
$G$, and let $\rho:\Gamma \to G$ be a $Q$-EGF representation with
boundary extension $\phi:\Lambda \to \bgamp$ and repelling strata
$\{C_z\}_{z \in \bgamp}$. For simplicity we will assume that $\rho$ is
faithful.

We will prove \Cref{prop:main_prop} below, which is a strengthening of
\Cref{thm:main_theorem}. For the statement of the proposition, recall
from the introduction that any representation $\sigma:\Gamma \to G$
determines a collection of (possibly trivial) Dehn filling kernels
$\{\ker(\sigma|_P)\}_{P \in \pargps}$. There is an associated Dehn
filling $\pi^\sigma:\Gamma \to \Gamma^\sigma$ and collection of groups
$\mc{P}^\sigma = \{P/\ker(\sigma|_P)\}_{P \in \pargps}$.
\begin{proposition}
  \label{prop:main_prop}
  Let $W \subseteq \ghom(\Gamma, G; \pargps)$ be an extended Dehn
  filling space, and fix a constant $\eps > 0$ and a compact subset $Z
  \subseteq \bgamp$.

  Then there exists an open subset
  $O \subseteq \ghom(\Gamma, G; \pargps)$ containing $\rho$ so that,
  for every representation $\sigma \in O \cap W$, the pair
  $(\Gamma^\sigma, \pargps^\sigma)$ is relatively hyperbolic, and
  $\sigma$ descends to a $Q$-EGF representation of
  $\Gamma^\sigma$. Moreover, the boundary set $\Lambda_\sigma$ for
  $\sigma$ is contained in an $\eps$-neighborhood of $\Lambda$, and
  has nonempty intersection with the $\eps$-neighborhood of
  $\phi^{-1}(Z)$.
\end{proposition}

The main ingredient of the proof of \Cref{prop:main_prop} will be the
relative automaton $\mc{G}$ and a certain $\mc{G}$-compatible system
$\mc{U} = \{U_v\}_{v \in V(\mc{G})}$ of open subsets of $\flags$, as
described in the previous section. The general idea is that the
automaton $\mc{G}$ and system $\mc{U}$ give a ``relatively locally
finite'' encoding of the extended convergence action of $\Gamma$ on
$\flags$, in the sense that the data of the extended convergence
action can be (partially) recovered from $\mc{G}$ and $\mc{U}$. We
will be able to show that, if we perturb $\rho$ in a manner which
roughly preserves $\mc{G}$-compatibility of the system, we then get a
new extended convergence action, hence a new EGF representation.

The outline here is similar to the proof of the special case of this
theorem appearing in \cite{Weisman2022}. The main conceptual
difference is that here, we allow perturbations of $\rho$ which do not
actually preserve $\mc{G}$-compatibility of the system $\mc{U}$.
Instead, we only ask for $\mc{G}$-compatibility to hold along paths
which avoid the nontrivial parts of Dehn filling kernels.

The (semi)continuity properties of the set $\Lambda$ expressed by the
``moreover'' part of \Cref{prop:main_prop} follow from the fact that
the union of vertex sets $\bigcup_{v \in V(\mc{G})} U_v$ is an
``approximation'' of the boundary set $\Lambda$, both for the original
representation $\rho$ and its perturbations. So, we can guarantee that
$\Lambda_\sigma$ is close to $\Lambda$ by picking all of the sets
$U_v$ to be small neighborhoods of $\phi^{-1}(z)$ for some
$z \in \bgamp$.

During the course of the proof of \Cref{prop:main_prop}, we will also
show the following:
\begin{proposition}
  \label{prop:relative_strong_to_strong}
  Let $W \subseteq \ghom(\Gamma, G; \pargps)$ be an extended Dehn
  filling space, and assume that for every parabolic point
  $q \in \bgamp$, the set $\phi^{-1}(q)$ has empty interior as a
  subset of $\flags$. If a sequence $\rho_n \in W$ converges
  \emph{relatively strongly} to $\rho$ (i.e. in
  $\ghom(\Gamma, G; \pargps)$), then $\rho_n$ converges strongly to
  $\rho$ (i.e. in $\ghom(\Gamma, G)$).
\end{proposition}

\begin{remark}
  There exist examples of boundary extensions for EGF representations
  where the fiber over a parabolic point has \emph{nonempty} interior;
  see \cite[Example 2.6]{GW2024}. These boundary extensions do not
  satisfy the hypotheses of
  \Cref{prop:relative_strong_to_strong}. However, in the cited
  examples, there is a different choice of boundary extension where
  each fiber \emph{does} have empty interior, and it seems likely to
  the author that it is always possible to choose such a boundary
  extension for any EGF representation.
\end{remark}

\subsection{Choosing a neighborhood in the Dehn filling space}

We fix once and for all an extended Dehn filling space $W$ about the
EGF representation $\rho:\Gamma \to G$. To choose the open
neighborhood $O$ in $\ghom(\Gamma, G; \pargps)$ as in
\Cref{prop:main_prop}, we need to construct a relative automaton for
$\rho$ in a way which depends on the coarse metric geometry of the
pair $(\Gamma, \pargps)$, as well as the parameters $\eps$ and $Z$
appearing in the statement of the proposition. So, for the rest of the
section, let $X = X(\Gamma, \pargps)$ be the cusped space for
$(\Gamma, \pargps)$, and fix a compact set $Z \subseteq \bgamp$ and a
constant $\eps > 0$.

Before defining the automaton, we need to set up a few more auxiliary
parameters.
\begin{definition}
  \label{defn:distance_lowerbound}
  Recall that we have fixed a metric $\bdist$ on the Bowditch boundary
  $\bgamp$. Define the quantity $\Delta > 0$ by:
  \[
    \Delta := \min\{\bdist(z_1, z_2) : z_1, z_2 \in \bgamp \textrm{
      are joined by a geodesic in } X \textrm{ through id}\}.
  \]
\end{definition}

For each point $p \in \Pi$, we also pick a compact set
$K_p \subset \bgamp \minus \{p\}$ as follows:
\begin{definition}
  \label{defn:fund_domains}
  For each $p \in \Pi$, let $K_p \subset \bgamp$ be the set
  \[
    \{z \in \bgamp : \textrm{there exists a geodesic through id in
    }X\textrm{ joining }z \textrm{ to } p\}.
  \]
\end{definition}
Note that $K_p$ is compact because $X$ is locally compact. Moreover,
since every point in $\bgamp \minus \{p\}$ is joined to $p$ by a
geodesic passing through the boundary of the combinatorial horoball
with ideal boundary $\{p\}$, we have
$\Stab_\Gamma(p) \cdot K_p = \bgamp \minus \{p\}$.

\begin{definition}[The relative automaton]
  Let $\mc{G}$ be a relative automaton for $(\Gamma, \mc{P})$, and let
  $\mc{W} = \{W_v\}_{v \in V(\mc{G})}$ and
  $\mc{U} = \{U_v\}_{v \in V(\mc{G})}$ be $\mc{G}$-compatible systems
  of subsets adapted to the EGF representation $\rho$, satisfying each
  of the properties
  \ref{item:g_paths_exist}-\ref{item:parabolic_edge_condition2} in
  \Cref{prop:automaton_exists}. The constant $\eps$ in condition
  \ref{item:g_sets_small} is chosen to be smaller than our fixed
  $\eps > 0$.

  Shrinking $\eps$ if necessary, we also ensure that $\mc{G}$,
  $\mc{W}$, and $\mc{U}$ satisfy conditions
  \ref{item:bgamp_sets_proper}-\ref{item:flag_sets_proper}. We
  additionally ensure that the relative automaton satisfies condition
  \ref{item:distinguish_points} (with respect to the constant $\Delta$
  in \Cref{defn:distance_lowerbound}), condition
  \ref{item:parabolic_domain_path} (with respect to the sets $K_p$ in
  \Cref{defn:fund_domains}), and condition
  \ref{item:approximate_compact_subset} (with respect to our fixed
  $\eps > 0$ and compact set $Z$).
\end{definition}

The automaton $\mc{G}$ and the systems $\mc{W}$, $\mc{U}$ determine a
few additional constants, which we will use to define a neighborhood
in the extended Dehn filling space $W$.

\begin{definition}
  \label{defn:automaton_metric_consts}
  Define a constant $C_1$ by:
  \[
    C_1 := \max_{v \in V(\mc{G})} \min_{\alpha \in T_v}
    \lcusp{\alpha}.
  \]
  Recall that for each parabolic vertex $v \in V(\mc{G})$, the label
  set $T_v$ lies in a unique coset $gP$ for $g \in \Gamma$ and
  $P \in \pargps$. For each of these vertices, fix such a $g = g_v$ to
  minimize $\lcusp{g_v}$, and define
  \[
    C_2 := \max\{\lcusp{g_v} \st v \in V(\mc{G}) \textrm{
      parabolic}\}.
  \]
  Finally, recall that for each parabolic vertex $v$, the set
  $F_v = gP \minus T_v$ is finite. Define
  \[
    C_3 := \max\{\lcusp{\gamma} \st \gamma \in F_v \textrm{ for some
      parabolic }v \in V(\mc{G})\}.
  \]
\end{definition}

Now, fix $\delta > 0$ so that the cusped space
$X = X(\Gamma, \pargps)$ is $\delta$-hyperbolic, and so that (by
\Cref{prop:uniformly_hyperbolic_fillings}) the cusped space for any
sufficiently long Dehn filling of $(\Gamma, \pargps)$ is also
$\delta$-hyperbolic. For convenience, assume that $\delta > 1$
throughout.

\begin{definition}
  \label{defn:codings_quasigeodesic_nbhds}
  Let $D > 0$ be a constant satisfying the following: if $c$ is any
  $(2, 18\delta)$-quasi-geodesic ray in $X$ from $\identity$ to
  $z \in \bgamp$, and $\{(v_i, \alpha_i)\}_{i=1}^N$ is a $\mc{G}$-path
  limiting to $z$, then the set of elements
  \[
    g_n = \alpha_1 \cdots \alpha_n
  \]
  lies within Hausdorff distance $D$ of
  $c([0, \infty)) \cap \cay(\Gamma)$, with respect to the metric
  $\dcusp$.

  Note that such a $D$ exists because of
  \Cref{cor:gpaths_track_quasigeodesics}. The specific quasi-geodesic
  constants $(2, 18\delta)$ we have chosen will arise in later
  arguments.
\end{definition}

As noted previously, any representation $\sigma:\Gamma \to G$
determines a Dehn filling $\pi^\sigma:\Gamma \to \Gamma^\sigma$ and a
collection of subgroups
$\mc{P}^\sigma = \{P/\ker(\sigma|_P)\}_{P \in \pargps}$. The kernel
$\pi^\sigma$ is a subgroup of $\ker(\sigma)$, so we will view $\sigma$
as a representation of both $\Gamma$ and $\repquot$.
\begin{definition}[Metrics determined by a representation $\Gamma \to G$]
  Recall that we have fixed a finite generating set $S$ for
  $\Gamma$. For any representation $\sigma:\Gamma \to G$, and any
  $\gamma \in \repquot$, let $\lgams{\gamma}$ denote the word length
  of $\gamma$ with respect to the generating set $\pi^\sigma(S)$. Let
  $\dgams$ denote the word metric
  $\dgams(\gamma_1, \gamma_2) = \lgams{\gamma_1^{-1}\gamma_2}$. When
  $\gamma \in \Gamma$, we will usually write $\lgams{\gamma}$ for
  $\lgams{\pi^\sigma(\gamma)}$, and similarly for $\dgams$.
  
  Let $X^\sigma$ denote the cusped space $X\quotpair$; recall from
  \Cref{sec:dehnfill_geometry} that we view this as a quotient of the
  cusped space $X$ by the action of the kernel of the map
  $\Gamma \to \repquot$. Let $\dcusps$ denote the path metric on
  $X^\sigma$ (which restricts to a metric on $\repquot$), and let
  $\lcusps{\gamma} = \dcusps(\identity, \gamma)$.
\end{definition}

We are now ready to describe the open neighborhood of $\sigma$ in
$\ghom(\Gamma, G; \pargps)$ appearing in \Cref{prop:main_prop}.
\begin{proposition}
  \label{prop:ghom_nbhd}
  There exists an open subset $O \subset \ghom(\Gamma, G; \pargps)$
  containing $\rho$ and a constant $r > 0$ so that every
  $\sigma \in O \cap W$ satisfies all of the following conditions:
  \begin{enumerate}[label=(O\arabic*)]
  \item \label{item:perturbed_kernel_bounds} The cusped space
    $X^\sigma$ is $\delta$-hyperbolic (so $(\repquot, \pargps^\sigma)$
    is a relatively hyperbolic pair). Further, if $c$ is any geodesic
    in $X$ lying in a $(C_1 + 3D)$-neighborhood of $\cay(\Gamma)$,
    then its image under the quotient map $X \to X^\sigma$ is a
    $(2, 18\delta$)-quasi-geodesic.
  \item \label{item:nonparabolic_vertex_inclusions} For every
    non-parabolic vertex $v$ in the relative quasi-geodesic automaton
    $\mc{G}$, and every edge $v \to w$ in $\mc{G}$, we have
    \[
      \sigma(\alpha_v) \cdot \nbhd{U_w}{r} \subset U_v.
    \]
  \item \label{item:small_elements_ok} For every parabolic vertex $v$
    in the automaton $\mc{G}$, if $\alpha \in T_v$ satisfies
    \[
      \lcusp{\alpha} \le 2C_2 + 4C_3 + 6D + 18\delta,
    \]
    then for every edge $v \to w$ in $\mc{G}$, we have
    \[
      \sigma(\alpha) \cdot \nbhd{U_w}{r} \subset U_v.
    \]
  \item \label{item:dehn_filling_space} For every parabolic vertex $v$
    in the automaton $\mc{G}$, with corresponding parabolic point
    $q_v = gp$ for $p \in \Pi$, and every edge $v \to w$ in $\mc{G}$,
    we have
    \[
      (\sigma(g\Gamma_p) \minus \sigma(F_v)) \cdot \nbhd{U_w}{r} \subset U_v.
    \]
\end{enumerate}
\end{proposition}
\begin{remark}
  The quantity $C_1 + 3D$ in \ref{item:perturbed_kernel_bounds} is
  motivated by \Cref{lem:short_jumps_close_to_cay}. The inequality in
  \ref{item:small_elements_ok} will arise later, in the proof of
  \Cref{lem:perturbed_nesting}.
\end{remark}
\begin{proof}[Proof of \Cref{prop:ghom_nbhd}]
  First, note that when $\sigma = \rho$, condition
  \ref{item:perturbed_kernel_bounds} holds by assumption, since $\rho$
  is faithful and thus $X = X^\rho$. In this case the existence of an
  $r$ so that the other three conditions hold is equivalent to the
  $\mc{G}$-compatibility of the system $\mc{U}$ for the $\rho$-action
  of $\Gamma$ on $\flags$. So we need to show that all four of these
  conditions are relatively open in $W$, in the topology of relative
  strong convergence.

  Condition \ref{item:nonparabolic_vertex_inclusions} only refers to
  finitely many elements of $\Gamma$, so it is already open in
  $\Hom(\Gamma, G)$; by \Cref{lem:cusp_cay_comparison}, the same is
  true for Condition \ref{item:small_elements_ok}. From
  \Cref{prop:uniformly_hyperbolic_fillings} and
  \Cref{prop:cay_geodesics_descend}, the conditions in
  \ref{item:perturbed_kernel_bounds} hold whenever $X^\sigma$ is the
  cusped space associated to a sufficiently long filling of
  $\Gamma$. Thus this condition holds as long as the nontrivial
  elements in $\ker(\sigma)$ avoid a fixed finite subset of
  $\Gamma$. We assumed that $\rho$ was faithful, so this also
  corresponds to finitely many conditions on $\sigma$ which are open
  in $\Hom(\Gamma, G)$.

  On the other hand, the final condition \ref{item:dehn_filling_space}
  is \emph{not} necessarily open in $\Hom(\Gamma, G)$, but it
  \emph{is} relatively open in the extended Dehn filling space
  $W \subset \relghom(\Gamma, G; \pargps)$, essentially by definition.
\end{proof}

\subsection{Compatibility of the deformed automaton along lifts}

For the rest of the section, fix a neighborhood
$O \subset \ghom(\Gamma, G; \pargps)$ as in \Cref{prop:ghom_nbhd}. The
lemma below is in some sense the key step in the proof of
\Cref{prop:main_prop}: it says that, if $\sigma$ is any deformation of
$\rho$ in $O \cap W$, then for certain special $\mc{G}$-paths, the
$\mc{G}$-compatibility condition is preserved for the action of
$\sigma$ on $\flags$.

We should \emph{not} expect the $\mc{G}$-compatibility condition to be
preserved for \emph{arbitrary} $\mc{G}$-paths, since we need to worry
about elements in the kernel of $\sigma$ appearing along the
path. However, the lemma below shows (roughly) that if we only
consider $\mc{G}$-paths which stay close to \emph{lifts} of
quasi-geodesic paths in $X^\sigma$, then each element appearing in the
path is uniformly far from the nontrivial part of the kernel of
$\sigma$, and so $\mc{G}$-compatibility still holds. (Recall the
definition of a \emph{lift} in $X$ of a path in $X^\sigma$ from the
end of \Cref{sec:rel_hyp_dehn_filling}.)

\begin{lemma}
  \label{lem:perturbed_nesting}
  Let $\sigma \in W \cap O$, let $I \subseteq [0, \infty)$ be an
  interval, and let $c:I \to X^\sigma$ be a
  $(2, 18\delta)$-quasi-geodesic path based at $\identity$ in
  $X^\sigma$ with a lift $\tilde{c}$ based at $\identity$ in $X$. Let
  $\{(v_i, \alpha_i)\}_{i=1}^N$ be a (finite or infinite)
  $\mc{G}$-path so that the sequence $g_n = \alpha_1 \cdots \alpha_n$
  lies in a $D$-neighborhood of $\tilde{c}(I) \cap \cay(\Gamma)$.

  Then, for any $1 \le i < N$, we have
  \begin{equation}\label{eq:alpha_inclusion}
    \sigma(\alpha_i) \cdot \nbhd{U_{v_{i+1}}}{r} \subset U_{v_i}.
  \end{equation}
\end{lemma}
\begin{proof}
  Fix $1 \le i < N$. Condition
  \ref{item:nonparabolic_vertex_inclusions} on our open subset $O$
  implies that the desired inclusion holds whenever $v_i$ is a
  non-parabolic vertex. So, assume that $v_i$ is parabolic.

  Fix $\tilde{a}_j, \tilde{a}_{k} \in \Gamma$ lying along $\tilde{c}$
  so that $\dcusp(\tilde{a}_j, g_i) \le D$ and
  $\dcusp(\tilde{a}_{k}, g_{i+1}) \le D$, and let
  $a_j = \pi^\sigma(\tilde{a}_j)$, $a_k =
  \pi^\sigma(\tilde{a}_k)$. Then, since $\pi^\sigma$ is $1$-Lipschitz
  with respect to the metrics $\dcusp$, $\dcusps$, we have
  \begin{align*}
    \dcusps(a_j, a_k)
    &\le \dcusps(g_i, g_{i+1}) + 2D\\
    &= \lcusps{\alpha_i} + 2D.
  \end{align*}
  Now, if $\sigma(\alpha_i) \notin \sigma(F_{v_i})$, condition
  \ref{item:dehn_filling_space} on $O$ implies that the desired
  inclusion \eqref{eq:alpha_inclusion} holds. So assume that
  $\sigma(\alpha_i) = \sigma(f)$ for some $f \in F_{v_i}$. By way of
  \Cref{defn:automaton_metric_consts}, we can write $\alpha_i = gh_i$
  and $f = gh$, where $h_i, h$ both lie in some parabolic subgroup
  $P \in \pargps$, and $g \in \Gamma$ satisfies $\lcusp{g} \le
  C_2$. We must have $\sigma(h_i) = \sigma(h)$, so
  $\lcusps{h_i} = \lcusps{h}$. Therefore,
  \begin{align*}
    \lcusps{\alpha_i} &\le \lcusps{g} + \lcusps{h_i} = \lcusps{h} +
                        \lcusps{g}\\
                      &\le \lcusps{gh} + 2\lcusps{g}\\
                      &\le \lcusp{gh} + 2\lcusp{g} \le C_2 + 2C_3.
  \end{align*}
  This means that
  \[
    \dcusps(a_j, a_{k}) \le C_2 + 2C_3 + 2D.
  \]
  Since $c$ is a $(2,18\delta)$-quasi-geodesic path and $\tilde{c}$ is
  a lift of $c$, this implies
  \[
    \dcusp(\tilde{a}_k, \tilde{a}_{k'}) \le 2C_2 + 4C_3 + 4D +
    18\delta,
  \]
  and therefore
  \[
    \lcusp{\alpha_i} = \dcusp(g_i, g_{i+1}) \le 2C_2 + 4C_3 + 6D + 18\delta.
  \]
  Then we are done after applying condition
  \ref{item:small_elements_ok}.
\end{proof}

\subsection{Strong convergence}

With the setup above, we are in position to prove
\Cref{prop:relative_strong_to_strong}, which says that convergence to
$\rho$ in $\ghom(\Gamma, G; \pargps)$ (\emph{relative} strong
convergence) implies convergence to $\rho$ in $\ghom(\Gamma, G)$
(strong convergence).

We first prove the following lemma, which will also be useful for the
proof of \Cref{prop:main_prop}.
\begin{lemma}
  \label{lem:coned_unbounded_divergent}
  Let $\rho_n$ be a sequence in $O \cap W$, and let $\gamma_n$ be a
  sequence in $\Gamma$. If the sequence of relative word-lengths
  $|\gamma_n|^{\rho_n}_{\hat{\Gamma}}$ tends to infinity, then
  $\rho_n(\gamma_n)$ is $Q$-divergent.
  
  In addition, if there is a sequence of $\rho_n$-invariant sets
  $\Lambda_n \subset \flags$ such that
  $\Lambda_n \cap U_v \ne \emptyset$ for each $v \in V(\mc{G})$, then
  every $Q$-limit point of $\rho_n(\gamma_n)$ is an accumulation point
  of some sequence $\xi_n \in \Lambda_n$.
\end{lemma}
\begin{proof}
  It suffices to show that every subsequence of
  $\gamma_n, \rho_n, \Lambda_n$ has a further subsequence which
  satisfies the desired statement, so we can extract subsequences
  throughout the proof below.

  Let $X_n$ be the sequence of cusped spaces $X^{\rho_n}$, and let
  $d_X^n$ denote the induced metric $d_{X}^{\rho_n}$ on the quotient
  $\Gamma_n = \Gamma^{\rho_n}$. Condition
  \ref{item:perturbed_kernel_bounds} ensures that each of these cusped
  spaces is $\delta$-hyperbolic, so by
  \Cref{prop:cusp_rays_quasidense}, there is a geodesic ray
  $c_n:[0, \infty) \to X_n$ and $t_n \in [0, \infty)$ so that
  $d_X^n(c_n(t_n), \gamma_n) < 8 + 21\delta$.

  Let $\tilde{c}_n$ be a lift of $c_n$ in $X$ based at the
  identity. The quotient map $X \to X_n$ defined in
  \Cref{sec:dehnfill_geometry} takes combinatorial horoballs to
  combinatorial horoballs, so by lifting a geodesic in $X_n$ from
  $c(t_n)$ to $\gamma_n$, we see that the point $\tilde{c}_n(t_n)$
  lies $d_X^n$-distance at most $8 + 21\delta$ from some
  $\tilde{\gamma}_n \in \Gamma$, satisfying
  $\pi_n(\tilde{\gamma}_n) = \gamma_n$. Since the quotient
  $\Gamma \to \Gamma_n$ also induces a $1$-Lipschitz map
  $\hcay(\Gamma) \to \hcay(\Gamma_n)$, the relative word lengths
  $\lcon{\tilde{\gamma}_n}$ must tend to infinity.

  For each $n$, let $\{(v_i^{(n)}, \alpha_i^{(n)})\}_{i=1}^N$ be a
  $\mc{G}$-path limiting to the endpoint of $\tilde{c}_n$ in $\bgamp$,
  and for each finite $m \le N$ define
  \[
    g_m^{(n)} = \alpha_1^{(n)} \cdots \alpha_m^{(n)}.
  \]
  By \Cref{lem:cayley_hausdorff_neighborhoods} and
  \Cref{cor:gpaths_track_quasigeodesics}, for each $n$, we can find
  some index $m_n < \infty$ so that the element $g_{m_n}^{(n)}$ lies
  within uniformly bounded $\cay(\Gamma)$-distance of
  $\tilde{\gamma}_n$. Let $h_n = g_{m_n}^{(n)}$. We must have
  $\lcon{h_n} \to \infty$, and since each $\alpha_i^{(n)}$ has
  uniformly bounded length in $\hcay(\Gamma)$, we have
  $m_n \to \infty$.

  Now, after extracting a subsequence, we can assume that the vertices
  $v_1^{(n)}, v^{(n)}_{m_n + 1}$ are both fixed (independent of $n$),
  respectively equal to vertices $v, w \in V(\mc{G})$. Consider the
  sequence of sets $\rho_n(h_n)U_w$. We have chosen the constant $D$
  (in \Cref{defn:codings_quasigeodesic_nbhds}) in a way which ensures
  that the sequence $h_n$ lies within Hausdorff distance $D$ of
  $\tilde{c}_n \cap \cay(\Gamma)$, so \Cref{lem:perturbed_nesting}
  ensures that the nesting condition
  \[
    \rho_n(\alpha_i^{(n)})\nbhd{U_{v_{i+1}^{(n)}}}{r} \subset
    U_{v_i^{(n)}}
  \]
  is satisfied for every $1 \le i \le m_n$. We also know from property
  \ref{item:flag_sets_proper} that, if $r$ is sufficiently small, each
  vertex set $U_v$ satisfies $\Opp(\nbhd{U_v}{r}) \ne
  \emptyset$. Then, it follows from \cite[Cor. 712]{Weisman2022} that,
  for a particular choice of metric on $U_{v_1}$ (compatible with the
  topology on $\flags$), the diameter of the set $\rho_n(h_n)U_w$
  tends to zero uniformly exponentially in $m_n$. In particular, since
  $\overline{U_1}$ is compact, these sets tend to a singleton
  $\{\xi\} \subset \flags$. Since $U_w$ is open,
  \cite[Prop. 3.7]{Weisman2022} shows that $\rho_n(h_n)$ is
  $Q$-divergent and $\xi$ is the unique $Q$-limit point of the
  sequence. Since the $\cay(\Gamma)$-distance between $h_n$ and
  $\tilde{\gamma}_n$ is uniformly bounded,
  \cite[Lem. 4.23(i)]{KLP2017} implies that
  $\rho_n(\gamma_n) = \rho_n(\tilde{\gamma}_n)$ is also $Q$-divergent
  with unique $Q$-limit point $\xi$.

  To finish the proof, we just need to show that, if
  $\Lambda_n \subset \flags$ is a sequence of subsets as in the
  hypotheses, then $\xi$ is an accumulation point of some sequence
  $\xi_n \in \Lambda_n$. However, we know that $\xi$ is the limit of
  \emph{any} sequence $\rho_n(h_n)\eta_n$ for any sequence
  $\eta_n \in U_1$. Since $\Lambda_n$ is $\rho_n$-invariant and
  $\Lambda_n \cap U_1 \ne \emptyset$ by assumption, the result
  follows.
\end{proof}

\begin{remark}
  \label{rem:rel_q_div}
  Applying \Cref{lem:coned_unbounded_divergent} to a constant sequence
  of representations shows that any representation in $O \cap W$ is
  relatively $Q$-divergent. In particular, this proves that $\rho$ is
  relatively $Q$-divergent.
\end{remark}

The next lemma says (modulo an additional assumption on the boundary
extension $\phi$) that the representations $\sigma \in O \cap W$
induce ``uniformly proper'' representations $\Gamma^\sigma \to G$. It
implies in particular that every representation in $O \cap W$ is
discrete.
\begin{lemma}
  \label{lem:uniformly_proper}
  Assume that for each parabolic point $q \in \bgamp$, the set
  $\phi^{-1}(q)$ has empty interior. Let $\rho_n$ be a sequence in
  $O \cap W$, and for each $n$, let $\gamma_n \in \Gamma^{\rho_n}$. If
  $|\gamma_n|_{\Gamma}^{\rho_n} \to \infty$, then $\rho_n(\gamma_n)$
  leaves every compact subset of $G$.
\end{lemma}
\begin{proof}
  Since $Q$-divergent sequences necessarily leave every compact subset
  of $G$, by employing the previous lemma, we only need to consider
  the case where the relative lengths $\lcon{\gamma_n}^{\rho_n}$ are
  bounded. 

  Let $\Gamma_n = \Gamma^{\rho_n}$, and let
  $\pi_n:\Gamma \to \Gamma_n$ be the quotient map. After extracting a
  subsequence, there is a fixed integer $k$ and lifts
  $\tilde{\gamma}_n \in \Gamma$ of $\gamma_n$ with the form
  \begin{equation}
    \label{eq:lift_alternating_word}
    \tilde{\gamma}_n = \tilde{g}_1\tilde{h}_1^{(n)} \ldots
    \tilde{g}_k\tilde{h}_k^{(n)}\tilde{g}_{k+1}, 
  \end{equation}
  where each $\tilde{g}_i \in \Gamma$ is fixed, and each
  $\tilde{h}_i^{(n)}$ is an element in some fixed $P_i \in \pargps$,
  such that $\pi_n(\tilde{h}_i^{(n)}) = h_i^{(n)}$ satisfies
  \[
    \lim_{n \to \infty} \lgam{h_i^{(n)}}^{\rho_n} = \infty.
  \]
  For simplicity we will assume $\tilde{g}_{k+1} = \identity$, since
  the general case follows easily from this one. Let $p_i \in \Pi$ be
  the parabolic point fixed by $P_i$.

  Note that for each $1 < j \le k$, the fact that
  $\lgam{h_j^{(n)}}^{\rho_n} \to \infty$ as $n \to \infty$ means that
  for any finite subset $F \subset P_j$, we eventually have
  $\rho_n(h_j^{(n)}) \notin \rho_n(F)$. So, due to the extended Dehn
  filling assumption (\Cref{defn:edf}), for each $j$, if $K$ is a
  compact subset of $C_{p_j}$, then the set $\rho_n(h_j^{(n)})K$
  eventually lies in an arbitrarily small neighborhood of
  $\phi^{-1}(p_j)$. Since $\rho_n \to \rho$ in $\Hom(\Gamma, G)$, we
  also know that as $n \to \infty$, the set
  $\rho_n(\tilde{g}_j\tilde{h}_j^{(n)})K$ eventually lies in an
  arbitrarily small neighborhood of $\rho(\tilde{g}_j)\phi^{-1}(p_j)$.

  We can induct on $j$ to show that for any compact subset
  $K \subset C_{p_j}$, the set
  \[
    \rho_n(\tilde{g}_1 \tilde{h}_1^{(n)} \cdots
    \tilde{g}_j\tilde{h}_j^{(n)})K
  \]
  eventually lies in an arbitrarily small neighborhood of
  $\phi^{-1}(\tilde{g}_1p_1)$. The case $j = 1$ is immediate from the
  previous paragraph. For the case $j > 1$, first observe that since
  $\tilde{g}_jp_j \ne p_{j-1}$, the group element $\rho(\tilde{g}_j)$
  takes (the closure of) a neighborhood of $\phi^{-1}(p_j)$ to
  $C_{p_{j-1}}$. Again using algebraic convergence $\rho_n \to \rho$,
  the sequence $\rho_n(\tilde{g}_j)$ eventually takes some
  neighborhood of $\phi^{-1}(p_j)$ to land in a compact subset of
  $C_{p_{j-1}}$. Applying induction, we see that
  \[
    \rho_n(\tilde{g}_1 \tilde{h}_1^{(n)} \cdots
    \tilde{g}_j\tilde{h}_j^{(n)})K
  \]
  eventually lies in an arbitrary neighborhood of
  $\phi^{-1}(\tilde{g}_jp_1)$.

  Now, since $C_{p_k}$ is an open subset of the flag manifold
  $\flags$, it contains a compact subset $K$ with nonempty
  interior. Suppose for a contradiction that $\rho_n(\gamma_n)$ does
  not leave every compact subset of $G$. Then, after extraction,
  $\rho_n(\gamma_n)$ converges to some $g \in G$, and
  $\rho_n(\gamma_n)K$ converges (with respect to Hausdorff distance on
  $\flags$) to $gK$, which is a subset of $\flags$ with nonempty
  interior. However, from the argument above, we know that
  $\rho_n(\gamma_n)K$ eventually lies in every neighborhood of
  $\phi^{-1}(\tilde{g}_1p_1)$, and therefore
  $gK \subseteq \phi^{-1}(\tilde{g}_1p_1)$. This is a contradiction if
  $\phi^{-1}(\tilde{g}_1p_1)$ has empty interior.
\end{proof}

\begin{proof}[Proof of \Cref{prop:relative_strong_to_strong}]
  Suppose that a sequence $\rho_n \in W$ converges to $\rho$ in
  $\ghom(\Gamma, G; \pargps)$. We need to show that also
  $\rho_n \to \rho$ in $\ghom(\Gamma, G)$, or equivalently that the
  subgroups $\rho_n(\Gamma)$ converge to $\rho(\Gamma)$ in the
  Chabauty topology. We will employ the criteria in
  \Cref{prop:chab_converge}. Since $\rho_n \to \rho$ algebraically,
  criterion \ref{item:chab_contain_1} in the proposition is
  immediate. For criterion \ref{item:chab_contain_2}, consider a
  sequence $\gamma_n \in \Gamma^{\rho_n}$ such that
  $\rho(\gamma_n) \to g$ for $g \in G$. By
  \Cref{lem:uniformly_proper}, the sequence $\lgam{\gamma_n}^{\rho_n}$
  is bounded, so there is a sequence $\tilde{\gamma}_n \in \Gamma$
  with $\lgam{\tilde{\gamma}_n}$ bounded, projecting to $\gamma_n$
  under the Dehn filling $\Gamma \to \Gamma^{\rho_n}$. After
  extraction $\tilde{\gamma}_n$ is a constant $\tilde{\gamma}$, so
  $\rho_n(\gamma_n) = \rho_n(\tilde{\gamma})$ converges to
  $\rho(\tilde{\gamma})$, and criterion \ref{item:chab_contain_2} is
  satisfied.
\end{proof}

\subsection{Lifting points in boundaries of Dehn fillings}

For the rest of the section, we fix a representation
$\sigma \in O \cap W$. Let $\iquotgps$ denote the set of infinite
subgroups in $\quotgps$ (note that this may be \emph{smaller} than the
set of images of groups in $\ipargps$ in $\repquot$, since
$\ker(\sigma|_P)$ could have finite index in some $P \in
\ipargps$). As in \Cref{sec:bowditch_boundary}, we will identify the
Bowditch boundary $\biquot$ with a compact $\Gamma^\sigma$-invariant
subset of $\bquot$, which is the boundary of the cusped space
$X^\sigma = X(\repquot, \quotgps)$.

Our aim now is to show that $\sigma$ induces an EGF representation of
$\Gamma^\sigma$. This means we need to construct a boundary extension
$\phi_\sigma:\Lambda_\sigma \to \biquot$ for $\sigma$, and then show
that $\phi_\sigma$ satisfies all of the needed properties.

For the most part our strategy follows the proof of the main theorem
in \cite{Weisman2022}, which is a special case of the result we are
proving now. However, here we have a new technical challenge, since
the Bowditch boundary $\biquot$ is \emph{not} necessarily homeomorphic
to $\bgamp$. In a sense $\biquot$ is ``smaller,'' since the quotient
group $\Gamma^\sigma$ contains ``fewer'' geodesic rays, but there is
not a nice equivariant quotient map from a subset of $\bgamp$ to
$\biquot$. However, the lifting operation on (quasi)-geodesic rays in
$X^\sigma$ (see \Cref{sec:dehnfill_geometry}) allows us to loosely
associate points in $\bquot$ to points in $\bgamp$.

\begin{definition}
  Let $z \in \bquot$, and let $K \ge 1, A \ge 0$. A point $\tilde{z}$
  is a $(K,A)$-\emph{lift} of $z$ if there is a $(K,A)$-quasi-geodesic
  ray $c:[0, \infty) \to X^\sigma$ from the identity to $z$ and a lift
  $\tilde{c}:[0, \infty) \to X$ of $c$ from the identity to
  $\tilde{z}$.

  We say that $\tilde{z}$ is a \emph{lift} of $z$ if it is a
  $(K,A)$-lift of $z$ for some $K, A$. If $\tilde{z}$ is a
  $(1,0)$-lift of $z$ then we say it is a \emph{geodesic lift}.
\end{definition}

\begin{remark}
  Since every geodesic in $X^\sigma$ lifts, every point $z \in \bquot$
  has at least one geodesic lift in $\bgamp$; some points in $\bquot$
  may have a unique geodesic lift, while others may have many. The
  lifting operation also does \emph{not} need to be equivariant in any
  sense. Moreover, if $K' ,A'$ are different from $K, A$, then the set
  of $(K,A)$-lifts of a point $z \in \bquot$ may in principle be very
  different from the set of $(K', A')$-lifts.
\end{remark}

\subsection{Defining fibers of the boundary extension}

By considering lifts of points in $\biquot$, we will define a map
\[
  \psi_\sigma:\biquot \to \{\textrm{closed subsets of }\flags\},
\]
determining the fibers of $\phi_\sigma$. As a first step, we relate
lifts to $\mc{G}$-paths by the following lemma.
\begin{lem}
  \label{lem:gpaths_close_in_quot}
  Let $z \in \bquot$. For any lift $\tilde{z}$ of $z$, if
  $\{(v_i, \alpha_i)\}_{i=1}^N$ is a $\mc{G}$-path limiting to
  $\tilde{z}$, then for any quasi-geodesic ray
  $c:[0, \infty) \to X^\sigma$ with $c(\infty) = z$, the sequence
  \[
    \pi^\sigma(\alpha_1 \cdots \alpha_n)
  \]
  lies within finite Hausdorff distance of
  $c([0, \infty)) \cap \cay(\repquot)$, with respect to the metric
  $\dgams$ on $\cay(\repquot)$.
\end{lem}
\begin{proof}
  Let $c'$ be a quasi-geodesic ray in $X^\sigma$ from the identity to
  $z$, so that a lift $\tilde{c}'$ based at the identity has ideal
  endpoint $\tilde{z}$. By \Cref{cor:gpaths_track_quasigeodesics}, if
  $\{(v_i, \alpha_i)\}_{i=1}^N$ is a $\mc{G}$-path limiting to
  $\tilde{z}$, the Hausdorff distance between
  $\tilde{c}'([0, \infty)) \cap \cay(\Gamma)$ and the sequence
  $g_n = \alpha_1 \cdots \alpha_n$ is finite. Since $\pi^\sigma$ is
  1-Lipschitz on $X$, the Hausdorff distance between
  $c'([0, \infty)) \cap \cay(\repquot)$ and $\pi^\sigma(g_n)$ with
  respect to $\dgams$ is finite as well.

  Since $X^\sigma$ is $\delta$-hyperbolic, the images of $c$ and $c'$
  are also finite Hausdorff distance apart with respect to the metric
  on $X^\sigma$. Then by \Cref{cor:hausdorff_distance_bound}, the
  Hausdorff distance between $c'([0, \infty) \cap \cay(\repquot)$ and
  $c([0, \infty)) \cap \cay(\repquot)$ with respect to $\dgams$ is
  also finite, and the result follows.
\end{proof}

\subsubsection{Conical fibers} We will first define $\psi_\sigma$ on
the set of conical limit points in $\bquot$, using the lemma below.

\begin{lem}
  \label{lem:conical_singletons}
  Let $z$ be a conical limit point in $\bquot$. There is a unique
  point $\psi_{\sigma}(z) \in \flags$ satisfying the following: if
  $\tilde{z}$ is any $(2, 18\delta)$-lift of $z$ in $\bgamp$, and
  $\{(v_i, \alpha_i)\}_{i=1}^\infty$ is a $\mc{G}$-path limiting to
  $\tilde{z}$, then
  \[
    \{\psi_{\sigma}(z)\} = \bigcap_{i=1}^\infty \sigma(\alpha_1 \cdots
    \alpha_i) U_{v_{i+1}}.
  \]
\end{lem}
\begin{proof}
  By arguing as in the proof of \Cref{lem:coned_unbounded_divergent},
  we can use property \ref{item:flag_sets_proper} and employ
  \cite[Cor. 7.12]{Weisman2022} to see that the given intersection is
  always a singleton. We need to check that this singleton does not
  depend on the choice of $(2, 18\delta)$-lift $\tilde{z}$, or on the
  choice of $\mc{G}$-path limiting to $\tilde{z}$.
  
  So, let $\tilde{z}$, $\tilde{z}'$ be a pair of $(2, 18\delta)$-lifts
  of $z$, and let $c$, $c'$ be $(2,18\delta)$-quasi-geodesic rays in
  $X^\sigma$ from the identity to $z$ so that lifts $\tilde{c}$,
  $\tilde{c}'$ tend from the identity to $\tilde{z}, \tilde{z}'$
  respectively. We want to show that for any $\mc{G}$-paths
  \[
    \{(v_i, \alpha_i)\}_{i=1}^\infty, \quad \{(w_j,
    \beta_j)\}_{j=1}^\infty
  \]
  limiting to $\tilde{z}$, $\tilde{z}'$ respectively, we have
  \begin{equation}
    \label{eq:intersection_equality}
    \bigcap_{i=1}^\infty \sigma(\alpha_1 \cdots \alpha_i) U_{v_{i+1}} =
    \bigcap_{j=1}^\infty \sigma(\beta_1 \cdots \beta_j) U_{w_{j+1}}.
  \end{equation}
  Let $g_n = \alpha_1 \cdots \alpha_n$ and
  $h_n = \beta_1 \cdots \beta_n$. For any subsequence of $g_n$, we can
  choose a further subsequence so that $U_{v_{n+1}}$ is a fixed subset
  $U_v$ for all $n$. Then, \cite[Prop. 3.6]{Weisman2022} implies that
  $g_n$ is $Q$-divergent with $Q$-limit set equal to the singleton
  given by the left-hand side of \eqref{eq:intersection_equality}. The
  same argument implies that $h_n$ is also $Q$-divergent, with
  $Q$-limit set the right-hand side of
  \eqref{eq:intersection_equality}.

  By \Cref{lem:gpaths_close_in_quot}, both of $\pi^\sigma(g_n)$ and
  $\pi^\sigma(h_n)$ are within bounded Hausdorff distance of
  $r([0, \infty)) \cap \cay(\repquot)$, with respect to $\dgams$. So
  the sequences $\sigma(g_n)$ and $\sigma(h_n)$ are bounded distance
  apart in $G$. By \cite[Lemma 4.19]{KLP2017}, this means that
  $\sigma(g_n)$ and $\sigma(h_n)$ have the same $Q$-limit sets, so the
  intersections in \eqref{eq:intersection_equality} are equal.
\end{proof}

\begin{lem}
  The assignment $z \mapsto \psi_{\sigma}(z)$ determines a
  $\sigma$-equivariant map from the set of conical limit points in
  $\bquot$ to $\flags$.
\end{lem}
\begin{proof}
  The map $\psi_\sigma$ is well-defined on conical points by the
  previous lemma, so the only thing we need to check is
  equivariance. In fact, we just need to show that for any $s$ in the
  generating set $S$, we have
  $\psi_{\sigma}(\pi^\sigma(s)z) = \sigma(s)\psi_{\sigma}(z)$. Fix a
  conical point $z \in \bquot$, let $\tilde{z}$ be a geodesic lift of
  $z$, and let $\tilde{c}$ be a geodesic ray from $\identity$ to
  $\tilde{z}$, lifting a geodesic from $\identity$ to $z$. Let
  $\{(v_i, \alpha_i)\}_{i=1}^\infty$ be a $\mc{G}$-path limiting to
  $\tilde{z}$.
  
  Now fix $s \in S$, and consider the path $c^s$ in $X^\sigma$
  obtained by concatenating the edge $(\identity, \pi^\sigma(s))$ with
  the $\pi^\sigma(s)$-translate of $c$. Since $s$ has length 1 and we
  assume $\delta \ge 1$, this path is a $(1, 2\delta)$-quasi-geodesic,
  with endpoint $\pi^\sigma(s)z \in \bquot$. By
  \ref{item:perturbed_kernel_bounds}, $\pi^\sigma$ is injective on
  $S$, so $(\identity, \pi^\sigma(s))$ lifts uniquely to the edge
  $(\identity, s)$ in $\cay(\Gamma)$. Then, concatenating this edge
  with the $s$-translate of $\tilde{c}$, we obtain a
  $(1, 2\delta)$-quasi-geodesic lift $\tilde{c}^s$ of $c^s$, whose
  endpoint is $s\tilde{z}$, a $(1, 2\delta)$-lift of $\pi(s)z$.

  Let $\{(w_j, \beta_j)\}_{j=1}^\infty$ be a $\mc{G}$-path limiting to
  $s\tilde{z}$. By \Cref{cor:gpaths_track_quasigeodesics}, the
  sequence
  \[
    h_n = \beta_1 \cdots \beta_n
  \]
  lies within finite Hausdorff distance of
  $\tilde{c}^s([0, \infty)) \cap \cay(\Gamma)$, with respect to the
  metric $\dcusp$. The same corollary also implies that
  $g_n = \alpha_1 \cdots \alpha_n$ lies within finite Hausdorff
  distance of $\tilde{c}([0, \infty)) \cap \cay(\Gamma)$, which means
  that the sequence $sg_n$ lies within finite Hausdorff distance of
  $h_n$, again with respect to the metric $\dcusp$; by
  \Cref{lem:cusp_cay_comparison}, the Hausdorff distance is also
  finite with respect to $\dgam$.

  Then, arguing as in the previous lemma, we see that the
  intersections
  \[
    \psi_{\sigma}(\pi^\sigma(s)z) = \bigcap_{i=1}^\infty
    \sigma(h_n)U_{w_{n+1}}, \quad \bigcap_{i=1}^\infty
    \sigma(s)\sigma(g_n)U_{v_{n+1}} = \sigma(s)\psi_{\sigma}(z)
  \]
  must agree.
\end{proof}

We have now defined $\psi_\sigma$ as an equivariant map on conical
points in $\bquot$. Note that these are precisely the same as the
conical points in $\biquot$, when viewing $\biquot$ as a subset of
$\bquot$.

\subsubsection{Parabolic fibers}\label{sec:parabolic_fibers}
Next we want to extend the map $\psi_\sigma$ to an equivariant map
\[
  \psi_\sigma:\biquot \to \{\textrm{closed subsets of }\flags\}.
\]
Let $\Pi^\sigma_\infty$ be the (finite) set of fixed points of the
groups in $\iquotgps$; we will first define $\psi_\sigma$ on
$\Pi^\sigma_\infty$. As in the analogous construction in
\cite{Weisman2022}, for each vertex $v$ of the relative automaton
$\mc{G}$, define an open set $B_v \subset \flags$ by
\[
  B_v = \bigcup_{w} U_w,
\]
where the union is taken over all vertices $w \in V(\mc{G})$ such that
there is an edge $v \to w$ in $\mc{G}$.

Each point $p$ in $\Pi^\sigma_\infty \subseteq \Pi^\sigma$ is uniquely
identified with some parabolic point in $\Pi \subset \bgamp$. We leave
this identification implicit, so we can view the points in
$\Pi^\sigma_\infty$ as parabolic points in $\bgamp$.

By condition \ref{item:parabolic_edge_condition1} on the automaton
$\mc{G}$, for each $p \in \Pi^\sigma_\infty$, there is a parabolic
vertex $v$ of $\mc{G}$ so that the associated parabolic point $q_v$ is
$p$. Define the set $\psi_\sigma(p)$ to be the closure of the set of
accumulation points of sequences of the form $\sigma(\gamma_n)\xi_n$
where:
\begin{itemize}
\item $\xi_n$ is a sequence in $B_v$, and
\item $\gamma_n$ is a sequence of pairwise distinct elements in
  $\repquot$, all lying in $\repquot_p = \Stab_{\repquot}(p)$.
\end{itemize}

Extend $\psi_\sigma$ to all parabolic points in $\biquot$ by
equivariance. We observe:
\begin{lem}
  \label{lem:automaton_local_parabolic}
  Let $p \in \Pi^\sigma_\infty$, and suppose that a parabolic vertex
  $v \in V(\mc{G})$ satisfies $q_v = gp$ for $g \in \Gamma$. Then
  $\sigma(g)\psi_\sigma(p) \subset U_v$.
\end{lem}
\begin{proof}
  Let $\xi \in \psi_\sigma(p)$; we will show that
  $\sigma(g)\xi \in U_v$. Let $P = \Stab_\Gamma(p)$, and let
  $P^\sigma = \Stab_{\repquot}(p)$, so $P^\sigma =
  P/\ker(\sigma|_P)$. Fix a vertex $w \in V(\mc{G})$ with $q_w =
  p$. Then by definition $\xi$ is the limit of a sequence
  $\sigma(\gamma_n)\xi_n$, where $\gamma_n$ is a sequence of pairwise
  distinct elements of $P^\sigma$ and $\xi_n$ is a sequence in $B_w$.

  By the definition of $P^\sigma$, the sequence $\gamma_n$ must
  eventually satisfy
  \[
    \sigma(g)\sigma(\gamma_n) \notin \sigma(F_v),
  \]
  where $F_v$ is the finite set $gP \minus T_v$. Condition
  \ref{item:dehn_filling_space} then implies that
  $\sigma(g)\sigma(\gamma_n)B_v$ is eventually a subset of
  $U_v$. Since $\lim_{n \to \infty}\sigma(\gamma_n)\xi_n = \xi$, we
  must have $\sigma(g)\xi \in U_v$, as required.
\end{proof}

We have now defined $\psi_\sigma$ as an equivariant map on all of the
conical points and parabolic point in $\biquot$, hence on all of
$\biquot$.

\begin{remark}
  The definition of $\psi_\sigma$ on parabolic points makes it clear
  why we have only defined $\psi_\sigma$ on $\biquot$, rather than
  $\bquot$: by definition, there are no infinite sequences of pairwise
  distinct elements in the $\repquot$-stabilizer of a point in
  $\bquot \minus \biquot$.
\end{remark}

\subsection{Defining the boundary extension}

Define a $\sigma$-invariant subset $\Lambda_\sigma \subset \flags$ by
\[
  \Lambda_\sigma = \bigcup_{z \in \biquot}\psi_\sigma(z).
\]
We wish to define a surjective equivariant antipodal map
$\Lambda_\sigma \to \biquot$. To do so, we need to show that, as $z$
ranges over $\biquot$, the sets $\psi_\sigma(z)$ partition
$\Lambda_\sigma$.

We will continue to imitate the proof of the special case in
\cite{Weisman2022}, but now we run into yet another complication: the
proof in \cite{Weisman2022} takes advantage of the fact that a
relatively hyperbolic group acts cocompactly on pairs of distinct
points in its Bowditch boundary. We have two groups to work with here,
which may not be isomorphic and which may not have homeomorphic
Bowditch boundaries, so we need the slightly more intricate statement
below.
\begin{lem}
  \label{lem:preferred_lift_separation}
  For any pair of distinct points $z_1, z_2 \in \bquot$, there is some
  $g \in \Gamma^\sigma$ and a pair of geodesic lifts $\tilde{z}_1$,
  $\tilde{z}_2$ of $gz_1$, $gz_2$ respectively such that
  $\bdist(\tilde{z}_1, \tilde{z}_2) > \Delta$, where $\Delta$ is the
  constant from \Cref{defn:distance_lowerbound}.
\end{lem}
\begin{proof}
  Since $z_1$ and $z_2$ are distinct, and $X^\sigma$ is hyperbolic,
  there is a bi-infinite geodesic $c:(-\infty, \infty) \to X^\sigma$
  such that $c(-\infty) = z_1$ and $c(\infty) = z_2$. This geodesic
  cannot lie in a single combinatorial horoball (since each such
  horoball has just one point in its ideal boundary) so it must pass
  through $\cay(\repquot)$ at some point $g \in \repquot$; after
  reparameterization we may assume $g = c(0)$.

  Consider the bi-infinite geodesic
  $c':(-\infty, \infty) \to X^\sigma$ given by $c'(t) =
  g^{-1}c(t)$. The endpoints of this geodesic are $g^{-1}z_1$ and
  $g^{-1}z_2$, respectively.

  Let $\tilde{c}$ be a lift of $c'$ in $X$, chosen so that
  $\tilde{c}(0) = \mathrm{id}$. Since $\tilde{c}$ is a lift of a
  geodesic, it is a bi-infinite geodesic in $X$, and its forward and
  backward ideal endpoints are by definition lifts of $g^{-1}z_1$ and
  $g^{-1}z_2$. These lifts are at the opposite ends of a bi-infinite
  geodesic in $X$ passing through the identity, so by definition of
  $\Delta$, the distance between them (with respect to $\bdist$) is at
  least $\Delta$.
\end{proof}

The lemma below similarly shows that we have a ``uniform'' version of
boundedness of the parabolic points in $\bquot$, via lifts in
$\bgamp$.
\begin{lem}
  \label{lem:lift_cocompactness}
  Let $p \in \Pi$, and let $P$ be the stabilizer of $p$ in
  $\Gamma$. If $K_p \subset \bgamp \minus \{p\}$ is the compact set
  given by \Cref{defn:fund_domains}, then for any
  $z \in \bquot \minus \{p\}$, there is some
  $h \in P^\sigma = \pi^\sigma(P)$ so that a geodesic lift of $hz$
  lies in $K_p$.
\end{lem}
Note that in the statement of the lemma, we have again implicitly
identified $\Pi^\sigma$ with $\Pi$.
\begin{proof}
  Fix $z$ in $\bquot - \{p\}$, and let
  $c:(-\infty, \infty) \to X^\sigma$ be a bi-infinite geodesic joining
  $z$ to $p$. This geodesic must eventually enter the combinatorial
  horoball $\mc{H}$ centered at $p$, i.e. the combinatorial horoball
  attached to $P^\sigma$. Let $h$ be the last point in
  $P^\sigma = \dee \mc{H}$ reached by $c$, so that after
  reparameterization we have $c(0) = h$ and $c(t) \in \mc{H}$ for all
  $t > 0$. Then we may assume that the positive direction of $c$ only
  traverses vertical edges in the combinatorial horoball $\mc{H}$.

  The bi-infinite geodesic $h^{-1}c$ has ideal endpoints $c^{-1}p = p$
  and $c^{-1}z$. Let $\tilde{c}$ be a lift of $h^{-1}c$ in $X$. The
  positive direction of $\tilde{c}$ must only traverse vertical edges
  lying in the combinatorial horoball attached to $P$, so
  $\tilde{c}(+\infty) = p$. On the other hand, $\tilde{c}(-\infty)$ is
  by definition a geodesic lift of $h^{-1}z$, which lies in $K_p$ by
  definition.
\end{proof}

The next lemma says that we can use the vertex sets $U_v \in \mc{U}$,
as well as their iteratively nested translates under certain
$\mc{G}$-paths, to ``approximate'' the sets $\psi_\sigma(z)$.

\begin{lem}
  \label{lem:gpath_determines_nbhd}
  If $z$ is a point in $\biquot$, $\tilde{z}$ is any
  $(2, 18\delta)$-lift of $z$, and $\{(v_i, \alpha_i)\}_{i=1}^N$ is a
  $\mc{G}$-path limiting to $\tilde{z}$, then for any $0 \le n < N$ we
  have
  \[
    \psi_\sigma(z) \subset \sigma(\alpha_1 \cdots
    \alpha_n)U_{v_{n+1}}.
  \]
  In particular, $\psi_\sigma(z) \subset U_{v_1}$.
\end{lem}
\begin{proof}
  If $z$ is a conical limit point, this is immediate from the
  definition of $\psi_\sigma$, so suppose $z$ is parabolic. Let
  $c:[0, \infty) \to X^\sigma$ be a $(2, 18\delta)$-quasi-geodesic ray
  limiting to $z$. Since $z$ is the center of a horoball in
  $X^\sigma$, there is some $a \ge 0$ such that $c(a) \in \repquot$
  and $c(t)$ lies in a combinatorial horoball centered at $z$ for all
  $t > a$. This horoball is attached to $X^\sigma$ at the coset
  $c(a)P^\sigma$ for $P^\sigma \in \quotgps$. If $P^\sigma$ is the
  stabilizer of $p \in \Pi^\sigma_\infty$, then $z = c(a)p$.

  Fix $P \in \pargps$ so $\pi^\sigma(P) = P^\sigma$, and let
  $\tilde{c}$ be a lift of $c$. Since the combinatorial horoballs lift
  to combinatorial horoballs, for all $t > a$, $\tilde{c}(t)$ must lie
  in the combinatorial horoball attached to the coset $\tilde{c}(a)P$,
  and the ideal endpoint $\tilde{z}$ of $\tilde{c}$ is
  $\tilde{c}(a)p$, with $p$ now viewed as a point in $\Pi$.

  Let $\{(v_i, \alpha_i)\}_{i=1}^{n+1}$ be a $\mc{G}$-path limiting to
  $\tilde{z}$, so that $v_{n+1}$ is a parabolic vertex. Let
  $q \in \bgamp$ be the parabolic point corresponding to
  $v_{n+1}$. Since the $\mc{G}$-path limits to a parabolic point in
  the orbit of $p$, $q$ must also be in this orbit, so write
  $q = g' p$ for $g' \in \Gamma$. The fact that the $\mc{G}$-path
  limits to $\tilde{z}$ means precisely that
  \begin{align*}
    \alpha_1 \cdots \alpha_nq &= \tilde{z},
  \end{align*}
  and thus $\tilde{z} = \alpha_1 \cdots \alpha_n g'p$. This means that
  $\alpha_1 \cdots \alpha_ng'$ lies in the coset $\tilde{c}(a)P$, so
  write $\alpha_1 \cdots \alpha_ng'h = \tilde{c}(a)$ for $h \in P$.

  Then, by definition of $\psi_\sigma$, we have
  $\psi_\sigma(z) = \sigma(c(a))\psi_\sigma(p)$, and since $\tilde{c}$
  lifts $c$ this means
  \begin{align*}
    \psi_\sigma(z) &= \sigma(\alpha_1 \cdots
                     \alpha_n)\sigma(g'h)\psi_\sigma(p)\\
                   &= \sigma(\alpha_1 \cdots
                     \alpha_n)\sigma(g')\psi_\sigma(p).
  \end{align*}
  Now, we know that the sequence of elements
  $g_j = \alpha_1 \cdots \alpha_j$ for $1 \le j \le n+1$ lies within
  Hausdorff distance $D$ of $\tilde{c}([0, \infty)) \cap \cay(\Gamma)$
  (see \Cref{defn:codings_quasigeodesic_nbhds}). So, the $\mc{G}$-path
  satisfies the $\mc{G}$-compatibility property in
  \Cref{lem:perturbed_nesting}. And, by
  \Cref{lem:automaton_local_parabolic}, we have
  $\sigma(g')\psi_\sigma(p) \subset U_{v_{n+1}}$ and the desired
  inclusion follows.
\end{proof}
\begin{lem}
  \label{lem:fibers_antipodal}
  For any distinct $z_1, z_2 \in \biquot$, we have
  $\psi_\sigma(z_2) \subset \Opp(\psi_\sigma(z_1))$. In particular,
  $\psi_\sigma(z_1)$ and $\psi_\sigma(z_2)$ are disjoint.
\end{lem}
\begin{proof}
  By \Cref{lem:preferred_lift_separation} and $\sigma$-equivariance of
  $\psi_\sigma$, it suffices to show the statement in the case where
  there are geodesic lifts $\tilde{z}_1$, $\tilde{z}_2$ of $z_1$ and
  $z_2$ satisfying $\bdist(\tilde{z}_1, \tilde{z}_2) > \Delta$. Fix
  $\mc{G}$-paths $\{(v_i, \alpha_i)\}_{i=1}^N$ and
  $\{(w_j, \beta_j)\}_{j=1}^M$ limiting to $\tilde{z}_1$ and
  $\tilde{z}_2$ respectively. Then by property
  \ref{item:distinguish_points}, we have
  $\overline{U_{w_1}} \subset
  \Opp(U_{v_1})$. \Cref{lem:gpath_determines_nbhd} implies that
  $\psi_\sigma(z_1) \subset U_{v_1}$ and
  $\psi_\sigma(z_2) \subset U_{w_1}$, so the conclusion follows.
\end{proof}

\begin{definition}
  Define the map $\phi_\sigma:\Lambda_\sigma \to \bquot$ by the rule:
  \[
    \phi_\sigma(w) = z \iff w \in \psi_\sigma(z).
  \]
  \Cref{lem:fibers_antipodal} ensures that this determines a
  well-defined antipodal map. It is surjective by construction, and it
  is equivariant because $\psi_\sigma$ is equivariant.
\end{definition}

\subsection{Compactness and continuity}

We have defined a map $\phi_\sigma:\Lambda_\sigma \to \biquot$, which
is our candidate boundary extension for an extended convergence action
of $\Gamma^\sigma$ on $\flags$. The next step is to show that the
domain $\Lambda_\sigma$ of $\phi_\sigma$ is compact, and that
$\phi_\sigma$ is continuous. We will need to directly get our hands on
the topologies on both Bowditch boundaries $\bgamp$ and $\biquot$; our
approach uses the Gromov product $(\cdot|\cdot)_x$ on the compactified
cusped space $\overline{X^\sigma}$. First, we prove a general lemma:
\begin{lem}
  \label{lem:quasigeodesic_prefix}
  Let $Y$ be a $\delta$-hyperbolic metric space and let $x \in Y$. For
  any infinite geodesic ray $c:[0, \infty) \to Y$ with $c(0) = x$ and
  $c(\infty) = z \in \dee Y$, if $z' \in \dee Y$ satisfies
  $(z|z')_x \ge N + 5\delta$, then there is a
  $(1, 18\delta)$-quasi-geodesic ray $c':[0, \infty) \to X$ with
  $c'(\infty) = z'$ and $c'|_{[0,N]} = c|_{[0,N]}$.
\end{lem}
\begin{proof}
  Denote the image of the geodesic ray $c$ by $[x, z)$, and let
  $[x, z')$ be a geodesic ray from $x$ to $z'$. Fix sequences $y_n$ on
  $[x, z)$ and $y_m'$ on $[x, z')$ tending to infinity. Then the
  Gromov product $(z|z')_x$ can be coarsely estimated as
  $\liminf_{m,n \to \infty}(y_n|y_m')_x$. Explicitly, we have the
  estimate
  $\liminf_{n,m \to \infty}(y_n|y_m')_x \ge (z|z')_x - 2\delta$ (see
  e.g. \cite[III.H.3.17(5)]{bridsonhaefliger}), which means that for
  all sufficiently large $n, m$, we have
  \[
    (z|z')_x \le (y_n|y_m')_x + 3\delta.
  \]
  Now, consider a geodesic triangle with two sides equal to
  subsegments $[x, y_n] \subset [x, z)$ and
  $[x, y_m') \subset [x, z')$. There is a point $w$ on the side
  joining $y_n$ to $y_m'$ which lies within distance $\delta$ of both
  some $u \in [x, y_n]$ and some $u' \in [x, y_n']$. Thus we have
  \begin{align*}
    d(y_n, u) - d(y_n, w) \le \delta,\\
    d(y_m', u') - d(y_n', w) \le \delta.
  \end{align*}
  Moreover we also have $d(u, u') \le 2\delta$, hence
  \[
    d(x, u') \le d(x, u) + 2\delta.
  \]
  Therefore,
  \begin{align*}
    2(y_n|y_m')_x &= d(x, y_n) + d(x, y_m') - d(y_n,
                    y_n')\\
                  &= d(x, u) + d(u, y_n) + d(x, u') + d(u',
                    y_m') - d(y_n, w) - d(w, y_m')\\
                  &\le 2d(x, u) + 4\delta.
  \end{align*}
  Thus $d(x, u) \ge (z|z')_x - 5\delta$. Now let $[u, z')$ be a
  geodesic ray from $u$ to $z'$, and consider the ray $c'$ in $Y$
  obtained by concatenating $[x, u]$ to $[u, z')$. We claim that $c'$
  is a $(1, 18\delta)$-quasi-geodesic.

  To see this, let $a \in [x, u]$ and let $b \in [u, z')$. Then since
  $d(u, u') \le 2\delta$, by considering a $\delta$-thin triangle with
  vertices $x, u, u'$, we see that there is a point $a'$ on $[x, u']$
  with $d(a, a') \le 3\delta$. Similarly, by considering a
  $2\delta$-thin (partly ideal) triangle with vertices $u, u', z'$,
  there is a point $b' \in [u', z')$ with $d(b, b') \le
  4\delta$. Therefore
  \[
    d(a, u) + d(u, b) \le d(a', u') + d(u', b') + (3 + 2 + 2 +
    4)\delta.
  \]
  Since $u'$ lies on a geodesic between $a'$ and $b'$, this says that
  the sub-path of $c'$ between $a$ and $b$ is at most
  \[
    d(a', b') + 11\delta \le d(a, b) + 18\delta.
  \]
  Thus $c'$ is a $(1, 18\delta)$-quasi-geodesic which agrees with $c$
  on a segment of length at least $(z|z')_x - 5\delta$.
\end{proof}

The lemma above gives us a tool for relating convergence in $\bgamp$
to convergence in $\biquot$.
\begin{cor}
  \label{cor:quasigeodesic_lifts_converge}
  Let $z$ be a point in $\dee X^\sigma$, let $z_n$ be a sequence
  converging to $z$, and let $\tilde{z}$ be a geodesic lift of
  $z$. There exist $(1, 18\delta)$-lifts $\tilde{z}_n$ of $z_n$ so
  that $\tilde{z}_n$ converges to $\tilde{z}$.
\end{cor}
\begin{proof}
  Let $c$ be a geodesic ray from $\identity$ to $z$ and let
  $\tilde{c}$ be a lift of $c$ from $\identity$ to $\tilde{z}$. Using
  the previous lemma, we can find $N_n \to \infty$ and
  $(1, 18\delta)$-quasi-geodesic rays $c_n$ so that
  $c_n(\infty) = z_n$, and $c_n$ agrees with $c$ on the interval
  $[0, N_n]$. So, we can choose lifts $\tilde{c}_n$ of $c_n$ so that
  $\tilde{c}_n$ agrees with $\tilde{c}$ on $[0, N_n]$ as well. The
  ideal endpoints of these lifts necessarily converge to $\tilde{z}$.
\end{proof}

We can now show:
\begin{prop}
  The set $\Lambda_\sigma$ is compact.
\end{prop}
\begin{proof}
  The proof is similar to the proof of Lemma 9.14 in
  \cite{Weisman2022}. Let $y_n$ be a sequence in $\Lambda_\sigma$, and
  let $x_n = \phi_\sigma(y_n)$. Up to subsequence $x_n$ converges to
  $x \in \biquot$. We want to show that (possibly after further
  extraction) $y_n$ converges to a point in $\phi_\sigma^{-1}(x)$. We
  consider two cases:
  \begin{description}
  \item[Case 1: $x$ is a parabolic point] Write $x = gp$ for
    $g \in \repquot$ and $p \in \Pi$, and let $x_n' = g^{-1}x_n$. By
    $\sigma$-equivariance, it suffices to show that
    $y_n' = \sigma(g^{-1})y_n$ converges to a point in in
    $\phi_\sigma^{-1}(p)$.

    Let $P$ be the stabilizer of $p$ in $\Gamma$, and let
    $P^\sigma = \sigma(P)$. By \Cref{lem:lift_cocompactness}, for each
    $n$ we can find $h_n \in P^\sigma$ so that a geodesic lift
    $\tilde{x}_n'$ of $h_nx_n'$ lies in the compact set
    $K_p \subset \bgamp \minus \{p\}$ in \Cref{defn:fund_domains}.
    Let $v$ be a parabolic vertex of the automaton $\mc{G}$ such that
    $q_v = p$.

    We know that our automaton $\mc{G}$ satisfies condition
    \ref{item:parabolic_domain_path} in
    \Cref{prop:parabolic_approx_path}, with $K_p$ as above. So, there
    is an edge $v \to w$ in $\mc{G}$ and a $\mc{G}$-path limiting to
    $\tilde{x}_n'$ whose first vertex is
    $w$. \Cref{lem:gpath_determines_nbhd} implies that
    $\phi_\sigma^{-1}(h_nx_n')$ lies in $U_w$, therefore in the set
    $B_v$ defined in \Cref{sec:parabolic_fibers}. So, for every $n$,
    we have $\sigma(h_n)y_n' \in B_v$.

    The proof of \Cref{lem:lift_cocompactness} shows that for every
    $n$, there is a bi-infinite geodesic through $\mathrm{id}$ in
    $X^\sigma$ joining $h_nx_n'$ to $p$, so $h_nx_n'$ cannot converge
    to $p$. Since $x_n'$ converges to $p$, a subsequence of $h_n^{-1}$
    must consist of pairwise distinct elements. So, by definition of
    the set $\phi_\sigma^{-1}(p) = \psi_\sigma(p)$, the sequence
    $y_n' = \sigma(h_n^{-1}) \sigma(h_n)y_n'$ accumulates in
    $\phi_\sigma^{-1}(p)$, hence in $\Lambda_\sigma$.
  \item[Case 2: $x$ is a conical limit point] Fix a geodesic ray
    $c:[0, \infty) \to X^\sigma$ limiting to $x$, so that the ideal
    endpoint $\tilde{x}$ of a lift $\tilde{c}$ of $c$ is a geodesic
    lift of $x$. Let $\{(v_i, \alpha_i)\}_{i=1}^\infty$ be a
    $\mc{G}$-path limiting to $\tilde{x}$. By
    \Cref{cor:quasigeodesic_lifts_converge}, there are
    $(1, 18\delta)$-lifts $\tilde{x}_n$ of $x_n$ so that $\tilde{x}_n$
    converges to $\tilde{x}$.

    Let $r > 0$ be given. As in the proof of
    \Cref{lem:coned_unbounded_divergent}, we can employ
    \Cref{lem:perturbed_nesting}, property
    \ref{item:flag_sets_proper}, and \cite[Cor.  7.12]{Weisman2022} to
    see that there is a constant $N$ and a metric $d_{U_{v_1}}$ on an
    open subset $U_{v_1}'$ of $\flags$ containing $\overline{U_{v_1}}$
    such that, for any $\mc{G}$-path $\{(w_j, \beta_j)\}_{j=1}^m$ with
    $w_1 = v_1$ and $m \ge N$, the diameter of
    \[
      \sigma(\beta_1 \cdots \beta_m) U_{w_{m+1}}
    \]
    with respect to the metric $d_{U_{v_1}}$ is less than $r$. It is
    also true that this metric $d_{U_{v_1}}$ induces the subspace
    topology on $U_{v_1} \subset \flags$.

    Now, by \cite[Lem. 5.15]{Weisman2022}, for all sufficiently large
    $n$, we can find a $\mc{G}$-path
    \[
      \{(w_j^{(n)}, \beta_j^{(n)})\}_{j=1}^{M_n}
    \]
    limiting to $\tilde{x}_n$ such that $M_n \ge N$ and
    $w_i^{(n)} = v_i$, $\beta_i^{(n)} = \alpha_i$ for all $i \le
    N$. Then, by \Cref{lem:gpath_determines_nbhd}, we have
    \[
      \phi_\sigma^{-1}(x_n) \subset \sigma(\alpha_1 \cdots \alpha_N)
      U_{v_{N+1}}.
    \]
    Since $\phi_\sigma^{-1}(x)$ must also lie in this set by
    definition, we must have $d_{U_{v_1}}(y_n, y) < r$ for all
    sufficiently large $n$. Since $r$ was arbitrary this completes the
    proof.
  \end{description}
\end{proof}

\begin{prop}
  The map $\phi_\sigma:\Lambda_\sigma \to \bquot$ is continuous.
\end{prop}
\begin{proof}
  We follow the proof of Lemma 9.15 in \cite{Weisman2022}. Let
  $y \in \Lambda_\sigma$ and let $y_n$ approach $y$. We want to show
  that $\phi_\sigma(y_n)$ approaches $\phi_\sigma(y) = x$. Suppose
  otherwise. Then since $\bigamp$ is compact, up to subsequence
  $z_n = \phi_\sigma(y_n)$ approaches some $z \ne x$. By
  \Cref{lem:preferred_lift_separation}, there is some $g \in \repquot$
  so that for geodesic lifts $\tilde{x}$ of $gx$ and $\tilde{z}$ of
  $gz$, we have $d(\tilde{x}, \tilde{z}) > \Delta$. Since $gz_n$
  approaches $gz$, \Cref{cor:quasigeodesic_lifts_converge} tells us
  that we can find $(1, 18\delta)$-lifts $\tilde{z}_n$ of $gz_n$
  converging to $\tilde{z}$. So, for sufficiently large $n$, we have
  $d(\tilde{x}, \tilde{z}_n) > \Delta$.

  Property \ref{item:distinguish_points} of the automaton then means
  that for any $\mc{G}$-paths $\{(v_i, \alpha_i)\}_{i=1}^N$ and
  $\{(w_j, \beta_j)\}_{j=1}^M$ limiting to $\tilde{x}$ and
  $\tilde{z}_n$ respectively, we have
  $\overline{U}_{v_1} \cap \overline{U}_{w_1} =
  \emptyset$. \Cref{lem:gpath_determines_nbhd} implies that
  $\phi_\sigma^{-1}(x) \subset U_{v_1}$ and
  $\phi_\sigma^{-1}(z_n) \subset U_{w_1}$. Since
  $y \in \phi_\sigma^{-1}(x)$ and
  $y_n \in \phi_\sigma^{-1}(z_n) \subset U_{w_1}$, this contradicts
  the fact that $y_n \to y$.
\end{proof}

\subsection{Defining the repelling strata}

To finish showing that $\sigma$ is an EGF representation, we need to
construct the repelling stratum $C_z^\sigma$ for each point
$z \in \biquot$. We do this separately for conical points and
parabolic points:
\begin{itemize}
\item If $z$ is a conical limit point, set
  $C_z^\sigma = \Opp(\phi_\sigma^{-1}(z))$.
\item If $z$ is a parabolic point, equal to $gp$ for
  $p \in \Pi^\sigma$, $g \in \repquot$, then let $P^\sigma$ be the
  stabilizer of $p$ in $\repquot$. Let $v$ be the parabolic vertex of
  $\mc{G}$ with $q_v = p$, and set
  \[
    C_z^\sigma = \Opp(\phi_\sigma^{-1}(z)) \cap \sigma(gP^\sigma) B_v,
  \]
  where $B_v$ is the open subset defined in
  \Cref{sec:parabolic_fibers}.
\end{itemize}

The next lemma says that these strata satisfy property
\ref{item:repelling_strata_containment} from
\Cref{sec:egf_background}.

\begin{lem}
  \label{lem:weak_repelling_stratum}
  For each $z \in \biquot$, we have
  $\Lambda_\sigma \minus \phi_\sigma^{-1}(z) \subset C_z^\sigma$.
\end{lem}
\begin{proof}
  If $z$ is a conical limit point, then the claim follows directly
  from the antipodality of $\phi_\sigma$. So we can suppose that $z$
  is a parabolic point. Moreover, if $z = gp$ for $g \in \repquot$ and
  $p \in \Pi^\sigma$, then $C_z^\sigma = \sigma(g)C_p^\sigma$, so by
  $\sigma$-equivariance of $\phi_\sigma$ we can assume that
  $z = p \in \Pi^\sigma$. Let $w$ be a vertex in $\mc{G}$ with
  $q_w = p$ (see property \ref{item:parabolic_edge_condition1}).

  Let $y \in \biquot \minus \{p\}$. By \Cref{lem:lift_cocompactness},
  we can find $h \in P^\sigma$ so that a geodesic lift $\tilde{y}$ of
  $h y$ lies in the fixed compact set
  $K_p \subset \bgamp \minus \{p\}$. By property
  \ref{item:parabolic_domain_path}, there is a $\mc{G}$-path limiting
  to $\tilde{y}$ whose first vertex $v_1$ is connected to $w$ by an
  edge $w \to v_1$ in $\mc{G}$.

  \Cref{lem:gpath_determines_nbhd} tells us that
  $\phi_\sigma^{-1}(hy) \subset U_{v_1}$, and we know that
  $U_{v_1} \subset B_w$ by definition. So, since
  $\phi_\sigma^{-1}(hy)$ is also in $\Opp(\phi_\sigma^{-1}(p))$ by
  antipodality, we conclude that
  \[
    \phi_\sigma^{-1}(hy) \subset C_p^\sigma.
  \]
  Since $C_p^\sigma$ is $P^\sigma$-invariant this holds for
  $\phi_\sigma^{-1}(y)$ as well.
\end{proof}

Next we want to prove that the strata also satisfy property
\ref{item:extended_convergence_prop}, but first we need:
\begin{lemma}
  \label{lem:vertex_intersections_nonempty}
  For every vertex $v \in V(\mc{G})$, the set
  $\Lambda_\sigma \cap U_v$ is nonempty.
\end{lemma}
\begin{proof}
  Fix a vertex $v = v_1$ in $V(\mc{G})$. Property
  \ref{item:outgoing_edge} says that every vertex of $\mc{G}$ has at
  least one outgoing edge, so there is an infinite $\mc{G}$-path
  $\{(v_i, \alpha_i)\}_{i=1}^\infty$ whose first vertex is $v_1$. We
  can construct this $\mc{G}$-path by choosing each element
  $\alpha_i \in T_{v_i}$ to be minimal with respect to
  $\lcusp{\cdot}$, which means that $\lcusp{\alpha_i} \le C_1$ for
  every index $i$ (by definition of $C_1$).

  By \Cref{lem:infinite_codes_conical}, this $\mc{G}$-path limits to
  some conical point $\tilde{z} \in \bgamp$. Fix a geodesic
  $\tilde{c}:[0, \infty) \to X$ from $\identity$ to $\tilde{z}$, so
  that (from \Cref{defn:codings_quasigeodesic_nbhds}) the sequence
  $g_n = \alpha_1 \cdots \alpha_n$ has Hausdorff distance at most $D$
  from $\tilde{c}([0, \infty)) \cap \cay(\Gamma)$. Then, by
  \Cref{lem:short_jumps_close_to_cay}, the geodesic $\tilde{c}$ lies
  in a $(C_1 + 3D)$-neighborhood of $\cay(\Gamma)$.

  From property \ref{item:perturbed_kernel_bounds}, the image $c$ of
  $\tilde{c}$ in $X^\sigma$ is a $(2, 18\delta)$-quasi-geodesic. Thus,
  if $z = c(\infty)$, the point $\tilde{z}$ is a $(2, 18\delta)$-lift
  of $z$ and $\phi_\sigma^{-1}(z) \subset U_{v_1}$ by
  \Cref{lem:gpath_determines_nbhd}.
\end{proof}

The next lemma shows that, \emph{along parabolic sequences}, the
strata satisfy condition \ref{item:extended_convergence_prop}:
\begin{lemma}
  \label{lem:parabolic_convergence}
  Let $q \in \biquot$ be parabolic, let $h_n$ be a sequence of
  pairwise distinct elements in $\repquot_p = \Stab_{\repquot}(p)$,
  and let $K$ be a compact subset of $C_q^\sigma$. Then $\sigma(h_n)K$
  eventually lies in any open neighborhood $\phi_\sigma^{-1}(p)$.
\end{lemma}
\begin{proof}
  By the $\sigma$-equivariance of $\phi_\sigma$ and the repelling
  strata $C_z^\sigma$, we only need to consider the case where
  $q = p \in \Pi^\sigma$. Let $v$ be a parabolic vertex of $\mc{G}$ so
  that $q_v = p$. Since $K$ is a compact subset of $C_p^\sigma$, it is
  covered by finitely many sets of the form $\sigma(h) B_w$, where
  $h \in \repquot_p$, and $w$ is a vertex of $\mc{G}$ with $v \to w$
  an edge in $\mc{G}$. But then, for any sequence $\xi_n \in K$, any
  accumulation point of $\sigma(h_n)\xi_n$ lies in
  $\phi_\sigma^{-1}(p)$, by definition of $\phi_\sigma$. Thus
  $\sigma(h_n)K$ must eventually lie in any neighborhood of
  $\phi_\sigma^{-1}(p)$.
\end{proof}

Now we can show that condition \ref{item:extended_convergence_prop}
holds in general:
\begin{lemma}
  \label{lem:extended_convergence_holds}
  Suppose that a sequence $\gamma_n \in \repquot$ satisfies
  $\gamma_n^{\pm1} \to z_\pm$ for $z_\pm \in \bigamp$, and let
  $K \subset C^\sigma_{z_-}$ be compact. Then $\sigma(\gamma_n)K$
  eventually lies in any neighborhood of $\phi_\sigma^{-1}(z_+)$.
\end{lemma}
\begin{proof}
  Consider the sequence of relative lengths $\lcons{\gamma_n}$. After
  taking a subsequence, these either tend to infinity, or else they
  are bounded.

  \begin{description}
  \item[Case 1: $\lcons{\gamma_n} \to \infty$] In this case, by
    applying \Cref{lem:coned_unbounded_divergent} to the constant
    sequence $\sigma$, the sequence $\sigma(\gamma_n)$ is
    $Q$-divergent. Further, because of the previous lemma,
    \Cref{lem:coned_unbounded_divergent} also implies that the
    $Q$-limit points of $\sigma(\gamma_n^\pm)$ all lie in
    $\Lambda_\sigma$. So, after extraction, the sequences
    $\sigma(\gamma_n^{\pm 1})$ respectively have unique $Q$-limit
    points $\xi_\pm \in \Lambda_\sigma$. Let
    $z_\pm' = \phi_\sigma(\xi_\pm)$. Since
    $C_z \subseteq \Opp(\phi_\sigma^{-1}(z))$ for every
    $z \in \biquot$, it follows from
    \Cref{lem:q_convergence_sequences} that, if $K$ is any compact
    subset of $C^\sigma_{z_-'}$, then $\sigma(\gamma_n)K$ eventually
    lies in any neighborhood of $\phi_\sigma^{-1}(z_+')$. So we will
    be done if we can show $z_\pm = z_\pm'$.
    
    However, since $(\repquot, \iquotgps)$ is non-elementary, we can
    pick a compact set $K \subset C_{z_-'}$ so that
    $\phi_\sigma(K \cap \Lambda_\sigma)$ is a nonempty open subset $Y$
    of $\biquot$. Since $\phi_\sigma$ is equivariant and continuous,
    $\gamma_nY$ converges to $z_+'$, and since $\gamma_n$ converges to
    $z_+$ in $\overline{X^\sigma}$, this is only possible if
    $z_+ = z_+'$. An identical argument applied to $\gamma_n^{-1}$ shows
    that $z_- = z_-'$, and we can conclude this case.
  \item[Case 2: $\lcons{\gamma_n}$ is bounded] For this case, we argue
    as in the proof of \Cref{lem:uniformly_proper}. Extract a
    subsequence so that
    \[
      \gamma_n = g_1h_1^{(n)} \cdots g_kh_k^{(n)}g_{k+1},
    \]
    where each $g_i \in \repquot$ is fixed, $h_i^{(n)}$ is a sequence
    of pairwise distinct elements fixing some
    $p_i \in \Pi^\sigma_\infty$, and $g_ip_i \ne p_{i-1}$ for every
    $1 < i \le k$. Let $z_+' = g_1p_1$ and let
    $z_-' = g_{k+1}^{-1}p_k$. By applying
    \Cref{lem:parabolic_convergence} and inducting on $k$, we can see
    that for any compact set $K \subset C^\sigma_{z_-'}$, the set
    $\sigma(\gamma_n)K$ eventually lies in any neighborhood of
    $\phi_\sigma^{-1}(z_+')$. Once again, we can show that
    $z_+ = z_+'$ by choosing $K$ so that its intersection with
    $\Lambda_\sigma$ projects to a nonempty open subset of $\biquot$,
    and again the same argument applied to $\gamma_n^{-1}$ shows that
    $z_-' = z_-$. This completes the proof of the parabolic case.
  \end{description}
\end{proof}

Finally we can prove:
\begin{proposition}
  The representation $\sigma$ is $Q$-extended geometrically finite,
  with boundary extension $\phi_\sigma:\Lambda_\sigma \to \biquot$ and
  repelling strata $C_z^\sigma$ defined above.
\end{proposition}
\begin{proof}
  We have already shown that the map $\phi_\sigma$ is continuous,
  equivariant, surjective, and antipodal, and we know from
  \Cref{lem:weak_repelling_stratum} and
  \Cref{lem:extended_convergence_holds} that conditions
  \ref{item:repelling_strata_containment} and
  \ref{item:extended_convergence_prop} in the definition of an
  extended convergence action are satisfied. From
  \Cref{lem:coned_unbounded_divergent}, $\sigma$ is relatively
  $Q$-divergent, so \Cref{prop:weak_egf_to_egf} applies and $\sigma$
  must be $Q$-EGF.

  If we only consider the statement of \Cref{prop:weak_egf_to_egf}, it
  might seem possible that the $Q$-EGF boundary extension and
  repelling strata could be different from $\phi_\sigma$ and
  $C_z^\sigma$. However, examining the construction in the proof of
  \Cref{prop:weak_egf_to_egf} (see \Cref{sec:weak_egf_appendix}) shows
  that $\phi_\sigma$ and $C_z^\sigma$ will not be modified, so they
  serve as an actual boundary extension and repelling strata for
  $\sigma$.
\end{proof}

\subsection{Semicontinuity of the boundary sets}

We have shown the main part of \Cref{prop:main_prop}. We complete the
proof by observing:
\begin{proposition}
  The set $\Lambda_\sigma$ is contained in the $\eps$-neighborhood of
  $\Lambda$, and has nonempty intersection with
  $\nbhd{\phi^{-1}(Z)}{\eps}$.
\end{proposition}
\begin{proof}
  From property \ref{item:g_sets_small}, each vertex set of the
  automaton is contained in an $\eps$-neighborhood of $\Lambda$, and
  by \Cref{lem:gpath_determines_nbhd}, each set $\phi_\sigma^{-1}(z)$
  for $z \in \biquot$ lies in some vertex set $U_v$, so this proves
  the first part of the proposition.

  For the second part, we use property
  \ref{item:approximate_compact_subset} to see that there is at least
  one vertex set $U_v$ contained in the $\eps$-neighborhood of
  $\phi^{-1}(Z)$. From \Cref{lem:vertex_intersections_nonempty},
  $\Lambda_\sigma$ has nonempty intersection with this set.
\end{proof}

\begin{remark}
  \label{rem:bdry_limitset}
  In the special case where the subgroups $\sigma(P)$ are
  $Q$-divergent for all $P \in \pargps$, the construction of the limit
  set $\Lambda_\sigma$ in the proof above ensures that, for each
  $p \in \Pi$, the set $\phi_\sigma^{-1}(p)$ is precisely the
  $Q$-limit set of the subgroup $\sigma(\Gamma_p)$. Thus, in this
  case, the boundary set $\Lambda_\sigma$ is precisely the $Q$-limit
  set of $\sigma(\Gamma)$.
\end{remark}

\section{Relative Anosov representations and continuity of the limit
  set}
\label{sec:relative_anosov}

In this section we prove \Cref{thm:rel_anosov_fillings}, which applies
our main theorem to the special case of \emph{relatively $Q$-Anosov
  representations} defined in work of Kapovich--Leeb
\cite{kl2018relativizing} and Zhu \cite{Z2021}; see also
\cite{CZZ}. There are several characterizations of relative Anosov
representations (see e.g. \cite{zz1}), but we will work with the one
given by the proposition below. Here, and for the rest of the section,
$(\Gamma, \pargps)$ is again a relatively hyperbolic pair, and $Q$ is
a symmetric parabolic subgroup of a semisimple Lie group $G$.
\begin{proposition}[{See \cite[Sec. 4]{Weisman2022}}]
  \label{prop:rel_anosov_characterize}
  Let $\rho:\Gamma \to G$ be a representation. The following are
  equivalent:
  \begin{enumerate}[label=(\roman*)]
  \item $\rho$ is relatively $Q$-Anosov (in the sense of \cite{zz1}).
  \item\label{item:rel_anosov_homeo} $\rho$ is $Q$-EGF, and has a
    boundary extension which is a homeomorphism.
  \item\label{item:rel_anosov_parabolic_singleton} $\rho$ is $Q$-EGF,
    and has a boundary extension $\phi$ such that $\phi^{-1}(p)$ is a
    singleton for each $p \in \Pi$.
  \end{enumerate}
  Moreover, if $\phi:\Lambda \to \bgamp$ is the boundary extension in
  \ref{item:rel_anosov_homeo} or
  \ref{item:rel_anosov_parabolic_singleton}, then $\Lambda$ is the
  $Q$-limit set of $\rho$.
\end{proposition}

\begin{remark}
  When the set of peripheral subgroups $\pargps$ is empty, then all
  three of these conditions agree with the definition of a
  (non-relatively) \emph{$Q$-Anosov representation} of the hyperbolic
  group $\Gamma$ into $G$.
\end{remark}

We prove the following, which is a stronger form of
\Cref{thm:rel_anosov_fillings}:
\begin{proposition}
  \label{prop:rel_anosov_fillings}
  Let $\rho$ a relatively $Q$-Anosov representation with $Q$-limit set
  $\Lambda$, and let $W \subseteq \ghom(\Gamma, G; \pargps)$ be an
  extended Dehn filling space about $\rho$. Let $\rho_n \in W$ be a
  sequence of representations converging to $\rho$. Then:
  \begin{enumerate}[label=(\arabic*)]
  \item \label{item:rel_anosov_strong} The sequence $\rho_n$ converges
    strongly to $\rho$, i.e. in $\ghom(\Gamma, G)$.
  \item \label{item:rel_anosov_egf} For sufficiently large $n$,
    $\rho_n:\Gamma^{\rho_n} \to G$ is $Q$-EGF.
  \item\label{item:bdry_sets_converge} There is a sequence of boundary
    sets $\Lambda_n$ for $\rho_n$ converging to $\Lambda$ with respect
    to Hausdorff distance.
  \end{enumerate}
  Moreover, if in addition each $\rho_n$ is $Q$-divergent, the
  boundary sets in part \ref{item:bdry_sets_converge} can all be taken
  to be the $Q$-limit set of $\rho_n$.
\end{proposition}
\begin{proof}
  By \Cref{prop:rel_anosov_characterize}, $\rho$ is $Q$-EGF, with a
  homeomorphic boundary extension $\phi:\Lambda \to \bgamp$. Parts
  \ref{item:rel_anosov_strong} and \ref{item:rel_anosov_egf} are
  immediate consequences of \Cref{prop:main_prop} and
  \Cref{prop:relative_strong_to_strong}. Now, fix $\eps > 0$. We also
  know from \Cref{prop:main_prop} that the boundary sets $\Lambda_n$
  for $\rho_n$ eventually lie in the $\eps$-neighborhood of $\Lambda$,
  so for part \ref{item:bdry_sets_converge}, we just need to show that
  $\Lambda$ is eventually contained in the $\eps$-neighborhood of
  $\Lambda_n$.

  Since $\phi$ is a homeomorphism of compact sets, there is a finite
  family of compact subsets $Z_i$ covering $\bgamp$ such that the
  diameter of each $\phi^{-1}(Z_i)$ is at most $\eps/2$. By
  \Cref{prop:main_prop} again, for sufficiently large $n$, $\Lambda_n$
  has nonempty intersection with the $\eps/2$-neighborhood of each
  $\phi^{-1}(Z_i)$, which means an $\eps$-neighborhood of $\Lambda_n$
  contains $\phi^{-1}(Z_i)$. Since the $Z_i$ cover $\bgamp$, this
  proves part \ref{item:bdry_sets_converge}. The final part of the
  proposition is an immediate consequence of \Cref{rem:bdry_limitset}.
\end{proof}

\subsection{Filling to relative Anosov representations}
\label{sec:relativization}

Usually, it is straightforward to detect whether the Dehn fillings in
\Cref{thm:main_theorem} and \Cref{thm:rel_anosov_fillings} give rise
to (relatively) Anosov representations. The reason is the following
relativization theorem of Wang, which (combined with our main theorem)
immediately implies \Cref{cor:fill_to_anosov}:
\begin{theorem}[{\cite{wang}, see also \cite[Sec. 4]{Weisman2022}}]
  Let $\rho:\Gamma \to G$ be a $Q$-EGF representation. Then $\rho$ is
  (relatively) $Q$-Anosov if and only if $\rho$ restricts to a
  (relatively) $Q$-Anosov representation on each $P \in \pargps$.
\end{theorem}

\section{The rank one case}
\label{sec:rank_one}

In this section we prove
\Cref{thm:rank_one_dehn_filling,thm:rank_one_relative_strong} from the
introduction. Below, let $G$ be a \textbf{rank one} semisimple Lie
group, and let $\X$ be the associated Riemannian symmetric space. We
say that a representation $\rho:\Gamma \to G$ of a finitely generated
group $\Gamma$ is \emph{geometrically finite} if it has finite kernel
and geometrically finite image. We connect geometrically finite
representations to EGF representations via \Cref{prop:egf_gf} below,
but first we introduce some terminology.
\begin{definition}
  \label{defn:horospherical_subgroups}
  A \emph{horospherical subgroup} of $H < G$ is any subgroup with
  noncompact closure in $G$ which preserves a horosphere in the
  symmetric space $\X$. If $\rho:\Gamma \to G$ is a representation,
  then the \emph{$\rho$-horospherical subgroups} are the maximal
  subgroups of $\Gamma$ mapping to horospherical subgroups of $G$.
\end{definition}

\begin{remark}
  Horospherical subgroups are more usually called ``parabolic
  subgroups'' in the context of both rank-one Lie groups and
  convergence groups. However, this terminology conflicts with the
  definition of a parabolic subgroup of a semisimple Lie group (which
  in this case is the \emph{full} stabilizer in $G$ of a point in
  $\dee \X$, including hyperbolic elements), so we avoid it.
\end{remark}

Due to work of Bowditch \cite{Bowditch2012} \cite{BowditchGF},
whenever $\rho:\Gamma \to G$ is geometrically finite, then $\Gamma$ is
relatively hyperbolic, relative to any maximal collection $\pargps$ of
pairwise non-conjugate $\rho$-horospherical subgroups. In this case,
since $\rho$ has finite kernel and discrete image, each group in
$\pargps$ is a finite extension of a virtually nilpotent group. Then
it follows from Gromov's theorem on groups with polynomial growth
\cite{gromov1981} that each group in $\pargps$ is itself virtually
nilpotent (see also \cite[Prop. 2.4]{TZ24} for a direct elementary
proof of this fact).

It is also true that in this situation, the representation $\rho$ is
EGF with respect to any choice of $\pargps$ as above. More precisely,
we have the following:
\begin{proposition}[{See \cite[Prop 2.8 and Prop. 2.11]{GW2024}}]
  \label{prop:egf_gf}
  Let $(\Gamma, \pargps)$ be a relatively hyperbolic pair, and let
  $\rho:\Gamma \to G$ be a representation. Suppose that every group in
  $\pargps$ is virtually nilpotent. Then the following are equivalent:
  \begin{enumerate}
  \item $\rho$ is geometrically finite, and $\pargps$ contains a
    conjugate of each $\rho$-horospherical subgroup of $\Gamma$.
  \item $\rho$ is EGF with respect to $\pargps$.
  \end{enumerate}
  In this situation, if $\Lambda_\rho \subseteq \dee \X$ is the limit
  set of $\rho(\Gamma)$, there is a unique EGF boundary extension
  $\phi:\Lambda_\rho \to \bgamp$. Moreover, if $\pargps$ only contains
  $\rho$-horospherical subgroups, this boundary extension is a
  homeomorphism.
\end{proposition}

Note that in the statement of the proposition, $\pargps$ could contain
subgroups which are \emph{not} $\rho$-horospherical; this matters
because (especially in the context of Dehn filling) we want to be able
to consider deformations $\rho'$ of $\rho$ where the
$\rho$-horospherical subgroups are \emph{not} in correspondence with
the $\rho'$-horospherical subgroups.

Our main task in this section is to prove the proposition below, which
immediately implies \Cref{thm:rank_one_dehn_filling}. Afterwards we
will prove \Cref{thm:rank_one_relative_strong} as a corollary.
\begin{proposition}
  \label{prop:relative_strong_implies_edf}
  Let $\rho:\Gamma \to G$ be a geometrically finite representation and
  let $\pargps$ be a maximal collection of pairwise non-conjugate
  $\rho$-horospherical subgroups (so in particular $(\Gamma, \pargps)$
  is a relatively hyperbolic pair and $\rho$ is EGF with respect to
  $\pargps$ by \Cref{prop:egf_gf}).

  Then, if $\rho_n:\Gamma \to G$ is a sequence of representations
  converging to $\rho$ in $\ghom(\Gamma, G; \pargps)$, the subset
  \[
    \{\rho\} \cup \{\rho_n\}_{n \in \N}
  \]
  is an extended Dehn filling space about $\rho$ (with respect to the
  peripheral structure $\pargps$ and the boundary extension
  $\phi:\Lambda_\rho \to \bgamp$ given in \Cref{prop:egf_gf}).
\end{proposition}
In the special case where each $\rho_n$ in the proposition is
faithful, then the result above is a direct consequence of Proposition
4.4 in \cite{GW2024}. In fact, very little work is needed to adapt the
proof of the special case to the general statement, since the key
lemma we need in both cases is exactly the same.

We state this key lemma below. To set this up, for any subset
$V \subseteq \overline{\X}$, let $\mc{K}(V)$ denote the set of
elements $g \in G$ which preserve some nonempty convex subset
$C \subset V$. Then we have the following:
\begin{lemma}[{see \cite[Lem. 4.6]{GW2024}}]
  \label{lem:rank_one_compact_closure}
  Let $p \in \dee \X$, and let $K, K'$ be compact subsets of $\dee \X$
  not containing $p$. Then there exists an open neighborhood $V$ of
  $p$ in $\overline{\X}$, disjoint from $K$ and $K'$, so that the set
  of elements $g \in \mc{K}(V)$ satisfying $gK \cap gK' \ne \emptyset$
  has compact closure in $G$.
\end{lemma}
\begin{remark}
  The cited lemma in \cite{GW2024} only states a special case of this
  result, but its proof applies verbatim to the general situation
  above.
\end{remark}

\begin{proof}[Proof of \Cref{prop:relative_strong_implies_edf}]
  We proceed by contradiction, and suppose that the set
  $\{\rho\} \cup \{\rho_n\}_{n \in \N}$ is \emph{not} an extended Dehn
  filling space. Since $\rho_n$ converges to $\rho$ in
  $\ghom(\Gamma, G; \pargps)$, and each $P \in \pargps$ is
  $\rho$-horospherical, this can only occur if there exists:
  \begin{itemize}
  \item A $\rho$-horospherical subgroup $P \in \pargps$, with unique
    fixed point $p \in \bgamp$,
  \item An open neighborhood $U$ of $\phi^{-1}(p)$ in $\dee \X$,
  \item A compact subset $K \subset \dee \X \minus U$,
  \item A finite set $F \subset P$, satisfying $\rho(h)K \subset U$
    for every $h \in P \minus F$, and
  \item A sequence of elements $h_n \in P$ with
    $\rho_n(h_n) \notin \rho_n(F)$,
  \end{itemize}
  such that $\rho_n(h_n)K \notin U$ for every $n$.
   
  For any subgroup $H < G$, let $\Lambda(H)$ denote the limit set of
  $H$ in $\dee \X$, let $\Fix(H)$ denote the fixed-point set of $H$ in
  $\overline{\X}$, and let $\Hull(H)$ denote the convex hull in
  $\overline{\X}$ of $\Lambda(H) \cup \Fix(H)$. Since each
  $P \in \pargps$ is virtually nilpotent, and the restrictions
  $\rho_n|_P$ converge to $\rho|_P$ in the compact-open topology on
  $\Hom(P, G)$, it follows that $\Hull(\rho_n(P))$ converges to
  $\Hull(\rho(P))$, with respect to Hausdorff distance on any fixed
  metrization of $\overline{\X}$ (see \cite[Lem. 4.5]{GW2024}). In
  particular, $\Hull(\rho_n(P))$ is nonempty for all large $n$, since
  $\Hull(\rho(P))$ is the singleton $\{p\}$.

  Now, applying \Cref{lem:rank_one_compact_closure}, there is an open
  neighborhood $V$ of $p$ in $\overline{\X}$ such that the set of
  elements $g \in \mc{K}(V)$ satisfying $gK \notin U$ is relatively
  compact in $G$. Then by the previous paragraph, $\Hull(\rho_n(P))$
  eventually lies in $V$, and since each $\Hull(\rho_n(P))$ is a
  nonempty $\rho_n(h_n)$-invariant convex set, it follows that the set
  of elements appearing in the sequence $\rho_n(h_n)$ is relatively
  compact in $G$. Thus, a subsequence of $\rho_n(h_n)$ converges to
  some $g \in G$.

  Since $\rho_n$ converges to $\rho$ in $\ghom(\Gamma, G; \pargps)$,
  the restrictions $\rho_n|_P$ converge strongly to the restriction
  $\rho|_P$, and so actually $g = \rho(h)$ for some $h \in P$. Since
  $\rho_n(h_n)K$ is not contained in $U$, we must also have
  $\rho(h)K \not\subset U$ and thus $h$ lies in the finite set $F$.

  We also know that $\rho_n(h)$ converges to $\rho(h)$, so
  $\rho_n(h_n^{-1}h)$ converges to the identity. Now, the set of
  discrete subgroups of $G$ is open in the space $\chab(G)$ of closed
  subgroups of $G$ with the Chabauty topology (see \cite[Prop
  3.4]{BHK}), which tells us that $\rho_n(P)$ is discrete for
  sufficiently large $n$. But in fact the proof of the cited
  proposition gives us a slightly stronger statement (see \cite[Prop
  4.1]{GW2024}), implying that the subgroups $\rho_n(P)$ are
  \emph{uniformly} discrete---there is a neighborhood $O$ of the
  identity in $G$ so that $\rho_n(P) \cap O = \{\identity\}$ for all
  sufficiently large $n$. This implies that $\rho_n(h_n) = \rho_n(h)$
  for all sufficiently large $n$, which contradicts the fact that
  $\rho_n(h_n) \notin \rho_n(F)$ for all $n$.
\end{proof}

\begin{proof}[Proof of \Cref{thm:rank_one_relative_strong}]
  The implication \ref{item:rank1_rel_strong} $\implies$
  \ref{item:rank1_edf} is the content of
  \Cref{prop:relative_strong_implies_edf}, and the implication
  \ref{item:rank1_edf} $\implies$ \ref{item:rank1_strong} follows from
  \Cref{thm:rel_anosov_fillings} since geometrically finite
  representations are always relatively Anosov. So we will be done if
  we show \ref{item:rank1_strong} $\implies$
  \ref{item:rank1_rel_strong}.

  For this, we can argue similarly to the proof of Proposition 4.4 in
  \cite{GW2024}. Assume that the sequence of subgroups
  $\rho_n(\Gamma)$ converges to the group $\rho(\Gamma)$ in the
  Chabauty topology. We wish to show that for any peripheral subgroup
  $P \in \pargps$, the subgroups $\rho_n(P)$ converge to $\rho(P)$ in
  the Chabauty topology. We employ
  \Cref{prop:chab_converge}. Algebraic convergence of representations
  passes to restrictions, which gives us \ref{item:chab_contain_1} in
  the proposition. So, we just need to show condition
  \ref{item:chab_contain_2}, meaning we need to show that if $h_n$ is
  a sequence in $P$, and $\rho_n(h_n)$ converges to $g \in G$, then
  $g \in \rho(P)$. So, suppose $\rho_n(h_n) \to g \in G$. Since we
  assume that $\rho_n \to \rho$ strongly, we know that
  $g \in \rho(\Gamma)$.

  We know that $P$ is virtually nilpotent, which means that
  $\rho_n(P)$ is as well. Arguing as in the proof of the previous
  proposition, this means that the $\rho_n$-invariant set
  $\Hull(\rho_n(P)) \subset \overline{\X}$ must converge to the
  singleton $\{z\}$, where $z \in \dee \X$ is the unique point fixed
  by $\rho(P)$. It follows that the limit $g$ of $\rho_n(h_n)$ must
  fix $z$. However, since $P$ is precisely the stabilizer of $z$ for
  the $\rho$-action of $\Gamma$ on $\dee \X$, we have $g \in \rho(P)$,
  as required.
\end{proof}

\appendix

\section{Relative $Q$-divergence and extended convergence actions}
\label{sec:weak_egf_appendix}

In this section of the paper, we prove
\Cref{prop:weak_egf_to_egf}. For convenience, we restate the
proposition below:

\begin{proposition}
  \label{prop:weak_egf_to_egf_app}
  Let $(\Gamma, \pargps)$ be a relatively hyperbolic pair, let $Q$ be
  a symmetric parabolic subgroup of a semisimple Lie group $G$, and
  let $\rho:\Gamma \to G$ be a representation. Suppose that there
  exists a compact $\rho$-invariant set $\Lambda \subset \flags$, an
  equivariant surjective antipodal map $\phi:\Lambda \to \bigamp$, and
  open sets $C_z \subset \flags$ for each $z \in \bigamp$ satisfying
  conditions \ref{item:repelling_strata_containment} and
  \ref{item:extended_convergence_prop} above.

  Then, if $\rho$ is relatively $Q$-divergent, there exists a
  (possibly different) equivariant surjective antipodal map
  $\hat{\phi}:\hat{\Lambda} \to \bigamp$ and open sets
  $\{\hat{C_z}\}_{z \in \bigamp}$ satisfying conditions
  \ref{item:repelling_strata_interior} and
  \ref{item:extended_convergence_prop} in
  \Cref{defn:extended_convergence_group}; in other words, $\rho$ is
  $Q$-EGF.
\end{proposition}

For the proof of the proposition, fix a representation
$\rho:\Gamma \to G$, a map $\phi:\Lambda \to \bigamp$, and open sets
$\{C_z\}_{z \in \bigamp}$ as in the statement, and assume that $\rho$
is relatively $Q$-divergent. We have the following basic observation.
\begin{lemma}
  \label{lem:q_pts_in_fiber}
  Suppose that $\gamma_n$ is a sequence in $\Gamma$ converging to some
  $z \in \bigamp$. If $\rho(\gamma_n)$ is $Q$-divergent, then every
  $Q$-limit point of $\rho(\gamma_n)$ lies in $\phi^{-1}(z)$.
\end{lemma}
\begin{proof}
  Let $\xi$ be a $Q$-limit point of $\rho(\gamma_n)$. Via
  \Cref{lem:q_convergence_sequences}, we can find a flag $\xi^-$ and
  extract a subsequence so that $\rho(\gamma_n)$ converges to the
  constant map $\xi$ uniformly on compact subsets of $\Opp(\xi^-)$. We
  can extract a further subsequence so that $\gamma_n^{-1}$ converges
  to some $z_- \in \bigamp$.

  Since $\Opp(\xi^-)$ is dense in $\flags$, it has nonempty
  intersection with the open set $C_{z_-}$. Let $K$ be a nonempty
  compact subset in this intersection. Then it follows from condition
  \ref{item:extended_convergence_prop} that $\rho(\gamma_n)K$
  eventually lies in an arbitrary small neighborhood of
  $\phi^{-1}(z)$; since $\rho(\gamma_n)K$ must converge to $\{\xi\}$
  this means that $\xi \in \phi^{-1}(z)$.
\end{proof}

Using this lemma, we show that one can shrink the set $\Lambda$ into a
certain nice form, while still preserving all of the assumptions. For
each conical limit point $z \in \bigamp$, define
$\Psi_z \subset \flags$ to be the union of the set of $Q$-limit points
of all sequences $\rho(\gamma_n)$, where $\gamma_n \to z$ in the
Bowditch compactification of $\Gamma$. (Note that since $z$ is
conical, any sequence $\gamma_n$ converging to $z$ must satisfy
$\lcon{\gamma_n} \to \infty$ and therefore by assumption
$\rho(\gamma_n)$ is $Q$-divergent and $\Psi_z$ is nonempty.)

Then, define a subset $\hat{\Lambda} \subset \flags$ via:
\[
  \hat{\Lambda} = \left(\bigcup_{\substack{z \in \bigamp,\\z \textrm{
          parabolic}}}\phi^{-1}(z)\right) \cup
  \left(\bigcup_{\substack{z \in \bigamp,\\z \textrm{ conical}}}
    \Psi_z\right).
\]

\begin{lemma}
  The set $\hat{\Lambda}$ is a compact $\rho$-invariant subset of
  $\Lambda$.
\end{lemma}
\begin{proof}
  That $\hat{\Lambda} \subset \Lambda$ is an immediate consequence of
  \Cref{lem:q_pts_in_fiber}. The $\rho$-invariance of $\hat{\Lambda}$
  follows from $\rho$-invariance of $\Lambda$, as well as the
  $\Gamma$-invariance of the set of conical limit points in
  $\bigamp$. So, we just need to show that $\hat{\Lambda}$ is compact.

  Consider a sequence of flags $\xi_n \in \hat{\Lambda}$, and extract
  a subsequence so that $\xi_n \to \xi$. We wish to show
  $\xi \in \hat{\Lambda}$. Since $\hat{\Lambda} \subset \Lambda$,
  there is a sequence of points $z_n \in \bigamp$ such that
  $\xi_n \in \phi^{-1}(z_n)$. By compactness of $\Lambda$ and
  continuity of $\phi$, we know $\xi \in \Lambda$ and
  $z_n \to z = \phi(\xi)$. If $z$ is parabolic, then
  $\phi^{-1}(z) \subset \hat{\Lambda}$, hence $\xi \in \hat{\Lambda}$,
  so assume $z$ is conical.

  Let $\gamma_m$ be a sequence in $\Gamma$ converging conically to
  $z$. By definition this means there are distinct points
  $u, v \in \bigamp$ such that $\gamma_m^{-1}z \to u$ and
  $\gamma_m^{-1}K \to \{v\}$ for all compact sets
  $K \subset \bigamp \minus \{z\}$; note that this implies
  $\gamma_m^{-1} \to v$ in the Bowditch compactification
  $\overline{\Gamma}$. Since $\gamma_m$ converges to a conical point,
  we must have $\lcon{\gamma_m} \to \infty$, so $\rho(\gamma_m)$ is
  $Q$-divergent. After extracting a subsequence, we can apply
  \Cref{lem:q_pts_in_fiber} and \Cref{lem:q_convergence_sequences} to
  see that there is a flag $\xi^- \in \phi^{-1}(v)$ such that
  $\rho(\gamma_n)$ converges to the constant map $\xi$, uniformly on
  compacts in $\Opp(\xi^-)$.

  Now, since $z_n$ converges to $z$, for each $n$ we can find an index
  $m_n$ so that $m_n \to \infty$, but $\gamma_{m_n}^{-1}z_n$ does
  \emph{not} accumulate at $v$. Then, since $\phi$ is antipodal and
  equivariant, this means that
  $\rho(\gamma_{m_n}^{-1})\xi_n \in \phi^{-1}(\gamma_{m_n}^{-1}z_n)$
  lies in a compact subset of $\Opp(\xi^-)$. Therefore
  $\xi_n = \rho(\gamma_{m_n})\rho(\gamma_{m_n}^{-1})\xi_n$ must
  converge to $\xi$, and we are done.
\end{proof}

\begin{remark}
  It turns out that actually each set $\Psi_z$ defined above must be a
  singleton (see \cite[Prop. 4.8]{Weisman2022}), but we will not show
  this at the moment.
\end{remark}

Restrict $\phi$ to obtain a new map
$\hat{\phi}:\hat{\Lambda} \to \bigamp$. This map is also equivariant,
surjective, and antipodal. Next, we define a new family $\hat{C}_z$ of
repelling strata. For each conical limit point $z \in \bigamp$, define
\[
  \hat{C}_z = \Opp(\hat{\phi}^{-1}(z)).
\]
Then, for each parabolic point $p \in \Pi$, let $P$ be the stabilizer
of $p$ in $\Gamma$, and define
\[
  \hat{C}_p = \Opp(\hat{\phi}^{-1}(p)) \cap \bigcup_{\gamma \in
    P}\rho(\gamma)C_p.
\]
Finally, for an arbitrary parabolic point $z \in \bigamp$, write
$z = gp$ for $p \in \Pi$, $g \in \Gamma$, and define
$\hat{C}_z = \rho(g)\hat{C}_p$.

The sets $\hat{C}_z$ are nonempty open sets satisfying condition
\ref{item:repelling_strata_containment}; this is a consequence of
transversality of the map $\hat{\phi}$ and the fact that the original
strata $C_z$ satisfied \ref{item:repelling_strata_containment}. Next
we show:
\begin{proposition}
  \label{prop:weak_to_convergence}
  The map $\hat{\phi}:\hat{\Lambda} \to \bigamp$ and sets $\hat{C}_z$
  satisfy condition \ref{item:extended_convergence_prop} in
  \Cref{defn:extended_convergence_group}.
\end{proposition}
\begin{proof}
  Consider a sequence $\gamma_n \in \Gamma$ such that
  $\gamma_n^{\pm 1} \to z_\pm \in \bigamp$, and let
  $K \subset \hat{C}_{z_-}$ be compact. Fix an arbitrary subsequence
  of $\gamma_n$. We will show that there is a further subsequence so
  that $\rho(\gamma_n)K$ lies in an arbitrarily small neighborhood of
  $\hat{\phi}^{-1}(z_+)$. After extraction we can assume that either
  $\lcon{\gamma_n} \to \infty$, or else $\lcon{\gamma_n}$ is bounded.

  In the first case, the sequences $\rho(\gamma_n^{\pm 1})$ are both
  $Q$-divergent by assumption, and by \Cref{lem:q_pts_in_fiber} all of
  the $Q$-limit points of $\rho(\gamma_n^{\pm1})$ respectively lie in
  $\hat{\phi}^{-1}(z_\pm)$. We have constructed the strata $\hat{C}_z$
  so that $\hat{C}_z \subseteq \Opp(\hat{\phi}^{-1}(z))$ for every
  $z \in \bigamp$. In particular, this implies that
  $K \subset \hat{C}_{z_-}$ is opposite to every $Q$-limit point of
  the sequence $\rho(\gamma_n^{-1})$. Thus by
  \Cref{lem:q_convergence_sequences}, $\rho(\gamma_n)K$ eventually
  lies in an arbitrarily small neighborhood of the $Q$-limit points of
  $\rho(\gamma_n)$, hence in an arbitrarily small neighborhood of
  $\hat{\phi}^{-1}(z_+)$.

  Otherwise, if $\lcon{\gamma_n}$ is bounded, then both $z_+$ and
  $z_-$ must be parabolic. Observe that, by definition,
  $\hat{C}_{z_-}$ is contained in a union of sets of the form
  $\rho(g)C_q$, where $q \in \bigamp$ is a parabolic point satisfying
  $gq = z_-$. The compact set $K$ is contained in a finite union of
  these sets, and in fact $K$ can be written as a finite union
  $K_1, \ldots, K_n$ of compact pieces, each of which lies in such a
  set $\rho(g_i)C_{q_i}$. Each $g_i$ is fixed (independent of $n$), so
  $g_i^{-1}\gamma_n^{-1} \to g_i^{-1}z_-$, and $\gamma_ng_i \to
  z_+$. Then, as the original repelling strata $C_{q_i}$ already
  satisfied condition \ref{item:extended_convergence_prop}, the set
  \[
    \rho(\gamma_n)K_i = \rho(\gamma_ng_i) \rho(g_i^{-1})K_i
  \]
  eventually lies in an arbitrarily small neighborhood of
  $\phi^{-1}(z_+) = \hat{\phi}^{-1}(z_+)$, and therefore the same is
  true for $K$.
\end{proof}

The proposition below completes the proof of
\Cref{prop:weak_egf_to_egf_app}.

\begin{proposition}
  The map $\hat{\phi}:\hat{\Lambda} \to \bigamp$ and sets $\hat{C}_z$
  satisfy condition \ref{item:repelling_strata_interior} in
  \Cref{defn:extended_convergence_group}.
\end{proposition}
\begin{proof}
  Fix $K \subset \bigamp$ compact. It suffices to show that for every
  $w \in \bigamp \minus K$, the set $\bigcap_{z \in K} \hat{C}_z$
  contains a neighborhood of $\hat{\phi}^{-1}(w)$. Suppose for a
  contradiction that this is not the case. We have already seen that
  the strata $\hat{C}_z$ satisfy condition
  \ref{item:repelling_strata_containment}, so the only possibility is
  that there is a sequence of points $z_n \in K$ so that for every
  neighborhood $U$ of $\hat{\phi}^{-1}(w)$, the set
  $\hat{C}_n = \hat{C}_{z_n}$ eventually fails to contain $U$.

  Up to a subsequence, every $z_n$ is either conical or
  parabolic. Suppose first that every $z_n$ is conical, in which case
  $\hat{C}_n = \Opp(\hat{\phi}^{-1}(z_n))$ for every $n$. Thus, there
  is a sequence of flags $\xi_n \in \hat{\phi}^{-1}(z_n)$, and a
  sequence of flags $\eta_n \in \flags$, so that $\eta_n$ accumulates
  in $\hat{\phi}^{-1}(w)$, and $\xi_n, \eta_n$ are not
  transverse. After extraction, $\xi_n$ converges to
  $\xi \in \hat{\phi}^{-1}(K)$ and $\eta_n$ converges to
  $\eta \in \hat{\phi}^{-1}(w)$. The flags $\xi$ and $\eta$ are not
  transverse, which contradicts transversality of $\hat{\phi}$.

  So, now consider the case where each $z_n$ is parabolic. Let
  $X = X(\Gamma, \pargps)$ be the cusped space, and fix a bi-infinite
  geodesic $c_n:(-\infty, \infty) \to X$ with its forward endpoint at
  $z_n$ and backward endpoint at $w$. This geodesic must eventually
  enter a combinatorial horoball in $X$ centered at $z_n$, so let
  $\gamma_n \in \Gamma$ be the earliest point where this occurs. Then,
  $\gamma_n^{-1}z_n$ is a parabolic point in $\bigamp$, in the
  boundary of a combinatorial horoball in $X$ passing through the
  identity, and so up to subsequence $\gamma_n^{-1}z_n = p$ for a
  fixed parabolic point $p \in \Pi$.

  We may assume that $z_n$ converges to some $z \in K$. Now, if the
  sequence $\gamma_n$ is bounded in $\cay(\Gamma)$, then eventually
  $z_n = z$ for all $n$, and therefore an open neighborhood of
  $\hat{\phi}^{-1}(w)$ is contained in every $\hat{C}_n = \hat{C}_z$
  by condition \ref{item:repelling_strata_containment}. So, assume
  that $\lgam{\gamma_n} \to \infty$. We claim that in this case,
  $\gamma_n$ converges to $z$. To see this, observe that since $z_n$
  does not accumulate at $w$, the bi-infinite geodesic $c_n$ passes
  within a uniform distance of the identity in $X$. This means that
  the point $\gamma_n$ must actually lie uniformly close (in $X$) to a
  point $r_n(t_n)$, where $r_n$ is a geodesic ray in $X$ from the
  identity to $z_n$, and $t_n \to \infty$. Since $z_n \to z$ it
  follows that $r_n(t_n)$, hence $\gamma_n$, tends to $z$.

  Now, since our chosen strata satisfy condition
  \ref{item:repelling_strata_containment}, and $w \ne z$, there is a
  neighborhood $U$ of $\hat{\phi}^{-1}(w)$ whose closure lies in
  $\hat{C}_z$. Extract a subsequence so that $\gamma_n^{-1} \to y$ for
  some $y \in \bigamp$. From \Cref{prop:weak_to_convergence}, the
  repelling strata satisfy condition
  \ref{item:extended_convergence_prop}, so $\rho(\gamma_n^{-1})U$
  eventually lies in an arbitrarily small neighborhood of
  $\hat{\phi}^{-1}(y)$. However, we also know that $y \ne p$: there is
  a bi-infinite geodesic $\gamma_n^{-1}c_n$ in $X$ from $p$ to
  $\gamma_n^{-1}w$ passing through the identity, and since $w \ne z$,
  we have $\gamma_n^{-1}w \to y$ and therefore there is also a
  geodesic in $X$ joining $p$ to $y$.

  By \ref{item:repelling_strata_containment} again, the interior of
  $\hat{C}_p$ contains $\hat{\phi}^{-1}(y)$, and therefore eventually
  $\hat{C}_p$ contains $\rho(\gamma_n^{-1})U$. Then by equivariance,
  $\hat{C}_n = \rho(\gamma_n)\hat{C}_p$ contains $U$ for all
  sufficiently large $n$. This contradicts our original assumption.
\end{proof}

\printbibliography

\end{document}